\documentclass[10pt]{amsart}

\pdfoutput=1

\usepackage[utf8]{inputenc}
\usepackage[T1]{fontenc}
\usepackage[backend=biber,
            isbn=false,
            doi=false,
            maxbibnames=9,
            giveninits=true,
            style=alphabetic,
            citestyle=alphabetic]{biblatex}
\usepackage{url}
\renewbibmacro{in:}{}
\bibliography{paper.bib}

\DeclareFontFamily{T1}{cbgreek}{}
\DeclareFontShape{T1}{cbgreek}{m}{n}{<-6>  grmn0500 <6-7> grmn0600 <7-8> grmn0700 <8-9> grmn0800 <9-10> grmn0900 <10-12> grmn1000 <12-17> grmn1200 <17-> grmn1728}{}
\DeclareSymbolFont{quadratics}{T1}{cbgreek}{m}{n}
\DeclareMathSymbol{\qoppa}{\mathord}{quadratics}{19}
\DeclareMathSymbol{\Qoppa}{\mathord}{quadratics}{21}

\usepackage{xcolor}
\definecolor{e-mail}{rgb}{0,.40,.80}
\definecolor{reference}{rgb}{.20,.60,.22}
\definecolor{citation}{rgb}{0,.40,.80}
\usepackage{todonotes,setspace}

\usepackage[colorlinks=true,
            linkcolor=reference,
            citecolor=citation,
            urlcolor=e-mail]{hyperref}
\usepackage[all]{xy}
\usepackage{amssymb}
\usepackage{eucal}
\usepackage{cleveref}
\usepackage{aliascnt}
\usepackage{fullpage}
\usepackage{verbatim}

\usepackage[RPvoltages]{circuitikz}
\usetikzlibrary{decorations.pathmorphing}

\newtheorem{thm}{Theorem}[section]

\newtheorem*{theorem}{Theorem}
\newaliascnt{prop}{thm}
\newtheorem{prop}[prop]{Proposition}
\aliascntresetthe{prop}
\crefname{prop}{proposition}{propositions}

\newtheorem{conjecture}[thm]{Conjecture}
\newtheorem{lm}[thm]{Lemma}
\newtheorem{assumption}[thm]{Assumption}
\theoremstyle{definition}
\newtheorem{defn}[thm]{Definition}
\theoremstyle{remark}
\newtheorem{remark}[thm]{Remark}
\newtheorem{example}[thm]{Example}

\newcommand{\A}{\mathrm{A}}
\newcommand{\B}{\mathrm{B}}
\newcommand{\C}{\mathrm{C}}
\newcommand{\D}{\mathrm{D}}
\newcommand{\rH}{\mathrm{H}}
\newcommand{\K}{\mathrm{K}}
\newcommand{\rL}{\mathrm{L}}
\newcommand{\LA}{\mathrm{LA}}
\newcommand{\T}{\mathrm{T}}
\newcommand{\U}{\mathrm{U}}

\newcommand{\bC}{\mathbf{C}}
\newcommand{\bE}{\mathbb{E}}
\newcommand{\Ga}{\mathbf{G}_{\mathrm{a}}}
\newcommand{\Gm}{\mathbf{G}_{\mathrm{m}}}
\newcommand{\bL}{\mathbb{L}}
\newcommand{\bN}{\mathbf{N}}
\newcommand{\bP}{\mathbf{P}}
\newcommand{\Q}{\mathbf{Q}}
\newcommand{\R}{\mathbf{R}}
\newcommand{\bS}{\mathbb{S}}
\newcommand{\bT}{\mathbb{T}}
\newcommand{\Z}{\mathbf{Z}}

\newcommand{\cC}{\mathcal{C}}
\newcommand{\cD}{\mathcal{D}}
\newcommand{\cF}{\mathcal{F}}
\newcommand{\cG}{\mathcal{G}}
\newcommand{\cL}{\mathcal{L}}
\newcommand{\cN}{\mathcal{N}}
\newcommand{\cO}{\mathcal{O}}

\newcommand{\cS}{\mathcal{S}}

\newcommand{\g}{\mathfrak{g}}
\newcommand{\fl}{\mathfrak{l}}
\newcommand{\flie}{\mathfrak{lie}}
\newcommand{\fp}{\mathfrak{p}}

\newcommand{\ad}{\operatorname{ad}}
\newcommand{\Arf}{\mathrm{Arf}}
\newcommand{\at}{\mathrm{at}}
\newcommand{\Aut}{\mathrm{Aut}}
\newcommand{\Br}{\mathrm{Br}}
\newcommand{\ch}{\mathrm{ch}}

\newcommand{\coev}{\mathrm{coev}}

\newcommand{\colim}{\operatornamewithlimits{colim}}
\newcommand{\coMod}{\mathrm{coMod}}
\newcommand{\detgr}{\mathrm{det}^{\mathrm{gr}}}
\newcommand{\Dol}{\mathrm{Dol}}
\newcommand{\dR}{\mathrm{dR}}
\newcommand{\Eul}{\mathrm{Eul}}

\newcommand{\ev}{\mathrm{ev}}

\newcommand{\Fun}{\mathrm{Fun}}
\newcommand{\GL}{\mathrm{GL}}
\newcommand{\gr}{\operatorname{gr}}

\newcommand{\GW}{\mathrm{GW}}
\newcommand{\uGW}{\underline{\mathrm{GW}}}
\newcommand{\hGa}{\widehat{\mathbf{G}}_{\mathrm{a}}}
\newcommand{\Hom}{\mathrm{Hom}}
\newcommand{\id}{\mathrm{id}}
\newcommand{\Ind}{\operatorname{Ind}}
\newcommand{\IndCoh}{\mathrm{IndCoh}}
\newcommand{\uK}{\underline{\mathrm{K}}}
\newcommand{\Loc}{\mathrm{Loc}}
\newcommand{\loccit}{\emph{loc. cit.}}
\newcommand{\LocSys}{\mathrm{LocSys}}
\newcommand{\Map}{\mathrm{Map}}
\newcommand{\Mod}{\mathrm{Mod}}
\newcommand{\motimes}{\otimes^*}

\newcommand{\op}{\mathrm{op}}

\newcommand{\Perf}{\mathrm{Perf}}
\newcommand{\uPerf}{\underline{\mathrm{Perf}}}
\newcommand{\PGL}{\mathrm{PGL}}
\newcommand{\Pic}{\mathrm{Pic}}
\newcommand{\uPic}{\underline{\mathrm{Pic}}}
\newcommand{\Pn}{\mathrm{Pn}}
\newcommand{\PrSt}{\mathrm{Pr}^{\mathrm{St}}}

\newcommand{\pt}{\mathrm{pt}}
\newcommand{\QCoh}{\mathrm{QCoh}}
\newcommand{\red}{\mathrm{red}}
\newcommand{\Ring}{\mathrm{Ring}}
\newcommand{\rk}{\operatorname{rk}}
\newcommand{\sgn}{\mathrm{sgn}}
\newcommand{\sh}{\mathrm{sh}}
\newcommand{\SL}{\mathrm{SL}}
\newcommand{\SO}{\mathrm{SO}}
\newcommand{\Sp}{\mathrm{Sp}}

\newcommand{\Spec}{\operatorname{Spec}}
\newcommand{\Spinc}{\mathrm{Spin}^c}
\newcommand{\Sym}{\operatorname{Sym}}
\newcommand{\Td}{\mathrm{Td}}
\newcommand{\tens}{\mathrm{tens}}
\newcommand{\Vect}{\mathrm{Vect}}
\newcommand{\uVect}{\underline{\mathrm{Vect}}}
\newcommand{\vol}{\mathrm{vol}}
\newcommand{\Wh}{\mathrm{Wh}}

\newcommand{\defterm}[1]{\textbf{\emph{#1}}}

\newcommand{\cosimp}[3]{
\xymatrix{
#1 \ar@<.5ex>[r] \ar@<-.5ex>[r] & #2 \ar@<.8ex>[r] \ar[r] \ar@<-.8ex>[r] & #3 \ar@<1.2ex>[r] \ar@<.4ex>[r] \ar@<-.4ex>[r] \ar@<-1.2ex>[r] & \ldots
}
}

\newcommand{\simp}[3]{
\xymatrix{
#1 & #2 \ar@<.5ex>[l] \ar@<-.5ex>[l] & #3 \ar@<.8ex>[l] \ar[l] \ar@<-.8ex>[l] & \ldots \ar@<1.2ex>[l] \ar@<.4ex>[l] \ar@<-.4ex>[l] \ar@<-1.2ex>[l]
}
}

\begin{document}
\title{Torsion volume forms}
\address{Trinity College Dublin, Dublin, Ireland}
\email{naeff@tcd.ie}
\author{Florian Naef}
\address{School of Mathematics, University of Edinburgh, Edinburgh, UK}
\email{p.safronov@ed.ac.uk}
\author{Pavel Safronov}
\begin{abstract}
We introduce volume forms on mapping stacks in derived algebraic geometry using a parametrized version of the Reidemeister--Turaev torsion. In the case of derived loop stacks we describe this volume form in terms of the Todd class. In the case of mapping stacks from surfaces, we compare it to the symplectic volume form. As an application of these ideas, we construct canonical orientation data for cohomological DT invariants of closed oriented 3-manifolds.
\end{abstract}
\maketitle

\section*{Introduction}

\subsection*{Simple homotopy types}

The key input to the construction of volume forms on mapping stacks in this paper is a local factorization of the determinant of cohomology on a stack using what we call a \emph{simple structure}. Our motivation comes from the theory of simple homotopy types which we briefly recall.

Consider a homotopy type $M$. To define its Euler characteristic $\chi(M)\in\Z$ one has to assume a finiteness condition on $M$, i.e. that $M$ is finitely dominated, which can be phrased in a homotopy-invariant way by saying that $M$ is a compact object of the $\infty$-category of spaces $\cS$. This ensures, for instance, that given any local system $\cL$ on $M$ whose fiber at any point is finite dimensional (a perfect complex), the homology $\rH_\bullet(M; \cL)$ is bounded and finite-dimensional in each degree. In particular, one may consider its Euler characteristic.

If $M$ is a finite CW complex, there is a \emph{local} formula for the Euler characteristic obtained by computing the homology $\rH_\bullet(M; \cL)$ using the cellular chain complex $\C_\bullet(M; \cL)$:
\[\chi(\rH_\bullet(M; \cL)) = \sum_\sigma (-1)^{\dim(\sigma)} \dim(\cL_{\alpha_\sigma}),\]
where the sum is over cells $\sigma$ of $M$ and $\alpha_\sigma\in M$ is a point in the interior of $\sigma$. One may ask what extra structure on the homotopy type $M$ allows for such a local description (i.e. which only involves information about the individual fibers of $\cL$, but not the parallel transport maps) of the Euler characteristic. To have a universal answer instead of computing the Euler characteristic of $\C_\bullet(M; \cL)$, we will describe the point in the $K$-theory space. Moreover, we will work over the sphere spectrum.

The corresponding $K$-theory is $A$-theory $\A(M)$ \cite{Waldhausen} and, if $M$ is finitely dominated, there is a homotopy-invariant Euler characteristic which defines a point $[\bS_M]\in\A(M)$. There is, moreover, a canonical \defterm{assembly} map
\[\alpha\colon \C_\bullet(M; \A(\pt))=\Sigma^\infty_+ M\otimes \A(\pt)\longrightarrow \A(M).\]
The structure of a finite CW complex on $M$ allows one to lift $[\bS_M]\in\A(M)$ to $e_\A(M)\in\Omega^\infty\C_\bullet(M; \A(\pt))$ supported at the points $\alpha_\sigma$ and this allows one to obtain a homotopy
\[[\C_\bullet(M; \cL)]\sim \sum_\sigma (-1)^{\dim(\sigma)} [\cL_{\alpha_\sigma}]\]
in $\Omega^\infty \A(M)$ for any (dualizable) parametrized spectrum $\cL$ over $M$. Let us recall some known results:
\begin{itemize}
    \item For a given finitely dominated space $M\in\cS$ the obstruction to the existence of a lift of $[\bS_M]$ along the assembly map is Wall's finiteness obstruction \cite{WallFiniteness}. It vanishes if, and only if, $M$ is homotopy equivalent to a finite CW complex.
    \item Given two finite CW complexes $M_1, M_2$ together with a homotopy equivalence $f\colon M_1\rightarrow M_2$ the difference between the corresponding lifts is captured by the Whitehead torsion $\Wh(f)$ of the homotopy equivalence $f$.
\end{itemize}

We will think of the space of lifts of $[\bS_M]$ along the assembly map as the space of ways of endowing $M$ with the structure of a simple homotopy type. The description of the homotopy type of this space is the content of the stable parametrized $h$-cobordism theorem \cite{JahrenRognesWaldhausen}.

In this paper we transport this notion of a simple homotopy type to the world of derived stacks over a ground ring $k$. Given such a derived stack $X$ satisfying a finiteness assumption (analogous to finite domination in the topological setting, see \cref{mainassumptionassembly}) there is an assembly map
\[\C_\bullet(X(k); \K(k))\longrightarrow \K^\omega(X),\]
where $\K^\omega(X)$ is the $\K$-theory of the stable $\infty$-category $\QCoh(X)^\omega$ of compact quasi-coherent complexes over $X$ together with a canonical point $[\cO_X]\in\Omega^\infty\K^\omega(X)$. A \defterm{simple structure} on $X$ (see \cref{def:simplestructure}) is then a lift of $[\cO_X]$ along the assembly map. For the Betti stack $X=M_\B$ of a homotopy type, we get exactly the notion of a simple homotopy type from before and we describe this lift in concrete terms in \cref{sect:Bettilifts}. We also use the theory of de Rham $\epsilon$-factors \cite{Patel,Groechenig} to define simple structures on de Rham and Dolbeault stacks $M_{\dR}$ and $M_{\Dol}$ in \cref{sect:DolbeaultdeRham}. We also expect that simple structures can be defined in other sheaf contexts with a 6-functor functoriality, e.g. in the arithmetic setting.

With the definition of simple structure the following theorem is then straightforward which provides a local description (an $\epsilon$-factorization in the sense of \cite{Beilinson}) of cohomology of stacks.

\begin{theorem}[See \cref{thm:indextheorem} for a complete statement]
Let $X$ be a derived stack equipped with a simple structure and $\cF\in\Perf(X)$. Then there is a homotopy
\[[p_\sharp \cF]\sim\sum_i [\cF_{x_i}]\alpha_i\]
in $\Omega^\infty\K(k)$ for some points $x_i\in X(k)$ and $\alpha_i\in\Omega^\infty\K(k)$. Here $p\colon X\rightarrow \pt$ and $p_\sharp$ is the functor of homology of $X$.
\end{theorem}

In many settings Poincar\'e duality allows one to state the above index theorem for cohomology (i.e. for $[p_* \cF]$) in a similar way. The approach to (parametrized) topological index theorems via lifts along the assembly map was described in \cite{DwyerWeissWilliams} and our construction and setup are directly inspired by that work.

\subsection*{Volume forms on mapping stacks}

For a derived stack $Y$ with a perfect cotangent complex $\bL_Y$ the analog of the sheaf of volume forms is the determinant line $\det(\bL_Y)$. So, we define a volume form on $Y$ as a trivialization of $\det(\bL_Y)$ (see \cref{def:volumeform}). Smooth schemes with a trivial canonical bundle provide examples. There are also more interesting examples defined as follows. Recall that given a smooth symplectic scheme $(X, \omega)$, the symplectic volume form is $\frac{\omega^{\dim X}}{(\dim X)!}$. We generalize this construction to the derived setting (where $\omega^{\dim X}$ no longer defines a section of the determinant line) as follows. Recall the notion of an $n$-shifted symplectic structure on a derived stack from \cite{PTVV}. Using the formalism of Grothendieck--Witt spectra of stable $\infty$-categories with duality (or Poincare $\infty$-categories) from \cite{Schlichting,CDHHLMNNS1} in \cref{sect:symplecticvolume} we define \defterm{symplectic volume forms} on $n$-shifted symplectic stacks for any $n\in\Z$ divisible by $4$.

One of our main theorems establishes the existence of volume forms on derived mapping stacks $\Map(X, Y)$.

\begin{theorem}[See \cref{thm:volumelift} for the precise statement]
Let $X, Y$ be derived stacks, where $X$ is equipped with a simple structure. Either suppose $\dim(Y) = 0$ or choose an isomorphism $\det(p_\sharp \cO_X)\cong k$. Moreover, choose either a volume form on $Y$ or a trivialization of the Euler class $e(X)\in\C_\bullet(X(k); \Z)$ (see \cref{def:simplestructure}). Then $\Map(X, Y)$ carries a canonical \defterm{torsion volume form}.
\end{theorem}

The construction of the torsion volume form (and the name) is directly inspired by the theory of Reidemeister torsion with a refinement by Turaev \cite{TuraevReidemeister,TuraevEuler}. We refer to \cite{TuraevBook,Nicolaescu} for a pedagogical introduction. In fact, it directly reduces to the adjoint Reidemeister--Turaev torsion in the following important example (see \cref{sect:Reidemeister} for more details). Let $M$ be a finite CW complex and $G$ an algebraic group and consider the derived character stack $\Loc_G(M) = \Map(M_\B, \B G)$ whose classical stack parametrizes representations of the fundamental group $\pi_1(M)$. Then the fiber of $\det(\bL_{\Loc_G(M)})$ at a given $G$-local system $\cL$ is given by the determinant of the cohomology of the adjoint local system $\ad \cL$ and the torsion volume form is given by its Reidemeister torsion. This extends the well-known construction (see \cite{Witten2d,HeusenerPorti}) of a volume form on (an open subset of good representations of) the character variety of a surface for a semisimple group $G$ to the full derived moduli stack.

When $M=\Sigma$ is a closed oriented surface, we prove the following related results connecting the torsion volume form to the symplectic volume form:
\begin{itemize}
    \item If $G$ is a unimodular algebraic group with a nondegenerate pairing on its Lie algebra, we show in \cref{thm:symplectictorsionvolumeform} that the torsion volume form on the derived character stack $\Loc_G(\Sigma)$ differs from the symplectic volume form by a sign determined by the second Stiefel--Whitney class of the adjoint representation (which is trivial when $G$ is simply-connected).
    \item If $\cL$ is an orthogonal rank 1 local system over $\Sigma$, the torsion element $\tau_s(\cL)$ of $\det\rH_\bullet(\Sigma; \cL)$ depends on a spin structure $s$ on $\Sigma$, while the symplectic volume form $\vol_\cL\in\det\rH_\bullet(\Sigma; \cL)$ is defined canonically. We show in \cref{prop:Johnsonsign} that the map $\rH^1(\Sigma; \mu_2)\rightarrow \mu_2$ given by $\cL\mapsto \tau_s(\cL)/\vol_\cL$ is given by Johnson's quadratic refinement \cite{Johnson} of the intersection pairing on $\Sigma$, where $\mu_2$ is the algebraic group of second roots of unity. This gives an interesting new perspective on this function.
\end{itemize}

When $M$ is a closed oriented 3-manifold, we give an application of our results to the theory of cohomological Donaldson--Thomas (DT) invariants. Let us briefly recall the setting. For any complex $(-1)$-shifted symplectic stack $X$ equipped with an \emph{orientation data} (the choice of a square root of $\det(\bL_X)$) the authors of \cite{BBBBJ} have defined a perverse sheaf $\phi_X$ on the underlying classical stack $t_0(X)$ whose local Euler characteristic gives the Behrend function. The cohomology $\rH^\bullet(t_0(X), \phi_X)$ is the cohomological DT invariant of $X$. We refer to \cite{JoyceUpmeier1,JoyceUpmeier2} for a construction of orientation data for many moduli spaces using techniques from differential geometry.

\begin{theorem}[See \cref{prop:LocGoriented} and \cref{thm:Locparabolicinductionorientation}]
Let $M$ be a closed oriented 3-manifold and $G$ a split connected reductive group. Then the $(-1)$-shifted symplectic stack $\Loc_G(M)=\Map(M_\B, \B G)$ of $G$-local systems on $M$ has a canonical orientation data.

Suppose $P\subset G$ is a parabolic subgroup and $L$ the Levi factor. Assume that either the modular character of $P$ admits a square root or that $M$ is equipped with a spin structure. Then the $(-1)$-shifted Lagrangian correspondence $\Loc_L(M)\leftarrow \Loc_P(M)\rightarrow \Loc_G(M)$ has a canonical orientation data.
\end{theorem}

The orientation data on the Lagrangian correspondence $\Loc_L(M)\leftarrow \Loc_P(M)\rightarrow \Loc_G(M)$ gives, assuming a certain functoriality of the perverse sheaf $\phi_X$ conjectured by Joyce, a parabolic induction map \[\rH^\bullet(t_0(\Loc_L(M)), \phi_{\Loc_L(M)})\rightarrow \rH^\bullet(t_0(\Loc_G(M)), \phi_{\Loc_G(M)})\]
as we explain in \cref{thm:DTinduction}.

\subsection*{Derived loop stacks}

Given a morphism $f\colon X\rightarrow Y$ of smooth and proper schemes (over a field $k$ of characteristic zero) the Grothendieck--Riemann--Roch theorem asserts that the commutativity of the diagram
\[
\xymatrix{
\K_0(X) \ar^-{\ch}[r] \ar^{f_*(-)}[d] & \bigoplus_n \rH^n(X, \Omega^n_X) \ar^{\int_f \Td_{X/Y}\cup(-)}[d] \\
\K_0(Y) \ar^-{\ch}[r] & \bigoplus_n \rH^n(Y, \Omega^n_Y)
}
\]
where $\ch$ is the Chern character and $\Td_{X/Y}$ is the relative Todd class. The correction by the Todd class has the following interpretation in derived algebraic geometry \cite{Markarian,BZNNonlinear,KondyrevPrikhodko}. There is a commutative diagram
\[
\xymatrix{
\K_0(X) \ar^-{\ch}[r] \ar^{f_*(-)}[d] & \rH^0(LX, \cO_{LX}) \ar^{\int_f(-)}[d] \\
\K_0(Y) \ar^-{\ch}[r] & \rH^0(\cL Y, \cO_{LY})
}
\]
where $LX = \Map(S^1_\B, X)\cong X\times_{X\times X} X$ is the derived loop space of $X$ and $\int_f\colon \rH^0(LX, \cO_{LX})\rightarrow \rH^0(LY, \cO_{LY})$ is a certain natural integration map constructed using the formalism of traces. As this is an integration map of functions, it is determined by a relative volume form along the fibers. As shown in \cite{KondyrevPrikhodko}, this volume form comes from the natural structure of $\cL X\rightarrow X$ as a derived group scheme (with the group structure given by loop composition). The Todd class then appears when we use the Hochschild--Kostant--Rosenberg (HKR) isomorphism to identify $LX\cong \T[-1] X$ with the shifted tangent bundle, where the natural volume form on $\T[-1] X$ comes from its structure as a derived vector bundle over $X$.

We provide an interpretation of the two volume forms using torsion volume forms. We identify the derived loop space as $LX = \Map(S^1_\B, X)$ and the shifted tangent bundle as $\T[-1] X = \Map(\B\hGa, X)$, where $\hGa$ is the formal additive group. The circle $S^1$ is a finite CW complex with $\chi(S^1) = 0$, so there is a natural torsion volume form $\vol_{LY}$ on $LX$. There is also a (unique) simple structure on $\B\hGa$ which induces a torsion volume form $\vol_{\T[-1] X}$ on $\T[-1] X$.

\begin{theorem}[See \cref{thm:Toddclass}]
The ratio $\vol_{LX} / \vol_{\T[-1] X}$ is given by an invertible function on $\T[-1] X$ which is the Todd class $\Td(X)$.
\end{theorem}

In fact, we prove the above statement for any derived stack $X$ (with a perfect cotangent complex). In this case there is no isomorphism between $LX$ and $\T[-1] X$, but, nevertheless, there is a correspondence $\T[-1] X\leftarrow \Map(\B\Ga, X)\rightarrow LX$ with both maps formally \'etale, which allows us to compare the volume forms.

\subsection*{Cotangent AKSZ theories}

Given a $(d-1)$-shifted symplectic stack $Z$, there is an associated $d$-dimensional topological classical field theory defined using the ASKZ formalism \cite{AKSZ,PTVV,CalaqueHaugsengScheimbauer}. Let us discuss the case when $Z=\T^*[d-1] Y$ is a shifted cotangent bundle. For instance, we have the following topological field theories of this kind:
\begin{itemize}
    \item The BF theory for $Y=\B G$, the classifying stack of a group $G$, and any $d$.
    \item The 2d $B$-model with target a smooth complex variety $Y$.
    \item The 3d Rozansky--Witten theory with target $\T^* Y$ for $Y$ a smooth complex variety.
\end{itemize}

As explained in \cite{Witten2d,CostelloWitten}, for cotangent AKSZ theories the quantization (for a closed oriented $d$-dimensional spacetime manifold $M$) is one-loop-exact, the one-loop determinant defines a volume form on $\Map(M_\B, Y)$ and the partition function is given by the volume of $\Map(M_\B, Y)$ with respect to this volume form. We expect that the torsion volume form on $\Map(M_\B, Y)$ provides a version of this volume form in the world of derived algebraic geometry. For instance, the conditions for the existence of the torsion volume form on $\Map(M_\B, \B G)$ match the anomaly cancellation condition in BF theory: either $G$ is unimodular or $\chi(M) = 0$.

\subsection*{Acknowledgements}

We would like to thank Benjamin Hennion and Markus Land for useful discussions.

\subsection*{Notation}

Throughout the paper we use the formalism of $\infty$-categories.

\begin{itemize}
    \item We denote the smash product of spectra by $\otimes$.
    \item For a topological space $X$ and a spectrum $Y$ we denote by $\C_\bullet(X; Y)$ the spectrum $\Sigma^\infty_+ X\otimes Y$. We denote by $\rH_n(X; Y)$ its $n$-th homotopy group.
    \item For a topological space $X$ and a spectrum $Y$ we denote by $\C^\bullet(X; Y)$ the spectrum $\Map(X, Y)$. We denote by $\rH^n(X; Y)$ its $(-n)$-th homotopy group.
    \item For a spectrum $Y$ we denote by $\tau_{\geq 0} Y$ the connective cover of $Y$ and by $\Omega^\infty Y$ the underlying space.
    \item Throughout the paper we work over a commutative ring $k$.
\end{itemize}

\section{Preliminaries}

In this section we introduce basic constructions in the paper: a finiteness condition on derived stacks we will use in the paper and some operations on their $K$-theory.

\subsection{Finiteness}

For a commutative dg algebra $A$ over $k$ we denote by $\Mod_A$ the stable $\infty$-category of $A$-modules.

Let $X$ be a derived prestack over $k$. Recall the following $\infty$-categories associated to it:
\begin{itemize}
    \item The symmetric monoidal $\infty$-category of quasi-coherent sheaves
    \[\QCoh(X) = \lim_{S\rightarrow X} \Mod_{\cO(S)},\]
    where the limit is taken over all derived affine schemes with a map to $X$. For a morphism $f\colon X\rightarrow Y$ of derived prestacks there is a symmetric monoidal pullback functor $f^*\colon \QCoh(Y)\rightarrow \QCoh(X)$.
    \item Let $\QCoh^\omega(X)\subset \QCoh(X)$ be the full subcategory of compact objects.
    \item The full subcategory $\Perf(X)\subset \QCoh(X)$ of perfect complexes, i.e. dualizable objects.
\end{itemize}

Dualizable objects are closed under the tensor product, so $\Perf(X)$ is a symmetric monoidal $\infty$-category. Moreover, if $\cF$ is perfect and $\cG$ is compact, the functor
\[\Hom_{\QCoh(X)}(\cF\otimes \cG, -)\cong \Hom_{\QCoh(X)}(\cG, \cF^\vee\otimes (-))\]
preserves colimits, so $\QCoh(X)^\omega$ is a $\Perf(X)$-module category.

In the paper we will be interested in the following finiteness conditions on a derived prestack.

\begin{assumption}
Let $X$ be a derived prestack satisfying the following conditions:
\begin{enumerate}
    \item The structure sheaf $\cO_X\in\QCoh(X)$ is compact, i.e. the pullback $p^*\colon \Mod_k\rightarrow \QCoh(X)$ along $p\colon X\rightarrow \pt$ admits a colimit-preserving right adjoint $p_*$.
    \item The pullback $p^*\colon \Mod_k\rightarrow \QCoh(X)$ admits a left adjoint $p_\sharp$.
\end{enumerate}
\label{mainassumption}
\end{assumption}

\begin{remark}
We think of $p_*\colon \QCoh(X)\rightarrow \Mod_k$ as the functor of \emph{cohomology} while $p_\sharp\colon \QCoh(X)\rightarrow \Mod_k$ as the functor of \emph{homology}.
\end{remark}

For a derived prestack $X$ satisfying \cref{mainassumption} the functor $p_\sharp\colon \QCoh(X)\rightarrow \Mod_k$ preserves compact objects and hence $p_\sharp \cO_X$ is a perfect complex.

\begin{defn}
Let $X$ be a derived prestack satisfying \cref{mainassumption}. Its \defterm{Euler characteristic} is
\[\chi(X) = \chi(p_\sharp \cO_X).\]
\end{defn}

Let us now present several corollaries of this assumption. Recall the notion of an \defterm{$\cO$-compact prestack} from \cite[Definition 2.1]{PTVV}.

\begin{prop}
Let $X$ be a derived prestack satisfying \cref{mainassumption}. Then:
\begin{enumerate}
    \item $\Perf(X)\subset \QCoh(X)^\omega$.
    \item Let $S$ be a derived affine scheme and let $\pi\colon S\times X\rightarrow S$ be the projection. Then there are colimit-preserving functors $\pi_\sharp, \pi_*\colon \QCoh(S\times X)\rightarrow \QCoh(S)$, where $\pi_\sharp\dashv \pi^*\dashv \pi_*$, which satisfy the projection formulas: the natural morphisms
    \[\pi_\sharp(\cF\otimes \pi^*\cG)\longrightarrow \pi_\sharp(\cF)\otimes \cG\]
    and
    \[\pi_*(\cF)\otimes \cG\longrightarrow \pi_*(\cF\otimes \pi^*\cG)\]
    are isomorphisms.
    \item $X$ is $\cO$-compact. In particular, for a derived affine scheme $S$ the functors $\pi_\sharp, \pi_*$ preserve perfect complexes, i.e. they restrict to functors
    \[\pi_\sharp,\pi_*\colon \Perf(S\times X)\longrightarrow \Perf(S).\]
\end{enumerate}
\label{prop:assumptioncorollaries}
\end{prop}
\begin{proof}
Let $S$ be a derived affine scheme. Let $p\colon X\rightarrow \pt$ and $\pi\colon S\times X\rightarrow S$ be the projections, so that $\pi = \id\times p$. As $\QCoh(S\times X)\cong \QCoh(S)\otimes \QCoh(X)$, the functor $(\id\times p)^*\colon \QCoh(S)\rightarrow \QCoh(S\times X)$ admits a colimit-preserving right adjoint $\pi_*=\id\otimes p_*$ and a left adjoint $\pi_\sharp=\id\otimes p_\sharp$. The forgetful functor $\QCoh(S)\rightarrow \Mod_k$ is colimit-preserving, so the total pushforward functor $\QCoh(S\times X)\rightarrow \Mod_k$ is colimit-preserving as well. Therefore, $\cO_{S\times X}$ is compact. The tensor product of a perfect complex and a compact object is still compact. So, $\Perf(S\times X)\subset \QCoh(S\times X)^\omega$.

As $\pi^*$ is colimit-preserving, $\pi_\sharp$ preserves compact objects. As compact and perfect objects in $\QCoh(S)$ coincide, we see that $\pi_\sharp$ preserves perfect objects. But then if $\cF\in\Perf(S\times X)$, we have
\[\Hom_{\QCoh(S)}(\pi_\sharp\cF^\vee, \cO_S)\cong \Hom_{\QCoh(S\times X)}(\cF^\vee, \cO_{S\times X})\cong \pi_*\cF\]
and hence $\pi_*\cF$ is a perfect complex on $S$.
\end{proof}

Finally, recall from \cite[Chapter 1, Definition 7.1.2]{GaitsgoryRozenblyum2} the notion of a derived prestack $X$ admitting a representable deformation theory. In this case there is a cotangent complex $\bL_X\in\QCoh(X)$. For two derived prestacks $X,Y$ we may consider the mapping prestack $\Map(X, Y)$ together with the evaluation morphism
\[\ev\colon \Map(X, Y)\times X\longrightarrow Y\]
and the projection
\[\pi\colon \Map(X, Y)\times X\longrightarrow \Map(X, Y)\]
on the first factor.

\begin{prop}
Suppose $X,Y$ are derived prestacks, where $Y$ admits a perfect cotangent complex and $X$ satisfies \cref{mainassumption}. Then $\Map(X, Y)$ admits a perfect cotangent complex given by the formula
\[\bL_{\Map(X, Y)} = \pi_\sharp\ev^* \bL_Y.\]
\label{prop:mappingcotangent}
\end{prop}
\begin{proof}
The formula for $\bL_{\Map(X, Y)}$ is proven in \cite[Proposition B.3.5]{Rozenblyum}. The fact that it is perfect follows from \cref{prop:assumptioncorollaries}: for any derived affine scheme $S$ the pushforward $\pi_\sharp\colon \QCoh(S\times X)\rightarrow \QCoh(S)$ preserves perfect complexes.
\end{proof}

\subsection{K-theory}

Let $X$ be a derived prestack. Consider the following objects:
\begin{itemize}
\item The prestack $\uPerf$ of symmetric monoidal stable $\infty$-categories which assigns $\Perf(S)$ to $S$. We denote by $\uPerf^\sim$ the underlying $\infty$-groupoid.
\item The prestack $\uPerf(X) = \Map(X, \uPerf)$ of stable $\infty$-categories which assigns $\Perf(S\times X)$ to $S$.
\item The stable $\infty$-category $\Perf^\vee(X)$ of exact functors $\Perf(S\times X)\rightarrow \Perf(S)$ natural in $S$ (i.e. compatible with base change); explicitly,
\[\Perf^\vee(X) = \Fun^{ex}(\uPerf(X), \uPerf).\]
\end{itemize}

There is a natural evaluation functor
\[\Perf^\vee(X)\times \uPerf(X)\longrightarrow \uPerf.\]

In this paper we use the formalism of algebraic $K$-theory of stable $\infty$-categories. Given a stable $\infty$-category $\cC$, there is a connective spectrum $\K(\cC)$. An object $x\in\cC$ defines a point $[x]\in\Omega^\infty\K(\cC)$. Moreover, a fundamental property of $K$-theory is its additivity; we will repeatedly use the following manifestation of this property: given a filtered object $x\in\cC$, there is a canonical homotopy between $[x]\in\Omega^\infty\K(\cC)$ and its associated graded $[\gr x]\in\Omega^\infty\K(\cC)$ which we call the \defterm{additivity homotopy}. For instance, given a fiber sequence $x\rightarrow y\rightarrow z$, one has a canonical homotopy from $[y]$ to $[x] + [z]$, where we think of $x\rightarrow y$ as the data of a two-step filtration on $y$ and $x\oplus z$ as its associated graded.

We will consider several versions of $K$-theory of $X$:
\begin{itemize}
\item $\K(X)$ denotes the connective $K$-theory of the stable $\infty$-category $\Perf(X)$. As $\Perf(X)$ is symmetric monoidal, $\K(X)$ has an $E_\infty$ structure.
\item $\K^\omega(X)$ denotes the connective $K$-theory of the stable $\infty$-category $\QCoh(X)^\omega$. As $\QCoh(X)^\omega$ is a $\Perf(X)$-module category, $\K^\omega(X)$ is a $\K(X)$-module.
\item $\uK$ is the prestack which sends a derived affine scheme $S$ to $\K(S)$.
\item $\uK(X)$ is the prestack which sends a derived affine scheme $S$ to $\K(S\times X)$. Note that there is a natural map $\uK(X) \to \Map(X,\uK)$ sending $\K(S \times X) \to \lim_{A \to X} \K(S \times A)$ that is generally not an equivalence.
\item $\K^\vee(X)$ is the connective $K$-theory of the stable $\infty$-category $\Perf^\vee(X)$.
\end{itemize}

There is a natural evaluation map
\[\K^\vee(X)\otimes \uK(X)\longrightarrow \uK.\]

If $X$ satisfies \cref{mainassumption}, we have several new features:
\begin{itemize}
\item The inclusion $\Perf(X)\subset \QCoh(X)^\omega$ induces a map $\K(X)\rightarrow \K^\omega(X)$ of connective spectra.
\item We may consider the class $[\cO_X]\in \Omega^\infty\K^\omega(X)$ of the structure sheaf $\cO_X\in\QCoh(X)^\omega$.
\item There are pushforward functors $\pi_\sharp,\pi_*\in\Perf^\vee(X)$.
\item There is a functor
\[\tens_X\colon \QCoh(X)^\omega\longrightarrow \Perf^\vee(X)\]
given by integral transform as follows. For a derived affine scheme $S$ it is the functor
\[\QCoh(X)^\omega\longrightarrow \Fun^{ex}(\Perf(S\times X), \Perf(S))\]
given by $\cF\mapsto (\cG\mapsto \pi_\sharp(\cF\otimes\cG))$. Under this functor $\cO_X\in\QCoh(X)^\omega$ is sent to $\pi_\sharp\in\Perf^\vee(X)$.
\end{itemize}

Let us now describe a situation when $\pi_\sharp\colon \uK(X)\rightarrow \uK$ is nullhomotopic.

\begin{defn}
Let $X$ be a derived prestack satisfying \cref{mainassumption}. An \defterm{Euler structure} on $X$ is a nullhomotopy of $[\cO_X]\in\Omega^\infty\K^\omega(X)$.
\end{defn}

\begin{thm}
Suppose $X$ is a derived prestack equipped with an Euler structure. Then $\pi_\sharp\colon \uK(X)\rightarrow \uK$ admits a nullhomotopy.
\label{thm:pushforwardEulertrivial}
\end{thm}
\begin{proof}
The functor $\tens_X\colon \QCoh(X)^\omega\rightarrow \Perf^\vee(X)$ descends to a morphism
\[\tens_X\colon \K^\omega(X)\longrightarrow \K^\vee(X).\]
Under this morphism $[\cO_X]\in\Omega^\infty\K^\omega(X)$ is sent to $[\pi_\sharp]\in\Omega^\infty\K^\vee(X)$. Thus, the nullhomotopy of $[\cO_X]\in\K^\omega(X)$ induces a nullhomotopy of $\pi_\sharp\colon \uK(X)\rightarrow \uK$.
\end{proof}

\subsection{Assembly and coassembly}

In some of our examples derived prestacks will not have an Euler structure, but instead a slightly weaker structure. To describe the precise situation let us introduce the assembly and coassembly maps.

Let $S$ be a derived affine scheme. For a point $i\colon \pt\rightarrow X$, i.e. an element $x\in X(k)$, the functor $(\id\times i)^*\colon \Perf(S\times X)\rightarrow \Perf(S)$ induces a map
\[\K(S\times X)\longrightarrow \K(S).\]
This map is natural in $x\in X(k)$ and $S$, so we obtain the \defterm{coassembly map}
\[\epsilon\colon \uK(X)\longrightarrow \C^\bullet(X(k); \uK).\]

We can also define a map ``dual'' to the coassembly map. For this we need a stronger assumption on $X$.

\begin{assumption}
Let $X$ be a derived prestack satisfying \cref{mainassumption} and the following condition: for every point $i\colon \pt\rightarrow X$ the pullback functor $i^*\colon \QCoh(X)\rightarrow \Mod_k$ admits a left adjoint $i_\sharp\colon \Mod_k\rightarrow \QCoh(X)$ satisfying the projection formula, i.e. such that the natural morphism
X)\[i_\sharp(i^* \cF\otimes V)\longrightarrow \cF\otimes i_\sharp V\]
is an isomorphism for every $\cF\in\QCoh(X)$ and $V\in\Mod_k$.
\label{mainassumptionassembly}
\end{assumption}

Let $X$ be a derived prestack satisfying \cref{mainassumptionassembly} and $S$ a derived affine scheme. For any point $i\colon \pt\rightarrow X$ the functor $(\id\times i)_\sharp\colon \QCoh(S)\rightarrow \QCoh(S\times X)$ preserves compact objects as it has a colimit-preserving right adjoint. Therefore, it induces a map
\[i_\sharp\colon \K(S)\longrightarrow \K^\omega(S\times X).\]
It is natural in $x\in X(k)$ and $S$, so we obtain the \defterm{assembly map}
\[\alpha\colon \C_\bullet(X(k); \K(k))\longrightarrow \K^\omega(X).\]

\begin{defn}
Let $X$ be a derived prestack satisfying \cref{mainassumptionassembly}. A \defterm{simple structure} on $X$ is the data of the $K$-theoretic Euler class $e_\K(X)\in\Omega^\infty\C_\bullet(X(k); \K(k))$ together with a homotopy $\alpha(e_\K(X))\sim [\cO_X]$ in $\Omega^\infty\K^\omega(X)$. In this case the \defterm{Euler class} is the image $e(X)\in\C_\bullet(X(k); \Z)$ of $e_\K(X)$ under the map $\chi\colon \K(k)\rightarrow \Z$.
\label{def:simplestructure}
\end{defn}

\begin{remark}
One can think of an Euler structure as a pair of a simple structure together with a trivialization of the Euler class $e_\K(X)\in\C_\bullet(X(k);\K(k))$.
\end{remark}

\begin{remark}
Tracing through the definitions one obtains that the pushforward of the Euler class $e(X)$ along $X(k)\rightarrow \pt$ coincides with the Euler characteristic of $X$.
\end{remark}

We will now state a version of \cref{thm:pushforwardEulertrivial} in the presence of a simple structure on $X$ rather than an Euler structure. Consider the composite
\[\langle -, -\rangle\colon \C^\bullet(X(k); \K(S))\otimes \C_\bullet(X(k); \K(k))\longrightarrow \K(S)\otimes \K(k)\longrightarrow \K(S),\]
where the first map is the natural pairing between chains and cochains on $X(k)$ and the second map is induced by the tensor product.

\begin{thm}
Suppose $X$ is a derived prestack equipped with a simple structure. Then the pushforward
\[\pi_\sharp\colon \uK(X)\longrightarrow \uK\]
factors as
\[\uK(X)\xrightarrow{\epsilon} \C^\bullet(X(k); \uK)\xrightarrow{\langle -, e_\K(X)\rangle} \uK.\]
\label{thm:indextheorem}
\end{thm}
\begin{proof}
The map
\[\K(S\times X)\longrightarrow \K^\omega(S\times X)\]
factors as 
\[\K(S\times X)\xrightarrow{\id\otimes [\cO_X]} \K(S\times X)\otimes \K^\omega(X)\longrightarrow \K^\omega(S\times X).\]
Given a simple structure on $X$, the latter map factors as
\begin{align*}
\K(S\times X)&\xrightarrow{\id\otimes e_\K(X)} \K(S\times X)\otimes \C_\bullet(X(k);\K(k))\\
&\xrightarrow{\alpha} \K(S\times X)\otimes \K^\omega(X)\\
&\longrightarrow \K^\omega(S\times X).
\end{align*}

The projection formula gives a commutative diagram
\[
\xymatrix{
& \K(S\times X)\otimes \C_\bullet(X(k); \K(k)) \ar_{\epsilon}[dl] \ar^{\alpha}[dr] & \\
\C^\bullet(X; \K(S)) \otimes \C_\bullet(X; \K(k)) \ar^{\otimes}[d] && \K(S\times X)\otimes \K^\omega(X) \ar^{\otimes}[d] \\
\C_\bullet(X; \K(S)) \ar^{\alpha}[rr] && \K^\omega(S\times X)
}
\]
Thus, the original map factors as
\begin{align*}
\K(S\times X)&\xrightarrow{\id\otimes e_\K(X)} \K(S\times X)\otimes \C_\bullet(X(k);\K(k))\\
&\xrightarrow{\epsilon} \C^\bullet(X(k);\K(S))\otimes \C_\bullet(X(k);\K(k))\\
&\xrightarrow{\otimes} \C_\bullet(X(k);\K(S))\\
&\xrightarrow{\alpha} \K^\omega(S\times X).
\end{align*}
Postcomposing with the pushforward map $\pi_\sharp\colon \K^\omega(S\times X)\rightarrow \K(S)$ and identifying the composite
\[\C_\bullet(X(k);\K(S))\xrightarrow{\alpha} \K^\omega(S\times X)\xrightarrow{\pi_\sharp} \K(S)\]
with the homology along $X(k)$ we get the claim.
\end{proof}

\begin{remark}
Consider a derived prestack $X$ with a simple structure and write
\[e_\K(X) = \sum_i x_i\alpha_i\]
for some points $x_i\in X(k)$ and $\alpha_i\in \Omega^\infty\K(k)$. Let $\cF\in\Perf(X)$ be a perfect complex. Then \cref{thm:indextheorem} provides a homotopy
\[[\pi_\sharp \cF]\sim \sum_i [\cF_{x_i}] \alpha_i\]
in $\Omega^\infty\K(k)$. This is an example of an $\epsilon$-factorization in the sense of \cite[Proposition 4.1]{Beilinson}.
\end{remark}

By taking pushouts we may glue derived prestacks with simple structures to obtain new derived prestacks with a simple structure as follows.

\begin{prop}
Consider a pushout diagram of derived prestacks
\[
\xymatrix{
X_0 \ar[r]^f \ar[d]_g & X_1 \ar[d]^{g'} \\
X_2 \ar[r]_{f'} & X
}
\]
Moreover, assume $X_0, X_1,X_2$ carry simple structures and suppose the functors $f^*\colon \QCoh(X_1)\rightarrow \QCoh(X_0)$ and $g^*\colon \QCoh(X_2)\rightarrow \QCoh(X_0)$ admit left adjoints $f_\sharp$ and $g_\sharp$ satisfying the projection formula. Then $X$ carries a \defterm{glued simple structure}.
\label{prop:gluedsimple}
\end{prop}
\begin{proof}
Let us begin by verifying \cref{mainassumptionassembly} for $X$. By definition $\QCoh(-)$ takes colimits of derived prestacks to limits of $\infty$-categories, so we have a Cartesian diagram of $\infty$-categories
\[
\xymatrix{
\QCoh(X) \ar^{(g')^*}[r] \ar^{(f')^*}[d] & \QCoh(X_1) \ar^{f^*}[d] \\
\QCoh(X_2) \ar^{g^*}[r] & \QCoh(X_0)
}
\]
As $f^*$ and $g^*$ preserve limits, so do $(f')^*$ and $(g')^*$ (as the inclusion $\mathrm{Pr}^\mathrm{R}\subset \widehat{\mathrm{Cat}}_\infty$ preserves limits). Therefore, passing to left adjoints we have a coCartesian diagram in $\mathrm{Pr}^{\mathrm{L}}$ (with functors satisfying the projection formula)
\[
\xymatrix{
\QCoh(X_0) \ar^{f_\sharp}[r] \ar^{g_\sharp}[d] & \QCoh(X_1) \ar^{g'_\sharp}[d] \\
\QCoh(X_2) \ar^{f'_\sharp}[r] & \QCoh(X)
}
\]
and a coCartesian square in $\QCoh(X)$
\begin{equation}\label{diag:pushoutofOX}
\xymatrix{    (g' \circ f)_\sharp \cO_{X_0} \ar[r] \ar[d] & g'_\sharp \cO_{X_1} \ar[d] \\
    f'_\sharp \cO_{X_2} \ar[r] & \cO_{X},
}
\end{equation}

As $\cO_{X_i}\in\QCoh(X_i)$ are compact, we get that $\cO_X\in\QCoh(X)$ is compact. Moreover, $p_\sharp\colon\QCoh(X)\rightarrow \Mod_k$ exists and it is tautologically induced by the compatible family of functors $(p_i)_\sharp\colon \QCoh(X_i)\rightarrow \Mod_k$ using the equivalence $\QCoh(X)\cong \QCoh(X_1)\coprod_{\QCoh(X_0)} \QCoh(X_2)$ in $\mathrm{Pr}^{\mathrm{L}}$.

Finally, consider a point of $X_1(k)$ corresponding to a map $i\colon \pt\rightarrow X_1$. Then the composite \[\QCoh(X)\xrightarrow{(f')^*} \QCoh(X_1)\xrightarrow{i^*} \Mod_k\] admits a left adjoint satisfying the projection formula given by $f'_\sharp\circ i_\sharp$ and similarly for points in $X_2(k)$. This immediately implies the claim for points in $X(k)=X_1(k)\coprod_{X_0(k)} X_2(k)$.

We can now produce a simple structure on $X$ by gluing together the simple structures on $X_i$ using the following homotopies in $\Omega^\infty\K^\omega(X)$:
\[
[\cO_{X}] \sim g'_\sharp[\cO_{X_1}] + f'_\sharp[\cO_{X_2}] - (g' \circ f)_\sharp [\cO_{X_0}] \sim \alpha( g' e_\K(X_1) + f' e_\K(X_2) - (g' \circ f) e_\K(X_0)).
\]
Here the first homotopy is obtained from the additivity homotopy by using the fiber sequence
\[f'_\sharp \cO_{X_1}\oplus g'_\sharp \cO_{X_2}\longrightarrow \cO_X\longrightarrow (g' \circ f)_\sharp \cO_{X_0}[1]\]
coming from the pushout square \eqref{diag:pushoutofOX} and the second homotopy is obtained from the simple structures of $X_i$.
\end{proof}

\subsection{Duality}

In this section we consider an even stronger assumption on the derived prestack $X$. Recall that the $\infty$-category $\PrSt_k$ of $k$-linear presentable stable $\infty$-categories has a natural symmetric monoidal structure with the unit given by $\Mod_k$.

\begin{assumption}
Let $X$ be a derived prestack satisfying \cref{mainassumptionassembly} and the following conditions:
\begin{itemize}
    \item The $\infty$-category $\QCoh(X)$ is compactly generated.
    \item The pullback functor $\Delta^*\colon \QCoh(X)\otimes \QCoh(X)\rightarrow \QCoh(X)$ admits a left adjoint
    \[\Delta_\sharp\colon \QCoh(X)\rightarrow \QCoh(X)\otimes \QCoh(X)\]
    satisfying the projection formula, i.e. it is a functor of $\QCoh(X)\otimes \QCoh(X)$-module categories; equivalently, the natural morphism
    \[\Delta_\sharp(\cF\otimes \cG)\longrightarrow \Delta_\sharp(\cF)\otimes (\cG\boxtimes \cO_X)\]
    is an isomorphism.
\end{itemize}
\label{mainassumptionduality}
\end{assumption}

Using the symmetric monoidal structure on $\PrSt_k$ we can talk about dualizable objects in $\PrSt_k$. Given two such dualizable categories $\cC,\cD\in\PrSt_k$ with duals $\cC^\vee,\cD^\vee\in\PrSt_k$ as well as a colimit-preserving functor $F\colon \cC\rightarrow \cD$, there is a naturally defined dual functor $F^\vee\colon \cD^\vee\rightarrow \cC^\vee$ which is uniquely specified by a natural isomorphism
\[\ev_\cD(F(x), y)\cong \ev_\cC(x, F^\vee(y))\]
for $x\in\cC$ and $y\in\cD^\vee$.

\begin{thm}
Let $X$ be a derived prestack satisfying \cref{mainassumptionduality}. Then:
\begin{enumerate}
    \item For any derived prestack $Y$ the natural functor $\boxtimes\colon \QCoh(X)\otimes \QCoh(Y)\rightarrow \QCoh(X\times Y)$ is an equivalence.
    \item The functors
    \[\ev\colon \QCoh(X)\otimes \QCoh(X)\xrightarrow{\Delta^*} \QCoh(X)\xrightarrow{p_\sharp}\Mod_k\]
    and
    \[\coev\colon \Mod_k\xrightarrow{p^*}\QCoh(X)\xrightarrow{\Delta_\sharp}\QCoh(X)\otimes \QCoh(X)\]
    establish a self-duality of $\QCoh(X)$ in $\PrSt_k$.
    \item Under this self-duality of $\QCoh(X)$ the functors $p_\sharp\colon \QCoh(X)\rightarrow \Mod_k$ and $p^*\colon \Mod_k\rightarrow \QCoh(X)$ are dual to each other.
    \item The symmetric bilinear functor
    \[B\colon \QCoh(X)^\omega\otimes \QCoh(X)^\omega\longrightarrow \Sp\]
    given by
    \[(x,y)\mapsto \Hom_{\QCoh(X\times X)}(x\boxtimes y, \Delta_\sharp \cO_X)\]
    is non-degenerate, i.e. there is an equivalence
    \[\D\colon \QCoh(X)^{\omega,\op}\longrightarrow \QCoh(X)^\omega\]
    satisfying $B(x, y)\cong \Hom_{\QCoh(X)}(x, \D(y))$. Moreover,
    \[\tilde{B}(x, y) = \Hom_{\QCoh(X)}(\D(x), y)\cong p_\sharp(x\otimes y)\]
    for every $x,y\in\QCoh(X)^\omega$.
\end{enumerate}
\label{thm:QCohselfdual}
\end{thm}
\begin{proof}
Since $\QCoh(X)$ is compactly generated, it is dualizable by \cite[Proposition D.7.2.3]{LurieSAG}. The first statement then follows from \cite[Lemma B.2.3]{Gaitsgory}. The second statement is standard (see e.g. \cite[Proposition 2.17]{HSSS} for a related statement).

The third statement follows from the obvious isomorphisms
\[\ev_{\QCoh(X)}(p^* V, \cF) = p_\sharp(V\otimes \cO_X\otimes \cF)\cong V\otimes p_\sharp \cF = \ev_{\Mod_k}(V, p_\sharp \cF).\]

Let us now prove the fourth statement. By \cite[Proposition D.7.2.3]{LurieSAG} the category $\QCoh(X)$ has a duality data with the dual $\Ind(\QCoh(X)^{\omega,\op})$ and the evaluation functor given by $\Hom_{\QCoh(X)}(x, y)$ for $x\in\QCoh(X)^{\omega,\op}$ and $y\in\QCoh(X)^\omega$. By (1) there is another duality data. So, by uniqueness of the duality data we obtain an equivalence
\[\D\colon \Ind(\QCoh(X)^{\omega, \op})\cong \QCoh(X)\]
which intertwines the two duality data. Concretely, it is given by the formula
\[\D(x) = \Hom^{\boxtimes}(x, \Delta_\sharp \cO_X),\]
where $\Hom^\boxtimes(x, -)\colon \QCoh(X\times X)\rightarrow \QCoh(X)$ is the right adjoint to the functor $\QCoh(X)\rightarrow \QCoh(X\times X)$ given by $y\mapsto x\boxtimes y$. This formula provides an isomorphism
\[\Hom_{\QCoh(X)}(x, \D(y))\cong \Hom_{\QCoh(X\times X)}(x\boxtimes y, \Delta_\sharp \cO_X).\]
Comparing the two evaluation functors, we get an isomorphism
\[\Hom_{\QCoh(X)}(\D(x), y)\cong p_\sharp(x\otimes y).\]
\end{proof}

\begin{example}
$X=\pt$ satisfies \cref{mainassumptionduality}. In this case the duality functor on $\QCoh(\pt)^\omega\cong \Perf_k$ is the usual linear duality.
\end{example}

We will use the language of Poincare $\infty$-categories from \cite{CDHHLMNNS1} and their Grothendieck--Witt spectra. The main points we will use are the following:
\begin{itemize}
    \item A Poincar\'e structure on a stable $\infty$-category $\cC$ consists of a quadratic functor $\Qoppa\colon \cC^{\op}\rightarrow \Sp$ such that there is an equivalence $\D\colon \cC^{\op}\rightarrow \cC$ satisfying $B(x, y)\cong \Hom_\cC(x, \D(y))$, where $B\colon \cC^{\op}\times \cC^{\op}\rightarrow \Sp$ is the underlying symmetric bilinear functor of $\Qoppa$.
    \item The equivalence $\D\colon \cC^{\op}\rightarrow \cC$ defines the structure of a $C_2$-homotopy fixed point on $\cC$ in the $\infty$-category of categories. In particular, there is an induced $C_2$-action on the $K$-theory spectrum $\K(\cC)$ and we may consider the spectrum of invariants $\K(\cC)^{C_2}$.
    \item Given a Poincar\'e $\infty$-category $(\cC, \Qoppa)$ we may talk about Poincar\'e objects which are objects $x\in\cC$ equipped with an element of $\Omega^\infty\Qoppa(x)$ such that the induced map $x\rightarrow \D(x)$ is an isomorphism. Let $\Pn(\cC, \Qoppa)$ be the space of Poincare objects in $\cC$.
    \item Given a Poincar\'e $\infty$-category $(\cC, \Qoppa)$ there is the corresponding Grothendieck--Witt spectrum $\GW(\cC)$ with a forgetful map $\GW(\cC)\rightarrow \K(\cC)^{C_2}$.
    \item Conversely, given any symmetric bilinear functor $B\colon \cC^{\op}\times \cC^{\op}\rightarrow \Sp$ such that there is an equivalence $\D\colon \cC^{\op}\rightarrow \cC$ satisfying $B(x, y)\cong \Hom_\cC(x, \D(y))$ (this is equivalent to $\cC$ being an $\infty$-category with duality in the sense of \cite{HLAS}) the functor $\Qoppa\colon \cC^{\op}\rightarrow \Sp$ given by $\Qoppa(x) = B(x, x)^{C_2}$ defines a Poincar\'e structure on $\cC$. All our examples will be of this form.
    \item In the previous setting there is a Poincar\'e structure $\Qoppa^{\op}$ on $\cC^{\op}$ with underlying symmetric bilinear functor $\tilde{B}(x, y) = \Hom_\cC(\D(x), y)$. There is a natural equivalence of spaces $\Pn(\cC, \Qoppa)\cong \Pn(\cC^{\op}, \Qoppa^{\op})$.
\end{itemize}

\begin{remark}
Even though we use the formalism of Grothendieck--Witt spectra from \cite{CDHHLMNNS1}, for our purposes it is enough to use earlier definitions of Grothendieck--Witt spectra for dg categories with duality as in \cite{Schlichting}.
\end{remark}

Consider the following Poincar\'e structures:
\begin{itemize}
    \item By the previous proposition the quadratic functor
    \[\Qoppa\colon \QCoh(X)^{\omega,\op}\longrightarrow \Sp\]
    given by $x\mapsto B(x, x)^{C_2} = \Hom(x\boxtimes x, \Delta_\sharp \cO_X)^{C_2}$ is a Poincar\'e functor. We denote its shifts by
    \[\Qoppa^{[n]}(x) = \Qoppa(x)[n].\]
    \item For any derived prestack $\Perf(X)$ is a rigid symmetric monoidal $\infty$-category. In particular, the quadratic functor
    \[\Qoppa\colon \Perf(X)^{\op}\longrightarrow \Sp\]
    given by $x\mapsto \Hom_{\QCoh(X)}(x\otimes x, \cO_X)^{C_2}$ is Poincar\'e. As before, we denote by $\Qoppa^{[n]}$ its shift.
    \item The Poincar\'e structure $\Qoppa^{[n]}$ on $\Perf(X)$ induces one on $\Perf^\vee(X)$ using the internal Hom of Poincar\'e categories described in \cite[Remark 6.2.4]{CDHHLMNNS1}.
\end{itemize}

Using the formalism of \cite{CDHHLMNNS2} we can define the Grothendieck--Witt (alias, hermitian $K$-theory) spectra:
\begin{itemize}
    \item $\GW^{\omega, [n]}(X) = \GW(\QCoh(X)^\omega, \Qoppa^{[n]})$.
    \item $\GW^{[n]}(X) = \GW(\Perf(X), \Qoppa^{[n]})$.
    \item $\uGW^{[n]}(X)$ is the prestack which sends a derived affine scheme $S$ to $\uGW^{[n]}(S\times X)$. $\uGW^{[n]}=\uGW^{[n]}(\pt)$ is the prestack which sends a derived affine scheme $S$ to $\uGW^{[n]}(S)$.
    \item $\GW^{[n], \vee}(X) = \GW(\Perf^\vee(X), \Qoppa^{[n]})$.
\end{itemize}

\begin{remark}
The functor $\cC\rightarrow \cC$ given by $x\mapsto x[2]$ defines an equivalence of Poincar\'e structures $\Qoppa^{[n]}$ and $\Qoppa^{[n-4]}$. So, the Grothendieck--Witt spectra $\GW^{[n]}$ are 4-periodic in $n$.
\end{remark}

We will now show that several natural functors preserve Poincar\'e structures.

\begin{prop}
Let $X$ be a derived prestack satisfying \cref{mainassumptionduality}.
\begin{enumerate}
    \item For a point $i\colon \pt\rightarrow X$ and a derived affine scheme $S$ the pullback functor
    \[(\id\times i)^*\colon \Perf(S\times X)\rightarrow \Perf(S)\]
    is Poincar\'e.
    \item For a point $i\colon \pt\rightarrow X$ the pushforward functor $i_\sharp\colon \Perf(k)\rightarrow \QCoh(X)^\omega$ is Poincar\'e.
    \item The functor $\tens_X\colon \QCoh(X)^\omega\rightarrow \Perf^\vee(X)$ is Poincar\'e.
\end{enumerate}
\label{prop:Poincarefunctors}
\end{prop}
\begin{proof}
The pullback functor $(\id\times i)^*\colon \Perf(S\times X)\rightarrow \Perf(S)$ is symmetric monoidal, so it is a Poincar\'e functor.

Let us now show that $i_\sharp\colon \Perf(k)\rightarrow \QCoh(X)^\omega$ is Poincar\'e. For this we need to show that it intertwines the symmetric bilinear functors and preserves the duality strictly. Indeed, the functor $i_\sharp$ has an oplax symmetric monoidal structure. Therefore, the composite
\[V\otimes W\cong p_\sharp(i_\sharp(V\otimes W))\longrightarrow p_\sharp(i_\sharp V\otimes i_\sharp W)\]
for $V,W\in\Perf(k)$ defines a natural transformation of symmetric bilinear functors underlying Poincar\'e structures on $\Perf(k)^{\op}$ and $\QCoh(X)^{\omega, \op}$. The corresponding dualities are preserved strictly as shown by the sequence of isomorphisms
\[\Hom_{\QCoh(X)}(\D(i_\sharp V), x)\cong p_\sharp(i_\sharp V\otimes x)\xleftarrow{\sim} p_\sharp i_\sharp(V\otimes i^* x)\cong V\otimes i^* x\cong \Hom_{\QCoh(X)}(i_\sharp V^\vee, x).\]

Finally, to show that $\tens_X\colon \QCoh(X)^\omega\rightarrow \Perf^\vee(X)$ is Poincar\'e, we have to show that the functor $\QCoh(X)^\omega\rightarrow \Fun^{ex}(\Perf(S\times X), \Perf(S))$ given by $x\mapsto (y\mapsto \pi_\sharp(x\otimes y))$ is Poincar\'e naturally in $S$. In turn, this is equivalent to showing that
\[\QCoh(X)^\omega\otimes \Perf(S\times X)\longrightarrow \Perf(S)\]
given by $x,y\mapsto \pi_\sharp(x\otimes y)$ is Poincar\'e with respect to the tensor product of the Poincar\'e structures on the left. For this we need to construct a natural transformation
\[p_\sharp(x_1\otimes x_2)\otimes \pi_*(y_1\otimes y_2)\longrightarrow \pi_\sharp(x_1\otimes y_1)\otimes \pi_\sharp(x_2\otimes y_2)\]
which is natural in $x_1,x_2\in\QCoh(X)^\omega$, $y_1,y_2\in\Perf(S\times X)$ and symmetric under the permutation $(x_1, y_1)\leftrightarrow (x_2, y_2)$. By duality this is equivalent to providing a map
\[\pi_\sharp(\cO_S\otimes x_1\otimes x_2)\longrightarrow \pi_\sharp(x_1\otimes y_1)\otimes \pi_\sharp(x_2\otimes y_2)\otimes \pi_\sharp(y_1^\vee\otimes y_2^\vee).\]
This map arises by applying the coevaluation for $y_1, y_2$ and the oplax symmetric monoidal structure on $\pi_\sharp$. To show that this hermitian functor is Poincar\'e we have to check that the corresponding dualities are preserved strictly. This follows from the following sequence of natural isomorphisms in $x\in\QCoh(X)^\omega$, $y\in \Perf(S\times X)$ and $z\in\Perf(S)$:
\begin{align*}
\Hom_{\QCoh(S)}(z, \pi_\sharp(\D(x)\otimes y^\vee))&\cong \Hom_{\QCoh(S\times X)}(x\otimes z, y^\vee)\\
&\cong \Hom_{\QCoh(S\times X)}(x\otimes y, z^\vee\boxtimes \cO_X)\\
&\cong \Hom_{\QCoh(S)}(\pi_\sharp(x\otimes y), z^\vee),
\end{align*}
which shows that $\pi_\sharp(x\otimes y)^\vee\cong \pi_\sharp(\D(x)\otimes y^\vee)$ in $\Perf(S)$.
\end{proof}

In particular, by the above proposition we obtain the assembly
\[\alpha\colon \C_\bullet(X(k); \GW^{[n]}(k))\longrightarrow \GW^{\omega, [n]}(X)\]
and coassembly
\[\epsilon\colon \uGW^{[n]}(X)\longrightarrow \C^\bullet(X(k); \uGW^{[n]})\]
maps which fit into commutative diagrams
\[
\xymatrix{
\C_\bullet(X(k); \GW^{[n]}(k)) \ar[r] \ar[d] & \GW^{\omega, [n]}(X) \ar[d] \\
\C_\bullet(X(k); \K(k)) \ar[r] & \K^\omega(X) 
}
\]
and
\[
\xymatrix{
\uGW^{[n]}(X) \ar[d] \ar[r] & \C^\bullet(X(k); \uGW^{[n]}) \ar[d] \\
\uK(X) \ar[r] & \C^\bullet(X(k); \uK).
}
\]

We will now define analogs of Poincare duality spaces.

\begin{defn}
Let $d\in\Z$ and $X$ a derived prestack satisfying \cref{mainassumptionduality}. A \defterm{fundamental class of $X$ of degree $d$} is a map
\[[X]\colon k\longrightarrow p_\sharp\cO_X[-d]\]
which is a unit of an adjunction $p^*[-d]\dashv p_\sharp$.
\end{defn}

Recall from \cite[Definition 2.4]{PTVV} the notion of an \defterm{$\cO$-orientation of degree $d$} on an $\cO$-compact derived prestack $X$ which is a morphism $p_*\cO_X\rightarrow k[-d]$ satisfying a nondegeneracy property. By \cref{prop:assumptioncorollaries} \cref{mainassumptionduality} implies that $X$ is $\cO$-compact.

\begin{prop}
Suppose $2$ is invertible in $k$. Let $X$ be a derived prestack satisfying \cref{mainassumptionduality}. The following pieces of data are equivalent:
\begin{enumerate}
    \item The structure of a Poincar\'e object on $\cO_X\in \QCoh(X)^\omega$ with respect to the Poincar\'e structure $\Qoppa^{[d]}$.
    \item A fundamental class of $X$ of degree $d$.
\end{enumerate}
Moreover, either of them gives rise to an $\cO$-orientation of degree $d$ on $X$.
\label{prop:fundamentalclassorientation}
\end{prop}
\begin{proof}
Recall from \cref{thm:QCohselfdual} that there are natural isomorphisms
\[\Hom(x\boxtimes y, \Delta_\sharp \cO_X)\cong \Hom(x, \D(y)),\qquad p_\sharp(x\otimes y)\cong \Hom(\D(x), y).\]

The structure of a Poincar\'e object on $\cO_X$ is that of a symmetric map
\[\cO_X\boxtimes \cO_X\rightarrow\Delta_\sharp \cO_X[d]\]
such that the induced map $\cO_X\rightarrow \D(\cO_X)[d]$ is an isomorphism. Recall from \cref{thm:QCohselfdual} that there is a symmetric self-duality data on $\QCoh(X)\in\PrSt_k$ with
\[\coev(k) = \Delta_\sharp\cO_X[d],\qquad \ev(x, y) = p_\sharp(x\otimes y).\]
Then the data of a Poincare object is that of a nondegenerate symmetric map
\[\cO_X\boxtimes \cO_X\rightarrow \coev(k)[d].\]
It is equivalent to the data of a nondegenerate symmetric map
\[k\rightarrow \ev(\cO_X, \cO_X)[-d] = p_\sharp \cO_X[-d],\]
where nondegeneracy means that the induced map $\D(\cO_X)\rightarrow \cO_X[-d]$ is an isomorphism. As $\cO_X$ is the unit, the $C_2$-action on $p_\sharp\cO_X$ is trivial and, since $2$ is invertible, $(p_\sharp \cO_X)^{C_2}\cong p_\sharp \cO_X$.

For $k\rightarrow p_\sharp \cO_X[-d]$ to be a fundamental class we need to ensure existence of the counit of the adjunction, i.e. a natural transformation
\[\cO_X\otimes p_\sharp \cF\rightarrow \cF[d]\]
of endofunctors of $\QCoh(X)$. Identifying $\Fun(\QCoh(X), \QCoh(X))\cong \QCoh(X\times X)$ using the self-duality of $\QCoh(X)$, such a counit is the same as a morphism
\[\cO_X\boxtimes \cO_X\rightarrow \Delta_\sharp \cO_X[d].\]
The adjunction axioms boil down to the condition that the map $\D(\cO_X)\rightarrow \cO_X[-d]$ induced by $k\rightarrow p_\sharp \cO_X[-d]$ is inverse to the map $\cO_X[-d]\rightarrow \D(\cO_X)$ induced by $\cO_X\boxtimes \cO_X\rightarrow \Delta_\sharp \cO_X[d]$. This shows the equivalence of the first two pieces of structure.

An $\cO$-orientation is the data of a morphism $p_*\cO_X\rightarrow k[-d]$ such that for derived affine scheme $S$ and a perfect complex $\cF\in\Perf(S\times X)$ the natural morphism
\[\pi_*\cF\rightarrow (\pi_*(\cF^\vee))^\vee[-d]\]
induced by $[X]$ is an isomorphism, where $\pi=\id\times p\colon S\times X\rightarrow S$ is the projection on the first factor. We may identify $(\pi_\sharp \cF)^\vee\cong \pi_* (\cF^\vee)$. Thus, to check that a given morphism $p_*\cO_X\rightarrow k[-d]$ is an $\cO$-orientation, we need to show that the natural morphism
\[\pi_*\cF\rightarrow \pi_\sharp \cF[-d]\]
is an isomorphism for every $\cF\in\Perf(S\times X)$.

Now fix a fundamental class on $X$ of degree $d$. The dual of $[X]\colon k\rightarrow p_\sharp \cO_X[-d]$ is a morphism $p_*\cO_X[d]\rightarrow k$. The fundamental class $[X]$ provides a natural isomorphism $p_*\rightarrow p_\sharp[-d]$ as $p_*$ is defined to be the right adjoint of $p^*$. We have $\pi_\sharp = (\id\otimes p_\sharp)$ and $\pi_*=(\id\otimes p_*)$ and it is easy to see that the morphism $\pi_*\cF\rightarrow \pi_\sharp \cF[-d]$ appearing in the definition of $\cO$-orientation is induced by the isomorphism $p_*\rightarrow p_\sharp[-d]$ and is, therefore, an isomorphism.
\end{proof}

So, a fundamental class on $X$ defines a point $[\cO_X]\in\Omega^\infty\GW^{\omega, [d]}(X)$. By \cref{prop:Poincarefunctors} the functor $\tens_X\colon \QCoh^\omega(X)\rightarrow \Perf^\vee(X)$ is Poincar\'e, so $\pi_\sharp\colon (\uPerf(X), \Qoppa^{[n]})\rightarrow (\uPerf, \Qoppa^{[n+d]})$ preserves Poincar\'e structures for any $n\in\Z$. Thus, in this case $\pi_\sharp$ descends to a map
\[\pi_\sharp\colon \uGW^{[n]}(X)\longrightarrow \uGW^{[n+d]}\]
of Grothendieck--Witt spectra.

We will now investigate what happens to the factorization \cref{thm:indextheorem} in the presence of a compatible fundamental class.

\begin{defn}
Let $X$ be a derived prestack equipped with a fundamental class $[X]$ of degree $d$ and a simple structure. We say the simple structure \defterm{is compatible with Poincar\'e duality} if we are given the Euler class $e_{\GW}(X)\in\Omega^\infty\C_\bullet(X(k); \GW^{[d]}(k))$ together with a homotopy $\alpha(e_{\GW}(X))\sim [\cO_X]$ in $\Omega^\infty\GW^{\omega, [d]}(X)$ which projects to the given simple structure in $\Omega^\infty\K^\omega(X)$.
\end{defn}

\begin{remark}
Note that the natural map
\[\C_\bullet(X(k); \tau_{\geq 0} \GW^{[d]}(k))\longrightarrow \tau_{\geq 0}\C_\bullet(X(k); \GW^{[d]}(k))\]
is not an equivalence as $\GW^{[d]}(k)$ is not connective (its negative homotopy groups are, up to a shift, the $L$-groups of $k$).
\end{remark}

For a derived affine scheme $S$ consider the composite
\[\langle-, -\rangle\colon \C^\bullet(X(k); \GW^{[n]}(S))\otimes \C_\bullet(X(k); \GW^{[d]}(k))\longrightarrow \GW^{[n]}(S)\otimes \GW^{[d]}(k)\longrightarrow \GW^{[n+d]}(S),\]
where the first map is given by the natural pairing between chains and cochains on $X(k)$ and the multiplication map on the Grothendieck--Witt spectra induced by the tensor product. The following statement is proven analogously to \cref{thm:indextheorem}.

\begin{thm}
Suppose $X$ is a derived prestack equipped with a fundamental class of degree $d$ and a simple structure compatible with Poincar\'e duality. Then the pushforward
\[\pi_\sharp\colon \uGW^{[n]}(X)\longrightarrow \uGW^{[n+d]}\]
factors as
\[\uGW^{[n]}(X)\xrightarrow{\epsilon} \C^\bullet(X(k); \uGW^{[n]})\xrightarrow{\langle -, e_{\GW}(X)\rangle} \uGW^{[n+d]}.\]
\label{thm:indextheoremwithPoincare}
\end{thm}

\section{Determinant lines and volume forms}
\label{sect:determinant}

In this section we construct the determinant line over the moduli stack of perfect complexes and the torsion volume form on the mapping stack.

\subsection{Perfect complexes and their determinants}

Let $R$ be a connective commutative dg $k$-algebra. There is a weight structure on the stable $\infty$-category $\Perf(R)$ whose heart is $\Vect(R)\subset \Perf(R)$, the subcategory of projective finitely generated $R$-modules (i.e. retracts of $R$-modules of the form $R^{\oplus n}$). The $\infty$-category $\Perf(R)$ has a natural $\bE_\infty$ semiring structure in the sense of \cite{GGN} with respect to the symmetric monoidal structures $\oplus$ and $\otimes$, see \cite[Example 8.12]{GGN}.

Consider the following derived stacks:
\begin{itemize}
\item $\uVect$ is the derived stack of vector bundles which sends $R\mapsto \Vect(R)$.
\item $\uPic=\B\GL_1$ is the derived stack of line bundles.
\item $\uPic^\Z=\uPic \times \Z$ is the derived stack of graded line bundles, where $\Z$ is the \'etale sheafification of the constant presheaf with value $\Z$.
\end{itemize}

If $R$ is a (discrete) commutative $k$-algebra, the groupoid $\Pic^\Z(R)$ has a $\bE_\infty$ ring structure. This can be modeled by a $\bE_\infty$-algebra object (we denote the corresponding symmetric monoidal structure by $\motimes$) in the (2, 1)-category of Picard groupoids (we denote the symmetric monoidal structure of the Picard groupoid by $\otimes$). We refer to \cite{Laplaza} for the distributivity conditions and axioms that $\otimes$ and $\motimes$ have to satisfy. In the case $\Pic^\Z(R)$ the data is as follows:
\begin{itemize}
    \item The first tensor product is $(\cL_1, n_1)\otimes (\cL_2, n_2) = (\cL_1\otimes \cL_2, n_1+n_2)$. It has an obvious associator and the braiding given by the flip on $\cL_1\otimes \cL_2$ multiplied by the sign $(-1)^{n_1n_2}$. The unit is $(\cO, 0)$.
    \item The second tensor product is $(\cL_1, n_1)\motimes (\cL_2, n_2) = (\cL_1^{n_2}\otimes \cL_2^{n_1}, n_1n_2)$. It has an obvious associator and the braiding given by the flip on $\cL_1^{n_2}\otimes \cL_2^{n_1}$ multiplied by the sign $(-1)^{(n_1(n_1-1)/2)(n_2(n_2-1)/2)}$. The unit is $(\cO, 1)$.
    \item For $(\cL_1, n_1), (\cL_2, n_2), (\cL_3, n_3)\in\Pic^\Z(R)$ we let the left distributivity isomorphism
    \[\cL_1^{n_2+n_3}\otimes (\cL_2\otimes \cL_3)^{n_1}\cong \cL_1^{n_2}\otimes \cL_2^{n_1}\otimes \cL_1^{n_3}\otimes \cL_3^{n_1}\]
    be the obvious isomorphism of line bundles multiplied by the sign $(-1)^{n_2n_3n_1(n_1-1)/2}$.
    \item For $(\cL_1, n_1), (\cL_2, n_2), (\cL_3, n_3)\in\Pic^\Z(R)$ we let the right distributivity isomorphism
    \[(\cL_1\otimes \cL_2)^{n_3}\otimes \cL_3^{n_1+n_2}\cong \cL_1^{n_3}\otimes \cL_3^{n_1}\otimes \cL_2^{n_3}\otimes \cL_3^{n_2}\]
    be the obvious isomorphism of line bundles.
\end{itemize}

\begin{prop}
Let $R$ be a (discrete) commutative $k$-algebra. There is a functor
\[\detgr\colon \Vect(R)^\sim\longrightarrow \Pic^{\Z}(R),\]
natural in $R$, of $\bE_\infty$ semiring categories sending $V\in\Vect(R)$ to $(\det(V)=\wedge^{\rk V} V, \rk V)$. The monoidal structure with respect to $\otimes$ is given by the isomorphism
\[\detgr(V)\otimes \detgr(W)\longrightarrow \detgr(V\oplus W)\]
given by
\[\wedge_{i=1}^{\rk V} v_i\otimes \wedge_{j=1}^{\rk W} w_i\mapsto (\wedge_{i=1}^{\rk V} v_i)\wedge (\wedge_{j=1}^{\rk W} w_i).\]
The monoidal structure with respect to $\motimes$ is given by the isomorphism
\[\detgr(V)\motimes \detgr(W)\longrightarrow \detgr(V\otimes W)\]
which is an isomorphism of line bundles
\[\det(V)^{\rk W}\otimes \det(W)^{\rk V}\longrightarrow \det(V\otimes W)\]
given by
\[(\otimes_{j=1}^{\rk W}\wedge_{i=1}^{\rk V} v_{ij})\otimes (\otimes_{i=1}^{\rk V}\wedge_{j=1}^{\rk W} w_{ij})\mapsto \wedge_{i=1}^{\rk V}\wedge_{j=1}^{\rk W} (v_{ij}\otimes w_{ij}).\]
\label{prop:determinantclassical}
\end{prop}
\begin{proof}
The fact that $\detgr\colon (\Vect(R)^\sim, \oplus)\rightarrow (\Pic^\Z(R), \otimes)$ is a symmetric monoidal functor is shown in \cite{KnudsenMumford}. The compatibility with the $\bE_\infty$ semiring structure is straightforward.
\end{proof}

Let us now extend the determinant functor to $K$-theory following \cite[Section 3.1]{STV} and \cite[Theorem 3.9]{Heleodoro}. Since $(\Pic^\Z(R), \otimes)$ is a Picard groupoid, we may regard it as a connective spectrum.

\begin{thm}
$\uPic^\Z$ is a stack of $\bE_\infty$ ring spectra on the site of derived affine schemes. Moreover, there is a morphism of prestacks of $\bE_\infty$ ring spectra
\[\detgr\colon \uK\longrightarrow \uPic^\Z\]
such that $\uVect\rightarrow \uK\rightarrow \uPic^\Z$ coincides with the classical determinant functor from \cref{prop:determinantclassical} when restricted to discrete commutative $k$-algebras.
\label{thm:determinantmap}
\end{thm}
\begin{proof}
Let us first construct an $\bE_\infty$ ring structure on $\uPic^\Z$. As shown in \cite[Section 2]{Heleodoro} and \cite[Appendix A]{EHKSY}, the stacks $\uVect$ and $\uPic^\Z$ are left Kan extended from smooth commutative $k$-algebras (in particular, discrete). Note that the corresponding pointwise formula for the left Kan extension is given by a sifted colimit as observed in \cite[Lemma A.0.5]{EHKSY}. The forgetful map $\Ring_{\bE_\infty}(\cS)\rightarrow \cS$ from $\bE_\infty$-ring spaces to spaces preserves sifted colimits, so to endow $\uPic^\Z$ with an $\bE_\infty$ ring structure it is enough to endow the stack $R\mapsto \Pic^\Z(R)$ for $R$ a smooth commutative $k$-algebra with an $\bE_\infty$ ring structure, which we have already defined above.

By \cite[Corollary 8.1.3]{HebestreitSteimle} (see also \cite[Corollary 1.40]{Heleodoro} and \cite{Fontes}) the natural map of spectra $\K(\Vect(R))\rightarrow \K(\Perf(R))$ is an equivalence for any connective commutative dg $k$-algebra $R$. So, it is enough to construct an $\bE_\infty$ ring structure on $\uPic^\Z$ as well as a morphism of $\bE_\infty$ ring spectra
\[\detgr\colon \K(\Vect(R))\longrightarrow \Pic^\Z(R)\]
natural in $R$. As $\uPic^\Z$ is an $\bE_\infty$ ring, this is the same as a morphism of $\bE_\infty$ semiring spaces
\[\detgr\colon \Vect(R)\longrightarrow \Pic^\Z(R).\]
As $\uVect$ is left Kan extended from smooth commutative $k$-algebras, it is enough to construct this morphism for $R$ classical which was done in \cref{prop:determinantclassical}.
\end{proof}

By precomposition with $\uPerf^\sim\rightarrow \uK$ the determinant map gives rise to a determinant morphism
\[\detgr\colon \uPerf^\sim\longrightarrow \uPic^{\Z}.\]
It splits into a pair of maps
\[\det\colon \uPerf^\sim\longrightarrow \uPic,\qquad \chi\colon \uPerf^\sim\longrightarrow \Z,\]
where we note that $\det\colon \uPerf^\sim\rightarrow \uPic$ is merely an $\bE_1$ map of spaces. The following statement is well-known.

\begin{prop}
Let $R$ be a (discrete) $k$-algebra. Suppose $V^\bullet = (V^{-m}\rightarrow V^{-m+1}\rightarrow\dots\rightarrow V^0)$ is a chain complex of projective finitely-generated $R$-modules. Then there is a canonical homotopy
\[[V^\bullet] \sim \sum_{n=0}^m (-1)^n [V^{-n}]\in\Omega^\infty\K(R).\]
Therefore, if $R$ is commutative, there is a canonical isomorphism
\[\detgr(V^\bullet)\cong \detgr(V^0)\otimes \detgr(V^{-1})^{-1}\otimes \dots\otimes \detgr(V^{-m})^{(-1)^m}.\]
\label{prop:chaincomplexfiltration}
\end{prop}
\begin{proof}
Consider a filtration on $V^\bullet$ by the brutal truncation on cohomological degree. Its associated graded $V^\bullet$ with the zero differential. As the class of any filtered object is equivalent to the class of its associated graded in $K$-theory, this provides the relevant homotopy.

The isomorphism on the level of determinant lines is obtained after applying $\detgr$ to the homotopy in $\K(R)$.
\end{proof}

The following construction will be useful. Suppose $V^\bullet$ is a chain complex as in \cref{prop:chaincomplexfiltration}, its cohomology $\rH^\bullet(V^\bullet)$ consists of projective modules and there is a quasi-isomorphism $\rH^\bullet(V^\bullet)\rightarrow V^\bullet$ (e.g. $R$ is a field). Then we obtain an isomorphism of the determinant lines
\begin{equation}
\phi_V\colon \detgr(V^\bullet)\longrightarrow \detgr(\rH^\bullet(V^\bullet)).
\label{eq:Eulermap}
\end{equation}
We refer to \cite[Section 2.2]{FarberTuraevPR} for explicit formulas.

\begin{defn}
Let $X$ be a derived prestack satisfying \cref{mainassumption}. The \defterm{determinant line bundle} is the graded line bundle on $\uPerf(X)^\sim$ given by the composite
\[\uPerf(X)^\sim\xrightarrow{\pi_\sharp} \uPerf^\sim\xrightarrow{\detgr} \uPic^{\Z}.\]
\end{defn}

Explicitly, for an object $\cF\in\Perf(X)$ the fiber of the determinant line bundle on $\uPerf(X)$ is
\[\cD_\cF = \detgr(\pi_\sharp \cF),\]
the determinant of the homology of $\cF$. If $X$ is equipped with an Euler structure, by \cref{thm:pushforwardEulertrivial} we obtain a \defterm{determinant section} of $\cD$ giving rise to an isomorphism $\cD\cong \cO_{\uPerf(X)}$.

\subsection{Volume forms}

\begin{defn}
Let $Y$ be a derived prestack which admits deformation theory whose cotangent complex $\bL_Y$ is perfect.
\begin{itemize}
\item The \defterm{virtual dimension} of $Y$ is the locally constant function $\dim(Y)\colon Y\rightarrow \Z$ given by $\chi(\bL_Y)$. We say $Y$ has pure dimension $n\in\Z$ if $\dim(Y)$ is the constant function with value $n$.
\item A \defterm{volume form} on $Y$ is an isomorphism $\det(\bL_Y)\cong \cO_Y$.
\item A \defterm{squared volume form} on $Y$ is an isomorphism $\det(\bL_Y)^{\otimes 2}\cong \cO_Y$.
\end{itemize}
\label{def:volumeform}
\end{defn}

\begin{example}
Let $G$ be an algebraic group equipped with a $G$-invariant volume form on the Lie algebra $\g$. Suppose $Y$ is a derived prestack equipped with a $G$-invariant volume form $\vol_Y$. Then there is a natural volume form on the quotient $[Y/G]$. Indeed, under the identification of $\QCoh([Y/G])$ with $G$-equivariant quasi-coherent complexes on $Y$ we have
\[\bL_{[Y/G]}\mapsto (\bL_Y\rightarrow \cO_Y\otimes \g^*).\]
Therefore, the $G$-invariant trivialization $\vol_Y$ of $\det(\bL_Y)$ and a $G$-invariant trivialization of $\det(\g^*)$ induce a trivialization of $\det(\bL_{[Y/G]})$.
\label{ex:quotientvolumeform}
\end{example}

Volume forms can be glued as follows. Suppose $Y_1,Y_2$ and $Y_0$ are derived prestacks equipped with volume forms and consider a pullback diagram
\[
\xymatrix{
Y \ar[r]^{g'} \ar[d]^{f'} & Y_1 \ar[d]^{f} \\
Y_2 \ar[r]^{g} & Y_0.
}
\]

The cotangent complex of $Y$ fits into a Cartesian square
\[
\xymatrix{
    \bL_Y \ar[r] \ar[d] & g'^*\bL_{Y_1} \ar[d] \\
    f'^*\bL_{Y_2} \ar[r] & (gf')^* \bL_{Y_0}.
}
\]
Applying determinants we get an isomorphism
\[
\det(\bL_Y) \cong \det(g'^*\bL_{Y_1}) \otimes \det(f'^*\bL_{Y_2}) \otimes \det((g f')^*\bL_{Y_0})^{-1}.
\]
The volume forms on $Y_1, Y_2, Y_0$ then give a volume form on $Y$. We call it the \defterm{glued volume form}.

In this section we consider mapping prestacks $\Map(X, Y)$, where $Y$ is a derived prestack with a perfect cotangent complex and $X$ a derived prestack satisfying \cref{mainassumption}. By \cref{prop:mappingcotangent} we have
\[
\bL_{\Map(X,Y)} = \pi_\sharp \ev^* \bL_Y
\]
in that case and hence, using that by definition $\cD = \detgr \circ \pi_\sharp$, it follows that
\[
\detgr(\bL_{\Map(X,Y)}) = \detgr(\pi_\sharp \ev^*\bL_Y) = \cD(\ev^*\bL_Y) = g^*\cD,
\]
where $g \colon \Map(X,Y) \to \uPerf$ is the classifying map of $\ev^*\bL_Y$, i.e. the map which sends a morphism $f\colon S\times X\rightarrow Y$ to $f^* \bL_Y\in\Perf(S\times X)$. This observation will allow us to apply the results of the previous section. Applying \cref{thm:pushforwardEulertrivial} we get the following.

\begin{thm}
Let $Y$ be a derived prestack with a perfect cotangent complex. Let $X$ be a derived prestack with an Euler structure. Then $\Map(X, Y)$ carries a canonical volume form and $\dim(\Map(X, Y)) = 0$.
\label{thm:volumeEuler}    
\end{thm}

Often one does not have a full Euler structure, but instead a simple structure. In that case we have the following statement.

\begin{thm}
Let $Y$ be a derived prestack of pure dimension $\dim(Y)$. Let $X$ be a derived prestack equipped with a simple structure. Consider either of the following data:
\begin{enumerate}
\item An equality $\dim(Y) = 0$.
\item An isomorphism $\det(p_\sharp\cO_X)\cong k$.
\item An isomorphism $\det(p_\sharp\cO_X)^{\otimes 2}\cong k$ with $\dim(Y)$ even.
\end{enumerate}
\emph{and} either of the following data:
\begin{enumerate}
\item A volume form on $Y$.
\item A trivialization of the Euler class $e(X)\in\C_\bullet(X(k); \Z)$.
\item A squared volume form on $Y$ as well as a a trivialization of the mod 2 Euler class $e(X)\in\C_\bullet(X(k);\Z/2)$.
\end{enumerate}
Then $\Map(X, Y)$ carries a canonical volume form, the \defterm{torsion volume form}, and
\[\dim(\Map(X, Y)) = \dim(Y) \chi(X).\]
\label{thm:volumelift}
\end{thm}

\begin{proof}
By \cref{prop:mappingcotangent} we can write $\det\bL_{\Map(X,Y)}$ as the map of prestacks
\[
\xymatrix{
\Map(X,Y) \ar[r]^{\bL_Y} & \Map(X, \uPerf^\sim) \ar[r]^-{\pi_\sharp} &\uPerf^\sim \ar[r]^{\det} & \uPic
}
\]
We can extend this to a commuting diagram
\[
\xymatrix{
\Map(X,Y) \ar[r]^-{\bL_Y} & \Map(X, \uPerf^\sim) \ar[r]^-{\pi_\sharp} \ar[d] &\uPerf^\sim \ar[r]^{\det}  \ar[d] & \uPic \\
& \uK(X) \ar[r]^{\pi_\sharp} \ar[d]^{\epsilon} & \uK \ar[r]^{\detgr} & \uPic^\Z \ar[u] \\
& \C^\bullet(X(k) ; \uK) \ar@{=}[r] & \C^\bullet(X(k) ; \uK) \ar[u]_-{\langle - , e_\K(X) \rangle} \ar[r]^\detgr & \C^\bullet(X(k) ; \uPic^\Z) \ar[u]_-{\langle - , \detgr(e_\K(X)) \rangle}
}
\]
where the top two squares commute by definition (of $\pi_\sharp \colon \uK(X) \to \uK$ and $\detgr$), the bottom-left square commutes by \cref{thm:indextheorem} and the bottom-right square commutes by \cref{thm:determinantmap}. 
From naturality of coassembly applied to 
\[\uK(X) \to \Map(X, \uK) \to \Map(X, \uPic^\Z)\]
we obtain the commuting diagram
\[
\xymatrix{
\Map(X, \uPerf^\sim) \ar[r]^{\detgr} \ar[d] & \Map(X,\uPic^\Z)  \ar[dd]^\epsilon \\
\uK(X) \ar[d] & \\
\C^\bullet(X(k) ; \uK) \ar[r]^{\detgr} & \C^\bullet(X(k); \uPic^\Z).
}
\]
We have thus obtained the isomorphism
\[\detgr(\bL_{\Map(X, Y)})\cong \langle \epsilon ( \ev^*\detgr(\bL_Y)), \detgr(e_\K(X))\rangle\in\Pic^{\Z}(\Map(X,Y)),\]
where 
\[\langle -, -\rangle\colon \C^\bullet(X(k); \Pic^\Z(\Map(X,Y))\otimes \C_\bullet(X(k); \Pic^\Z(k))\longrightarrow \Pic^\Z(\Map(X,Y))\]
is given by the natural pairing between chains and cochains on $X(k)$ and the tensor product map
\[\motimes\colon \Pic^\Z(\Map(X,Y))\otimes \Pic^\Z(k)\rightarrow \Pic^\Z(\Map(X,Y)).\]

Under the equivalence $\uPic^\Z \cong \uPic \times \Z$ we decompose the individual determinants as follows:
\begin{itemize}
\item $\detgr(\bL_Y) = (\det(\bL_Y), \dim(Y)) \in \Pic^\Z(Y)$ where we note that $\dim(Y)$ is constant by assumption.
\item $\detgr(e_\K(X)) = (\det(e_\K(X)), e(X))\in\C_\bullet(X(k); \Pic^{\Z}(k))$. Its pushforward under $X(k)\rightarrow \pt$ is $(\det(p_\sharp \cO_X), \chi(X))\in\Pic^{\Z}(k)$.
\end{itemize}
We write (recall that we write addition and multiplication in $\uPic^\Z$ as $\otimes$ and $\otimes^*$, respectively)
\begin{align*}
\detgr(\bL_{\Map(X, Y)}) &\cong \langle \epsilon ( \ev^*\detgr(\bL_Y)), \detgr(e_\K(X))\rangle \\
&= \langle \epsilon ( \ev^*\det(\bL_Y)),  e(X) \rangle \otimes \langle \dim(Y), \detgr(e_\K(X))\rangle\\
&= \langle \epsilon ( \ev^*\det(\bL_Y)),  e(X) \rangle \otimes \left( \dim(Y) \otimes^* p_\sharp\detgr(e_\K(X) \right)\\
&= \left( \langle \epsilon ( \ev^*\det(\bL_Y)),  e(X) \rangle \otimes \det(p_\sharp \cO_X)^{\otimes \dim(Y)}, \dim(Y) \chi(X) \right)
\end{align*}
where by the description of $\otimes^*$ the pairing $\langle -, -\rangle$ used in the last two lines is
\[\C^\bullet(X(k); \Pic(\Map(X,Y))\otimes \C_\bullet(X(k); \Z)\longrightarrow \Pic(\Map(X,Y))\otimes\Z\longrightarrow \Pic(\Map(X,Y)),\]
where the last map is $(\cL, n)\mapsto \cL^{\otimes n}$.

The two tensor factors of $\det(\bL_{\Map(X,Y)})$ are trivialized using the two pieces of data assumed in the statement. In this way we have constructed a canonical trivialization of $\det(\bL_{\Map(X, Y)})$ and hence a volume form on $\Map(X, Y)$.
\end{proof}

\begin{remark}
Rescaling the volume form on $Y$ by $A\in k^*$, the torsion volume form on $\Map(X, Y)$ gets rescaled by $A^{\chi(X)}$. Similarly, rescaling the trivialization of $\det(p_\sharp\cO_X)$ by $B\in k^*$, the torsion volume form on $\Map(X, Y)$ gets rescaled by $B^{\dim(Y)}$.
\end{remark}

Let us now show that the construction of the torsion volume form on mapping stacks is compatible with gluing.

\begin{prop}
Consider a diagram of derived prestacks $X_1\leftarrow X_0\rightarrow X_2$ equipped with simple structures and equip $X=X_1\coprod_{X_0} X_2$ with the glued simple structure. Consider isomorphisms $\det(p_\sharp \cO_{X_i})\cong k$ for $i=0,1,2$ and consider the isomorphism $\det(p_\sharp \cO_X)\cong k$ induced by the pushout square \eqref{diag:pushoutofOX}. Let $Y$ be a derived prestack of pure dimension $\dim(Y)$ and choose a volume form on $Y$. Then the torsion volume form on $\Map(X, Y)$ coincides with the volume form on
\[
\Map(X,Y) = \Map(X_1,Y) \times_{\Map(X_0, Y)} \Map(X_2,Y)
\]
glued from the torsion volume forms on $\Map(X_i, Y)$.
\label{prop:gluedvolonmappingstacks}
\end{prop}
\begin{proof}
We have a pushout diagram
\[
\xymatrix{
X_0 \ar[r]^f \ar[d]_g & X_1 \ar[d]^{g'} \\
X_2 \ar[r]_{f'} & X
}
\]
Let
\[
\xymatrix{
\Map(X, Y) \ar^{\tilde{g}'}[r] \ar^{\tilde{f}'}[d] & \Map(X_1, Y) \ar^{\tilde{f}}[d] \\
\Map(X_2, Y) \ar^{\tilde{g}}[r] & \Map(X_0, Y)
}
\]
be the induced pullback diagram of mapping prestacks.

Let us spell out the canonical volume form obtained from the glued simple structure. We first write
\[
\detgr(\bL_{\Map(X,Y)}) = \det (\pi_\sharp( \ev^*\bL_Y \otimes \cO_X )),
\]
as the image of $[\cO_X]$ under the map $F \colon \K^\omega(X) \to \Pic^\Z(\Map(X,Y))$ induced by the functor
\[\cF\mapsto \detgr (\pi_\sharp( \ev^*\bL_Y \otimes \cF )).\]
From the commutative diagram
\begin{equation}
\xymatrix{
\C_\bullet(X(k), \K(k)) \ar[r] \ar[d]^\alpha & \C_\bullet(X(k), \Pic^\Z(k)) \ar[d]_{ \langle \epsilon(\ev^*\bL_Y), - \rangle} \\
\K^\omega(X) \ar[r]^-F & \Pic^\Z(\Map(X,Y))
}
\label{diag:volformfromassembly}
\end{equation}
we obtained
\[
\detgr(\bL_{\Map(X,Y)}) = F([\cO_X]) \cong F(\alpha e_\K(X)) = \langle \detgr(\ev^*\bL_Y), e_\K(X) \rangle \in \Pic^\Z(\Map(X,Y)).
\]
We then split the projection to $\Pic(\Map(X,Y))$ into two components using the description of the tensor product $\otimes^*$ to obtain the two terms
\[
\langle \det(\ev^*\bL_Y), e(X) \rangle, \quad \langle \dim(Y), \det(p_\sharp \cO_X) \rangle
\]
which we trivialized separately: the first one using the volume form on $Y$ and the second one by using $\det(p_\sharp \cO_X) \cong \det(p_\sharp \cO_{X_1}) \otimes \det(p_\sharp \cO_{X_2}) \otimes \det(p_\sharp \cO_{X_0})^{-1}$.

Note that the diagram \eqref{diag:volformfromassembly} is natural in $X$. Moreover, the map $F$ factorizes as
\[
\K^\omega(X) \longrightarrow \K(\Map(X, Y))\xrightarrow{\detgr} \Pic^\Z(\Map(X,Y)),
\]
so we obtain the commuting diagram
\[
\xymatrix{
F([\cO_X]) \ar@{=}[r]^-{\sim} \ar@{=}[dr]^-{\sim} & F(g'_\sharp[\cO_{X_1}] + f'_\sharp [\cO_{X_2}] - (g'f)_\sharp [\cO_{X_0}]) \ar@{=}[d]^-{\sim} \\
 & \tilde{g}'^*F([\cO_{X_1}]) \otimes \tilde{f}'^* F([\cO_{X_2}]) \otimes (\tilde{g}'\tilde{f})^* F([\cO_{X_0}])^{-1}
}
\]
where the horizontal arrow is given by the additivity homotopy in $\K^\omega(X)$ and the diagonal arrow is the isomorphism
\[
\detgr(\bL_{\Map(X,Y)}) \cong \tilde{g}'^*\detgr(\bL_{\Map(X_1,Y)}) \otimes \tilde{f}'^*\detgr(\bL_{\Map(X_2,Y)}) \otimes (\tilde{g}'\tilde{f})^*\detgr(\bL_{\Map(X_0,Y)})^{-1}.
\]
We thus get that the isomorphism $F([\cO_X]) \cong F(\alpha e_\K(X)) \cong \langle \detgr(\ev^*\bL_Y), e_\K(X) \rangle$ can be factored as
\begin{align*}
F([\cO_X]) &\cong g'^*F([\cO_{X_1}]) \otimes f'^* F([\cO_{X_2}]) \otimes (g'f)^* F([\cO_{X_0}])^{-1} \\
&\cong g'^*F(\alpha(e_\K(X_1))) \otimes f'^* F(\alpha(e_\K(X_2))) \otimes (g'f)^* F(\alpha(e_\K(X_0)))^{-1} \\
&\cong \langle \detgr(\ev^*\bL_Y), e_\K(X_1) - e_\K(X_2) - e_\K(X_0) \rangle
\end{align*}
Finally, the trivialization of the two summands in the $\Pic(\Map(X,Y))$-component of
\[
\langle \detgr(\ev^*\bL_Y), e_\K(X_1) - e_\K(X_2) - e_\K(X_0) \rangle
\]
we used above are equivalent to the trivialization for each $\langle \detgr(\ev^*\bL_Y), e_\K(X_i)\rangle$ separately (clear for the first summand and by definition for the second). This finishes the proof.
\end{proof}

\subsection{Poincar\'e duality for volume forms}
\label{sect:Poincarevolumeform}

Assume $2$ is invertible in $k$ throughout this section. Let $R$ be a connective commutative dg $k$-algebra. Consider the following two Poincar\'e structures on $\Perf(R)$:
\begin{itemize}
    \item $\Qoppa^+ = \Qoppa^{[0]}$ given by $x\mapsto \Hom(x\otimes x, \cO_X)^{C_2}$. We denote $\GW^+(R) = \GW^{[0]}(R)$.
    \item $\Qoppa^-$ given by $x\mapsto (\Hom(x\otimes x, \cO_X)\otimes \sgn)^{C_2}$, where $\sgn$ is the sign representation of $C_2$. The functor $\Perf(R)\rightarrow \Perf(R)$ given by $x\mapsto x[1]$ defines an equivalence of Poincar\'e categories
    \[(\Perf(R), \Qoppa^-)\rightarrow (\Perf(R), \Qoppa^{[2]}).\]
    We denote by $\GW^-(R)$ the Grothendieck--Witt spectrum of this Poincar\'e structure which, using this equivalence, may be identified with $\GW^{[2]}(R)$.
\end{itemize}

The two Poincar\'e structures $\Qoppa^\pm$ define two $C_2$-actions on $\Perf(R)$ (we denote them by $\Perf^\epsilon(R)$ for $\epsilon=\pm 1$), where the underlying duality functor is $x\mapsto x^\vee$. The two are distinguished by the isomorphism $x\rightarrow (x^\vee)^\vee$, where it is either the canonical pivotal element of the symmetric monoidal structure on $\Perf(R)$ (in the case of $\Qoppa^+$) or the canonical pivotal element multiplied by $(-1)$ (in the case of $\Qoppa^-$). Moreover, both Poincar\'e structures restrict to Poincar\'e structures on the heart $\Vect(R)\subset \Perf(R)$ of the weight structure. We have that $(\Perf(R), \Qoppa^+)$ is a symmetric monoidal Poincar\'e $\infty$-category and $(\Perf(R), \Qoppa^-)$ is a $(\Perf(R), \Qoppa^+)$-module $\infty$-category.

\begin{prop}
Let $\epsilon=\pm 1$ be a sign. There is a $C_2$-action on $\uPic^\Z$ as a stack of $\bE_\infty$ spaces which we denote $\uPic^{\Z, \epsilon}$. On the level of underlying stacks it is given by sending $(\cL, n)\mapsto (\cL^*, n)$ with the isomorphism $(\cL, n)\rightarrow (\cL^{**}, n)$ given by the pivotal structure $\cL\cong \cL^{**}$ multiplied by $\epsilon^n$ and with the monoidal structure with respect to $\otimes$
\[\cL_1^*\otimes \cL_2^*\cong (\cL_1\otimes \cL_2)^*\]
determined by the braiding of line bundles with no extra signs. It has the following properties:
\begin{enumerate}
    \item The ring structure
\[\motimes\colon \uPic^{\Z, \epsilon_1}\otimes \uPic^{\Z, \epsilon_2}\longrightarrow \uPic^{\Z, \epsilon_1\epsilon_2}\]
is compatible with involutions via the isomorphism
\[(\cL_1^*)^{n_2}\otimes (\cL_2^*)^{n_1}\cong (\cL_1^{n_2}\otimes \cL_2^{n_1})^*\]
which is again determined by the braiding with no extra signs.
    \item The determinant functor
\[\detgr\colon \uPerf^{\sim, \epsilon}\longrightarrow \uPic^{\Z, \epsilon}\]
is $C_2$-equivariant, where the isomorphism
\[\det(V^*)\cong \det(V)^*\]
is given, for $V\in\Vect(R)$ and $R$ discrete, by the pairing $\det(V^*)\otimes\det(V)\rightarrow \cO$ which sends
\[(\phi_1\wedge \dots \wedge \phi_n, v_1\wedge \dots\wedge v_n)\mapsto \sum_{\sigma\in S_n} \sgn(\sigma) \prod_{i=1}^n \phi_i(v_{\sigma(i)}).\]
    \item The diagram
    \[
    \xymatrix{
    \uPerf^{\sim, \epsilon_1} \otimes \uPerf^{\sim, \epsilon_2} \ar^{\detgr\otimes\detgr}[d] \ar[r] & \uPerf^{\sim, \epsilon_1\epsilon_2} \ar^{\detgr}[d] \\
    \uPic^{\Z, \epsilon_1}\otimes \uPic^{\Z, \epsilon_2} \ar[r] & \uPic^{\Z, \epsilon_1\epsilon_2}
    }
    \]
    is a commutative diagram of $C_2$-equivariant stacks.
\end{enumerate}
\label{prop:dualitydeterminant}
\end{prop}
\begin{proof}
As in the proof of \cref{thm:determinantmap}, it is enough to construct a $C_2$-action on $\Pic^\Z(R)$ for $R$ discrete as well as show that the determinant functor
\[\detgr\colon \Vect^\epsilon(R)^\sim\longrightarrow \Pic^{\Z, \epsilon}(R)\]
is $C_2$-equivariant, again for $R$ discrete, which is a straightforward check.
\end{proof}

We can describe the invariant categories as follows:
\begin{itemize}
    \item $\Perf^+(R)^{C_2}$ is the $\infty$-category of perfect complexes equipped with a nondegenerate symmetric bilinear pairing.
    \item $\Perf^-(R)^{C_2}$ is the $\infty$-category of perfect complexes equipped with a symplectic pairing.
    \item $\Pic^{\Z, +}(R)^{C_2}$ is the $\infty$-category of graded line bundles $(\cL, n)$ equipped with an isomorphism $\cL\cong \cL^*$; equivalently, a trivialization of $\cL^{\otimes 2}$.
    \item $\Pic^{\Z, -}(R)^{C_2}$ is the $\infty$-category of graded line bundles $(\cL, 2n)$ equipped with a trivialization of $\cL^{\otimes 2}$.
\end{itemize}

By the above statement the determinant functor $\detgr\colon \uPerf\rightarrow \uPic^\Z$ descends to a morphism of spectra
\[\detgr\colon \uK^{\pm, C_2}\longrightarrow \uPic^{\Z, \pm, C_2}\]
and hence, by precomposition with the forgetful map $\uGW^\pm\rightarrow \uK^{\pm, C_2}$, to
\[\detgr\colon \uGW^\pm\longrightarrow \uPic^{\Z, \pm, C_2}.\]

Let $\mu_2$ be the algebraic group (defined over $k$) of second roots of unity. An element of $\mu_2(R)$ is then the same as an element $g\in\Gm(R)$ such that $g^2 = 1$.

\begin{prop}
Consider the determinant map of prestacks
\[\detgr\colon \tau_{[0, 1]}\uGW^+\longrightarrow \uPic^{\Z, +, C_2},\]
where $\tau_{[0, 1]}(-)$ denotes the 1-truncation of the connective cover. This map induces an isomorphism after \'etale sheafification.
\label{prop:GW1type}
\end{prop}
\begin{proof}
We have
\[\pi_0(\uPic^{\Z, +, C_2}) \cong \Z,\qquad \pi_1(\uPic^{\Z, +, C_2}) \cong \mu_2\]
as \'etale sheaves.

\'Etale locally every quadratic form admits an orthonormal basis which shows that $\rk\colon \GW^+_0(R)\rightarrow \Z(R)$ becomes an isomorphism after \'etale sheafification.

There is a homomorphism
\[\GW^+_1(R)\longrightarrow \Gm(R)/2\oplus \mu_2(R)\]
given by the spinor norm and the determinant which is shown to be an isomorphism Zariski locally in \cite[Corollary 4.7.7]{Bass}. But the multiplication by 2 map $\Gm\rightarrow \Gm$ is \'etale locally surjective.
\end{proof}

Recall the notion of a shifted symplectic structure on derived Artin stacks from \cite{PTVV}. Given an $n$-shifted symplectic structure $\omega_Y$ on a derived Artin stack $Y$ the perfect complex $\bT_Y[1]\cong \bL_Y[n+1]$ equipped with the pairing coming from the symplectic structure defines a map $Y\rightarrow \Omega^\infty\uGW^{[n+2]}$. Assume that $n=2k$ is even. Then $\bT_Y[-k]$ defines a map $Y\rightarrow \Omega^\infty\uGW^\epsilon$, where the sign is $\epsilon = (-1)^{k+1}$. In particular, by the above we obtain a squared volume form on $Y$ and we denote the corresponding pairing $\det(\bL_Y)\otimes \det(\bL_Y)\rightarrow \cO_Y$ by $(-, -)_{\omega_Y}$.

Let $X$ be a derived prestack with a fundamental class $[X]$ of degree $d$, where $d$ is even. The pushforward along $p\colon X\rightarrow \pt$ gives a map
\[p_\sharp\colon \GW^{[0]}(X)\longrightarrow \GW^{[d]}(k).\]
The object $\cO_X\in\Perf(X)$ has an obvious nondegenerate symmetric bilinear pairing. So, it defines a point $[\cO_X]\in\Omega^\infty\GW^{[0]}(X)$ and, hence, $[p_\sharp \cO_X]\in\Omega^\infty\GW^{[d]}(k)$. Since $d$ is even, we obtain a canonical pairing $\det(p_\sharp \cO_X)\otimes \det(p_\sharp \cO_X)\rightarrow k$ which we denote by $(-, -)_{[X]}$.

By \cref{prop:fundamentalclassorientation} the fundamental class $[X]$ on $X$ gives rise to an $\cO$-orientation of $X$ of degree $d$. Thus, by the AKSZ construction \cite[Theorem 2.5]{PTVV} we obtain an $(n-d)$-shifted symplectic structure $\omega_{\Map}$ on $\Map(X, Y)$. As both $n$ and $d$ are assumed to be even, $(n-d)$ is even as well, so there is a natural squared volume form on $\Map(X, Y)$ and we denote the corresponding pairing on $\det(\bL_{\Map(X, Y)})$ by $(-, -)_{\omega_{\Map(X, Y)}}$.

\begin{thm}
Let $X$ be a derived prestack with a fundamental class $[X]$ of even degree $d$ and a simple structure compatible with Poincar\'e duality. Let $Y$ be an $n$-shifted symplectic stack, where $n$ is even. Let $o$ be an isomorphism $\det(p_\sharp \cO_X)\cong k$ and $\vol_Y$ a volume form on $Y$. Let $\vol_{\Map(X, Y)}$ be the torsion volume form constructed in \cref{thm:volumelift} from this data. Then
\[(\vol_{\Map(X, Y)}, \vol_{\Map(X, Y)})_{\omega_{\Map(X, Y)}} = \langle \epsilon(\ev^* (\vol_Y, \vol_Y)_{\omega_Y}), e(X)\rangle ((o, o)_{[X]})^{\dim Y}.\]
\label{thm:volumeliftduality}
\end{thm}
\begin{proof}
Let $S$ be a derived affine scheme and consider a morphism $f\colon S\rightarrow \Map(X, Y)$ corresponding to $\tilde{f}\colon S\times X\rightarrow Y$. Using the $n$-shifted symplectic structure $\omega_Y$ on $Y$ we have that $[\bL_Y[n+1]]\in \Omega^\infty\GW^{[n+2]}(Y)$. Similarly, using the $(n-d)$-shifted symplectic structure $\omega_{\Map(X, Y)}$ obtained using the AKSZ construction we have that $[\bL_{\Map(X, Y)}[n-d+1]]\in\Omega^\infty\GW^{[n-d+2]}(\Map(X, Y))$. Therefore, \[[\bL_{\Map(X, Y)}[n+1]]\in\Omega^\infty\GW^{[n+d+2]}(\Map(X, Y)).\]
Examining the AKSZ construction, the isomorphism
\[\bL_{\Map(X, Y)}[n+1]\cong \pi_\sharp \ev^* \bL_Y[n+1]\]
from \cref{prop:mappingcotangent} is compatible with pairings, so
\[[\bL_{\Map(X, Y)}[n+1]] = \pi_\sharp [\ev^* \bL_Y[n+1]]\in\Omega^\infty\GW^{[n+d+1]}.\]
By \cref{thm:indextheoremwithPoincare} we have
\[[f^* \bL_{\Map(X, Y)}[n+1]]\sim \langle \epsilon(\tilde{f}^*\bL_Y[n+1]), e_\GW(X)\rangle\in\Omega^\infty\GW(S).\]

Using \cref{prop:dualitydeterminant} we get that an isomorphism
\[\det(f^*\bL_{\Map(X, Y)}[n+1])\cong \langle \det(\tilde{f}^*\bL_Y[n+1]), e(X)\rangle\otimes \det(p_\sharp\cO_X)^{\otimes (-\dim(Y))}\]
of lines equipped with nondegenerate pairings. In \cref{thm:volumelift} the volume form $\vol_{\Map(X, Y)}$ is constructed using the trivialization $\vol_Y$ of $\det(\tilde{f}^*\bL_Y[n+1])$ and $o$ of $\det(p_\sharp \cO_X)$. Compatibility with the nondegenerate pairings implies the result.
\end{proof}

We will also need a slight variant of the above statement. Another setting where \cref{thm:volumelift} can be applied is if $Y$ merely has a squared volume form $\vol^2_Y\in\det(\bL_Y)^{\otimes 2}$ and $X$ has a trivialization of the mod 2 Euler class $e(X)\in\C_\bullet(X; \Z/2)$. In fact, we can take as $\vol^2_Y$ the canonical squared volume form on $Y$ induced by the even shifted symplectic structure $\omega_Y$.

\begin{thm}
Let $X$ be a derived prestack with a fundamental class $[X]$ of even degree $d$ and a simple structure compatible with Poincar\'e duality. Let $Y$ be an $n$-shifted symplectic stack, where $n$ is even. Let $o$ be an isomorphism $\det(p_\sharp \cO_X)\cong k$ and choose a trivialization of the mod 2 Euler class $e(X)\in\C_\bullet(X; \Z/2)$. Let $\vol_{\Map(X, Y)}$ be the torsion volume form constructed in \cref{thm:volumelift} from this data. Then
\[(\vol_{\Map(X, Y)}, \vol_{\Map(X, Y)})_{\omega_{\Map(X, Y)}} = ((o, o)_{[X]})^{\dim Y}.\]
\label{thm:volumeliftdualitysquared}
\end{thm}

\subsection{Symplectic volume forms}
\label{sect:symplecticvolume}

In this section we continue assuming that $2$ is invertible in $k$. We have previously shown that if a derived stack $Y$ has an $n$-shifted symplectic structure for $n$ even, there is a natural squared volume form on $Y$. We will now refine the result when $n$ is divisible by $4$ by constructing an actual volume form. In this case $\bT_Y[-n/2]\cong \bL_Y[n/2]$ defines a class in $\Omega^\infty\GW^-(Y)$.

Besides the stack $\uPic^\Z$ of $\Z$-graded lines we may also consider the stack $\uPic^{\Z/2}$ of $\Z/2$-graded lines. There is a natural forgetful map $\uPic^\Z\rightarrow \uPic^{\Z/2}$.

\begin{thm}
There is a commutative diagram
\[
\xymatrix{
\tau_{\geq 0} \uGW^- \ar^-{0}[r] \ar[d] & \uPic^{\Z/2, -, C_2} \\
\uGW^- \ar^-{\detgr}[r] & \uPic^{\Z, -, C_2} \ar[u]
}
\]
\label{thm:symplecticvolumeform}
\end{thm}
\begin{proof}
The map $(\cL, n)\mapsto (\cL, n/2)$ defines an isomorphism $\uPic^{\Z, -, C_2}\cong \uPic^{C_2}\times \Z$ of $E_\infty$ stacks, where $\uPic^{C_2}$ is the stack of lines equipped with a nondegenerate pairing, which can be identified with $\B\mu_2$. So, we have to construct a nullhomotopy of the composite
\[\tau_{\geq 0} \uGW^-\longrightarrow \uGW^-\longrightarrow \B\mu_2.\]

By the topological invariance of the \'etale site (see \cite[Corollary 2.2.2.9]{HAGII}) it is enough to construct this nullhomotopy for discrete commutative $k$-algebras $R$. By \cite[Theorem A]{HebestreitSteimle} the connective spectrum $\tau_{\geq 0} \GW^-(R)$ is the group completion of the monoid $\Pn(R, -)$ of finitely generated projective $R$-modules $M$ equipped with a symplectic structure $\omega\in\wedge^2 M^*$. As $\B\mu_2$ is a group, this implies that we have to construct a nullhomotopy of the functor of symmetric monoidal groupoids
\[\det\colon \Pn(R, -)\longrightarrow (\B\mu_2)(R)\]
obtained by sending $M$ to the line bundle $\det(M)$ equipped with an isomorphism $\det(M)^{\otimes 2}\cong \cO$ using the symplectic structure.

We send $(M, \omega)$ to the section $\vol_M = \gamma_n(\omega)$ of $\det(M^*)$, where $n=\rk(M)$ and $\gamma_n$ is the $n$-th divided power. This construction is clearly functorial in $R$, so to check that it defines a nullhomotopy of $\Pn(R, -)\rightarrow (\B\mu_2)(R)$ we have to check the following:
\begin{itemize}
    \item The nullhomotopy has to be compatible with the symmetric monoidal structure. This boils down to the fact that under the isomorphism $\det(M_1^*)\otimes \det(M_2^*)\rightarrow \det(M_1^*\oplus M_2^*)$ the section $\vol_{M_1}\otimes \vol_{M_2}$ goes to $\vol_{M_1\oplus M_2}$. Indeed, if the ranks of $M_1$ and $M_2$ are $n_1$ and $n_2$, the symplectic structure on $M = M_1\oplus M_2$ is
    \[\omega_M = \omega_{M_1} + \omega_{M_2}.\]
    So,
    \[\gamma_{n_1+n_2}(\omega_M) = \gamma_{n_1}(\omega_{M_1})\wedge \gamma_{n_2}(\omega_{M_2})\]
    which implies the result.
    \item The volume form $\vol_M$ has to square to the canonical trivialization of $\det(M^*)^{\otimes 2}$. This can be checked Zariski locally on $R$, so that we may assume that $M$ admits a symplectic basis $\{e_1,f_1,e_2,f_2,\dots\}$. Let $\{e^1,f^1,\dots, e^n, f^n\}$ be the dual basis of $M^*$, so that the volume form is
    \[\vol_M = e^1\wedge f^1\wedge \dots \wedge e^n\wedge f^n\in\det(M^*).\]
    Under the isomorphism $\det(M^*)\cong \det(M)^*$ given by \cref{prop:dualitydeterminant} it corresponds to
    \[(e_1\wedge f_1\wedge \dots\wedge e_n\wedge f_n)^{-1} \in\det(M)^*.\]
    The isomorphism $\omega^\sharp\colon M\rightarrow M^*$ sends $e_i\mapsto f^i$ and $f_i\mapsto -e^i$. So, it sends
    \[e_1\wedge f_1\wedge \dots\wedge e_n\wedge f_n\mapsto e^1\wedge f^1\wedge \dots\wedge e^n\wedge f^n\]
    which proves the claim.
\end{itemize}
\end{proof}

Concretely, the above statement shows that there is a \defterm{symplectic volume form} on any $n$-shifted symplectic stack $Y$ with $n$ divisible by $4$; moreover, this symplectic volume form squares to the trivialization of $\det(\bL_Y)^{\otimes 2}$ constructed by taking the determinant of the isomorphism $\omega^\sharp\colon \bT_Y\rightarrow \bL_Y[n]$.

Let us now describe a compatibility of the symplectic volume form with respect to the tensor product. The tensor product gives a functor of Poincar\'e $\infty$-categories
\[(\Perf(R), \Qoppa^+)\otimes (\Perf(R), \Qoppa^-)\longrightarrow (\Perf(R), \Qoppa^-).\]
It gives a tensor product map on the Grothendieck--Witt spectra
\[\uGW^+\otimes \uGW^-\longrightarrow \uGW^-\]
and their connective covers
\[\tau_{\geq 0} \uGW^+\otimes \tau_{\geq 0} \uGW^-\longrightarrow \tau_{\geq 0} \uGW^-.\]

The multiplication map
\[\motimes\colon \uPic^{\Z, \epsilon_1, C_2}\otimes \uPic^{\Z, \epsilon_2, C_2}\longrightarrow \uPic^{\Z, \epsilon_1\epsilon_2, C_2}\]
descends to a multiplication map
\[\motimes\colon \uPic^{\Z, \epsilon_1, C_2}\otimes \uPic^{\Z/2, \epsilon_2, C_2}\longrightarrow \uPic^{\Z/2, \epsilon_1\epsilon_2, C_2},\]
where we use that $(\cL_1, n_1)\motimes (\cL_2, n_2) = (\cL_1^{n_2}\otimes \cL_2^{n_1}, n_1n_2)$ depends on $n_2$ only modulo 2 as $\cL_1^{\otimes 2}$ is canonically trivial. Therefore, from \cref{prop:dualitydeterminant} we obtain a commutative diagram
\[
\xymatrix{
\tau_{\geq 0} \uGW^+\otimes \tau_{\geq 0} \uGW^- \ar^{\detgr\otimes\detgr}[d] \ar[r] & \tau_{\geq 0} \uGW^- \ar^{\detgr}[d] \\
\uPic^{\Z, +, C_2}\otimes \uPic^{\Z/2, -, C_2} \ar[r] & \uPic^{\Z/2, -, C_2}
}
\]

\begin{prop}
The diagram
\[
\xymatrix{
\tau_{\geq 0} \uGW^+\otimes \tau_{\geq 0} \uGW^- \ar^{\detgr\otimes\detgr}[d] \ar[r] & \tau_{\geq 0} \uGW^- \ar^{\detgr}[d] \\
\uPic^{\Z, +, C_2}\otimes \uPic^{\Z/2, -, C_2} \ar[r] & \uPic^{\Z/2, -, C_2}
}
\]
is compatible with the nullhomotopies of $\detgr\colon \tau_{\geq 0}\uGW^-\rightarrow \uPic^{\Z/2, -, C_2}$ given by the symplectic volume form.
\label{prop:symplectictimesorthogonal}
\end{prop}
\begin{proof}
The difference of the two nullhomotopies defines a map
\[\GW^+_0(R)\otimes_\Z \GW^-_0(R)\longrightarrow \mu_2(R) = \pi_1(\Pic^{\Z/2, -, C_2}(R)).\]

We can check that this map is trivial as in \cref{thm:symplecticvolumeform}. By the homotopy invariance of the \'etale site it is enough to prove the claim for $R$ discrete. Then by \cite[Theorem A]{HebestreitSteimle} $\GW^+_0(R)$ is the group completion of the monoid $\Pn_0(R, +)$ of projective finitely generated $R$-modules equipped with a nondegenerate symmetric bilinear pairing and $\GW^-_0(R)$ is the group completion of the monoid $\Pn_0(R, -)$ of projective finitely generated $R$-modules equipped with a symplectic pairing. As $\mu_2(R)$ is a group, we need to check that the map
\[\Pn_0(R, +)\otimes_\bN \Pn_0(R, -)\longrightarrow \mu_2(R)\]
is trivial, where $\otimes_\bN$ denotes the tensor product of commutative monoids. Concretely, suppose $V$ is a projective finitely generated $R$-module equipped with a nondegenerate symmetric bilinear pairing and $W$ is a projective finitely generated $R$-module equipped with a symplectic pairing. Then $V\otimes W$ carries a natural symplectic pairing and we have to show that under the isomorphism
\[(\det(V)^{\otimes 2})^{\dim W/2}\otimes \det(W)^{\dim V}=\det(V)^{\dim W}\otimes \det(W)^{\dim V}\longrightarrow \det(V\otimes W)\]
the element $(\vol^2_V)^{\dim W/2}\otimes \vol_W^{\dim V}$ is sent to $\vol_{V\otimes W}$, where $\vol^2_V$ is the natural trivialization of $\det(V)^{\otimes 2}$ obtained from the nondegenerate pairing on $V$ and $\vol_W$ is the natural trivialization of $\det(W)$ given by the symplectic volume form.

To check this we may work \'etale locally on $R$, so we may assume that $V$ has an orthonormal basis $\{v_1, \dots, v_n\}$ and $W$ has a symplectic basis $\{e_1, f_1, \dots, e_m, f_m\}$. Then $V\otimes W$ has a symplectic basis $\{v_i\otimes e_j, v_i\otimes f_j\}_{i=1\dots n, j=1\dots m}$. We denote the dual bases by the same letters with upper indices. The isomorphism $V\rightarrow V^*$ given by the nondegenerate pairing on $V$ sends $v_i\mapsto v^i$. Therefore, $\det(V)\rightarrow \det(V^*)$ is given by $v_1\wedge \dots v_n\mapsto v^1\wedge \dots\wedge v^n$ and hence
\[\vol^2_V = (v_1\wedge \dots \wedge v_n)^2.\]
By definition the symplectic volume form on $W$ is
\[\vol_W = e_1\wedge f_1\wedge \dots\wedge e_n\wedge f_n.\]
Thus, the element of $\det(V)^{\dim W}\otimes \det(W)^{\dim V}$ is
\[(v_1\wedge \dots \wedge v_n)^{\dim W}\otimes (e_1\wedge f_1\wedge \dots\wedge e_n\wedge f_n)^{\dim V}.\]
By \cref{prop:determinantclassical} it is sent to
\[\wedge_{i=1}^n\wedge_{j=1}^m (v_i\otimes e_j)\wedge (v_i\otimes f_j)\in \det(V\otimes W)\]
which is exactly the symplectic volume form $\vol_{V\otimes W}$.
\end{proof}

\section{Betti setting}

In this section we describe the results of \cref{sect:determinant} for constant stacks.

\subsection{Finiteness}

Let $M\in\cS$ be a space and consider the constant derived stack $X=M_\B$ with value $M$. Then
\[\QCoh(M_\B)\cong\Fun(M, \Mod_k)=:\LocSys(M)\]
is the $\infty$-category of (infinite rank) local systems on $M$. The subcategory $\Perf(M_\B)\subset \QCoh(M_\B)$ is the full subcategory of local systems whose fibers are perfect complexes. If $M$ is connected with a chosen basepoint, we may further identify
\[\LocSys(M)\cong \Mod_{\C_\bullet(\Omega M; k)},\]
the $\infty$-category of modules over chains on the based loop space. In particular,
\[\K^\omega(M_\B)\cong \K(\C_\bullet(\Omega M; k)).\]
We refer to \cite{HaugsengBG} for details on the functoriality of the $\infty$-category of local systems. For any map $f\colon M_1\rightarrow M_2$ there is a pullback $f^*\colon \LocSys(M_2)\rightarrow \LocSys(M_1)$ given by restriction which admits a left adjoint $f_\sharp\colon \LocSys(M_1)\rightarrow \LocSys(M_2)$ given by the left Kan extension which satisfies the projection formula.

\begin{remark}
The assembly and coassembly maps in this context coincide with those defined in \cite{WilliamsRR}.
\end{remark}

Recall that $M$ is \defterm{finitely dominated} if it is a retract of a finite CW complex in the homotopy category of spaces. Equivalently, it is a compact object of $\cS$ (see \cite[Remark 5.4.1.6]{LurieHTT}).

\begin{prop}
Suppose $M$ is a finitely dominated space. Then $M_\B$ satisfies \cref{mainassumptionduality}.
\label{prop:Bettiassumptions}
\end{prop}
\begin{proof}
The functors $p^*, i^*, \Delta^*$ admit left adjoints satisfying the projection formula. By \cite[Lemma 4.8]{HaugsengBG} the constant local system $k_M\in\LocSys(M)$ is compact, i.e. $p_*$ is colimit-preserving. Finally, by \cite[Lemma 4.3.8]{HopkinsLurie} $\LocSys(M)$ is compactly generated.
\end{proof}

For a finitely dominated space $M$ the Euler characteristic $\chi(M)$ is well-defined and it coincides with the Euler characteristic of $M_\B$.

\subsection{Lifts along the assembly map}
\label{sect:Bettilifts}

Let $M$ be a finitely dominated space. Then the structure sheaf $\cO_{M_\B}\in\QCoh(M_\B)$ is compact and hence defines a class $[\cO_{M_\B}]\in\Omega^\infty\K^\omega(M_\B)$. We will be interested in lifts of $[\cO_{M_\B}]$ along the assembly map
\[\C_\bullet(M; \K(k))\longrightarrow \K^\omega(M_\B).\]

To describe the known results, let us temporarily switch from working over the commutative ring $k$ to working over the sphere spectrum. Let
\[\Sp^M = \Fun(M, \Sp)\]
be the $\infty$-category of parametrized spectra, so that
\[\LocSys(M) = \Sp^M\otimes_{\Sp} \Mod_k.\]
Consider the $A$-theory
\[\A(M) = \K(\Sp^{M, \omega})\]
defined to be the $K$-theory of the stable $\infty$-category of compact parametrized spectra. In this case the assembly map becomes
\[\C_\bullet(M; \A(\pt))\longrightarrow \A(M)\]
and we want to lift the class $[\bS_M]\in \Omega^\infty\A(M)$ of the constant parametrized spectrum $\bS_M\in\Sp^M$ with value the sphere spectrum. Let us recall some known results:
\begin{itemize}
\item Such a lift exists if, and only if, $M$ is homotopy equivalent to a finite CW complex \cite{WallFiniteness}.
\item If $M$ is a compact ENR, a canonical such lift was constructed in \cite[Section 8]{DwyerWeissWilliams} using controlled algebra.
\end{itemize}

The base change to $k$ defines a commutative diagram of assembly maps
\[
\xymatrix{
\C_\bullet(M; \A(\pt))\ar[r] \ar[d] & \A(M) \ar[d] \\
\C_\bullet(M; \K(k))\ar[r] & \K^\omega(M_\B)
}
\]
so any lift of $[\bS_M]$ over the sphere spectrum gives rise to a simple structure on $M_\B$. Let us now describe an explicit model of such a simple structure for $k=\Z$ if $M$ is a finite CW complex following \cite{TuraevEuler}. For simplicity we assume that $M$ is connected with a basepoint $x_0\in M$.

\begin{prop}
The 1-truncation of the assembly map
\[\tau_{\leq 1} \C_\bullet(M; \K(\Z))\xrightarrow{\alpha} \tau_{\leq 1} \K(\C_\bullet(\Omega M;\Z))\]
is equivalent to the map
\[\B \rH_1(M; \Z)\times \tau_{\leq 1} \K(\Z)\longrightarrow \tau_{\leq 1} \K(\Z[\pi_1(M)]),\]
where $\B \rH_1(M; \Z)$ is the one-object groupoid with automorphisms given by $\rH_1(M; \Z)$, which sends the class $[V]\in\Omega^\infty\K(\Z)$ of the abelian group $V$ to the class $[\Z[\pi_1(M)]\otimes V]\in\K(\Z[\pi_1(M)])$ of the free $\Z[\pi_1(M)]$-module generated by $V$. Under these identifications the class $[\Z]\in \Omega^\infty\K(\C_\bullet(\Omega M;\Z))$ of the trivial $\C_\bullet(\Omega M;\Z)$-module goes to the class $[\C_\bullet(\tilde{M};\Z)]\in \Omega^\infty\K(\Z[\pi_1(M)])$ of the $\Z[\pi_1(M)]$-module given by chains on the universal cover $\tilde{M}\rightarrow M$.
\label{prop:assemblytruncation}
\end{prop}
\begin{proof}
The natural map $\C_\bullet(\Omega M; \Z)\rightarrow \Z[\pi_1(M)]$ is 1-connected, so by \cite[Proposition 1.1]{Waldhausen} the induced map
\[\K(\C_\bullet(\Omega M; \Z))\longrightarrow \K(\Z[\pi_1(M)])\]
is 2-connected, i.e. it induces an equivalence on $\tau_{\leq 1}$. 
The Cartesian diagram
\[
\xymatrix{
\tilde{M} \ar[r] \ar[d] & \pt \ar[d] \\
M \ar[r] & \B \pi_1(M)
}
\]
shows that base changing along $\C_\bullet(\Omega M; \Z)\rightarrow \Z[\pi_1(M)]$ identifies
\[
\Z\otimes_{\C_\bullet(\Omega M; \Z)}\Z[\pi_1(M)]\cong \C_\bullet(\tilde{M};\Z).
\]
Finally we identify the $1$-truncation of $\C_\bullet(M; \K(\Z))$ as
\[\tau_{\leq 1} \C_\bullet(M; \K(\Z))\cong \B \rH_1(M; \Z)\times \tau_{\leq 1} \K(\Z).\]
\end{proof}

Let $A$ be the set of cells of $M$ and $A_n\subset A$ the set of $n$-dimensional cells. Let $\tilde{M}\rightarrow M$ be the universal cover. We can canonically lift the cell structure $A$ of $M$ to a $\pi_1(M)$-equivariant cell structure $\tilde{A}$ on $\tilde{M}$.

\begin{defn}
A \defterm{fundamental family of cells} $e$ in $\tilde{M}$ is the choice of a lift of a cell $a\in A$ on $M$ to a cell $\tilde{a}\in\tilde{A}$ on $\tilde{M}$.
\end{defn}

Given two fundamental families of cells $e, e'$ in $\tilde{M}$, for every $a\in A$ there is a unique $h_a\in\pi_1(M)$ such that $\tilde{a}' = h_a \tilde{a}$. Let $q\colon \pi_1(M)\rightarrow \rH_1(M;\Z)$ be the abelianization map and
\[e' / e = \sum_{a\in A} (-1)^{\dim(a)} q(h_a).\]

\begin{defn}
Two fundamental family of cells $e, e'$ in $\tilde{M}$ are \defterm{equivalent} if $e'/e = 1\in\rH_1(M;\Z)$. Denote by $E(M)$ the set of equivalence classes of fundamental families of cells.
\end{defn}

By definition $E(M)$ is a nonempty $\rH_1(M;\Z)$-torsor.

\begin{prop}
Suppose $M$ is a connected finite CW complex. Then there is a canonical lift of \[[\Z]\in\Omega^\infty\K(\C_\bullet(\Omega M;\Z)),\]
the class of the trivial $\C_\bullet(\Omega M; \Z)$-module, along the assembly map $\alpha\colon \C_\bullet(M; \K(\Z))\rightarrow \K(\C_\bullet(\Omega M; \Z))$ to an element \[e_\K(M)\in\Omega^\infty\C_\bullet(M; \K(\Z)).\]
In other words, $M_\B$ has a canonical simple structure.
\label{prop:CWlift}
\end{prop}
\begin{proof}
To describe a canonical lift of $[\Z]\in\Omega^\infty\K(\C_\bullet(\Omega M; \Z))$ it is enough to lift the class of chains on the universal cover $[\C_\bullet(\tilde{M};\Z)]\in\Omega^\infty\K(\Z[\pi_1(M)])$ along the 1-truncation $\tau_{\leq 1}$ of the assembly map described in \cref{prop:assemblytruncation}. Using cellular chains of $\tilde{M}$ we obtain that $\C_k(\tilde{M}; \Z)$ is a finitely generated free $\pi_1$-module. The choice of a fundamental family of cells $e$ defines a map of abelian groups $\C_k(M;\Z) \to \C_k(\tilde{M}; \Z)$ inducing an isomorphism of $\pi_1$-modules
\[
\Z[\pi_1(M)] \otimes_\Z \C_k(M;\Z) \cong \C_k(\tilde{M};\Z).
\]
Using \cref{prop:chaincomplexfiltration} we now obtain the homotopy
\[
[\C_\bullet(\tilde{M};\Z)]\sim \sum_{k=0} (-1)^k [\C_k(\tilde{M};\Z)]\sim \sum_{k=0} (-1)^k \left[\Z[\pi_1(M)]\otimes_\Z \C_k(M;\Z)\right],
\]
where the last term is in the image of the assembly map.



Given two fundamental family of cells $e, e'$, the corresponding lifts differ by $e'/e\in\rH_1(M;\Z)$ and hence are homotopic.
\end{proof}

\begin{remark}
More explicitly, we have constructed a lift of $\Z$ along the assembly map to
\[e_\K(M) = \sum_k (-1)^k [\Z[A_k]] x_0\]
depending on a fundamental family of cells in $\tilde{M}$. Changing the fundamental family of cells by a class $\gamma$ in $\rH_1(M;\Z)$ changes $e_\K(M)$ by an automorphism given by $\gamma$.
\end{remark}

\begin{remark}
Instead of choosing a single basepoint $x_0\in M$ it is often useful to choose a point $\alpha_a\in a$ for every cell $a\in A$. Then one can analogously identify
\[e_\K(M) = \sum_{a\in A} (-1)^{\dim(a)} [\Z]\alpha_a.\]
\end{remark}

\begin{remark}
It is shown in \cite{TuraevEuler} that the lift constructed in \cref{prop:CWlift} is invariant under cell subdivisions of $M$.
\end{remark}

Let us now show that the simple structure constructed in \cref{prop:CWlift} is compatible with gluing.

\begin{prop}
Suppose $A,B,C$ are finite CW complexes with $A\subset B$ and $A\subset C$ a subcomplex. Consider the pushout
\[\xymatrix{A \ar[r]^{f} \ar[d]_{g} & B \ar[d]^{g'} \\
C \ar[r]_{f'} & M
}\]
which endows $M$ with the structure of a finite CW complex. Then the simple structure on $M_\B$ constructed in \cref{prop:CWlift} is obtained by gluing the simple structures on $B_\B$ and $C_\B$ along $A_\B$ in the sense of \cref{prop:gluedsimple}.
\label{prop:gluedCWsimplestructures}
\end{prop}
\begin{proof}
As in the proof of \cref{prop:CWlift} it suffices to work in $\K(\Z[\pi_1(M)])$. Let $\tilde{M}\rightarrow M$ be the universal cover and set $\tilde{A} = A \times_M \tilde{M}$, $\tilde{B} = B \times_M \tilde{M}$ and $\tilde{C} = C \times_M \tilde{M}$. The glued lift of $[\C_\bullet(\tilde{M};\Z)]$ is obtained as follows.
\begin{align*}
[\C_\bullet(\tilde{M};\Z)] &\sim [\C_\bullet(\tilde{B};\Z)] + [\C_\bullet(\tilde{C};\Z)] - [\C_\bullet(\tilde{A};\Z)] \\
&\sim [\bigoplus_k \C_k(\tilde{B};\Z)] + [\bigoplus_k \C_k(\tilde{C};\Z)] - [\bigoplus_k \C_k(\tilde{A};\Z)] \\
&\sim \alpha(e_\K(B) + e_\K(C) - e_\K(A)),
\end{align*}
where the first two homotopies are given by additivity of $K$-theory (on the cellular chain complex). More precisely, the first homotopy is induced by the fiber sequence
\[
\C_\bullet(\tilde{A};\Z) \longrightarrow \C_\bullet(\tilde{B};\Z) \oplus \C_\bullet(\tilde{C};\Z) \longrightarrow  \C_\bullet(\tilde{M};\Z).
\]
The second homotopy is induced by \cref{prop:chaincomplexfiltration}, that is by filtering the three chain complexes $\C_\bullet(\tilde{A};\Z)$, $\C_\bullet(\tilde{B};\Z)$ and $\C_\bullet(\tilde{C};\Z)$ by the brutal truncation and finally by the choice of fundamental families of cells.

We conclude that the above homotopy is obtained by the construction in \cref{prop:CWlift} for the natural cell structure on $B \cup_A (A \times I) \cup_A C$ induced from the cell structures on $A$, $B$ and $C$, and the induced fundamental family of cells. Finally, note that the natural map $B \cup_A (A \times I) \cup_A C \to M$ is a simple homotopy equivalence.
\end{proof}

Next, we will show that one can endow $M$ with further extra structure to trivialize $[\Z]\in\Omega^\infty\K(\C_\bullet(\Omega M;\Z))$. By \cref{prop:turaevtrivialization} we need to trivialize the lift of $[\Z]$ to $e_\K(M)\in\Omega^\infty\C_\bullet(M;\K(\Z))$. Consider the map
\[\C_\bullet(M;\K(\Z))\longrightarrow \C_\bullet(M;\Z)\times \K(\Z),\]
where the first map is induced by the degree map $\K(\Z)\rightarrow \Z$ and the second map is induced by the pushforward along $M\rightarrow \pt$. As $M$ is connected and $\K(\Z)\rightarrow \Z$ is an isomorphism on $\pi_0$, this map is an isomorphism on 1-truncations. Therefore, to trivialize $e_\K(M)\in\Omega^\infty\C_\bullet(M; \K(\Z))$, we need to provide the following information:
\begin{itemize}
\item Trivialization of the image of $e_\K(M)$ under the degree map $\K(\Z)\rightarrow \Z$, i.e. the homological Euler class $e(M)\in\C_\bullet(M;\Z)$.
\item Trivialization of the image of $e_\K(M)$ under the pushforward $\C_\bullet(M;\K(\Z))\rightarrow \K(\Z)$, i.e. the class $[\C_\bullet(M;\Z)]\in\Omega^\infty\K(\Z)$.
\end{itemize}

Recall the following notion from \cite{TuraevEuler}.

\begin{defn}
An \defterm{Euler structure on $M$} is a singular 1-chain $\xi$ with integer coefficients with
\[\partial \xi = \sum_{a\in A} (-1)^{\dim(a)} \alpha_a,\]
where $\alpha_a\in a$. Two Euler structures $\xi,\eta$ with $\partial \xi = \sum_{a\in A} (-1)^{\dim(a)} \alpha_a$ and $\partial \eta = \sum_{a\in A} (-1)^{\dim(a)} \beta_a$ are \defterm{equivalent} if for some paths $x_a\colon [0, 1]\rightarrow a$ from $\alpha_a$ to $\beta_a$ the 1-cycle
\[\xi - \eta + \sum_{a\in A} (-1)^{\dim(a)} x_a\]
is a boundary. Let $\Eul(M)$ be the set of Euler structures on $M$.
\end{defn}

\begin{remark}
One should not confuse the notion of an Euler structure on the finite CW complex $M$ and an Euler structure on the derived prestack $M_\B$. We will show in \cref{prop:turaevtrivialization}, however, that the former induces the latter.
\end{remark}

The set of Euler structures on $M$ is nonempty if, and only if, $\chi(M) = 0$. In this case $\Eul(M)$ is a nonempty $\rH_1(M;\Z)$-torsor. Moreover, again under the assumption $\chi(M)=0$, there is a canonical isomorphism $E(M)\rightarrow \Eul(M)$ of $\rH_1(M; \Z)$-torsors. Clearly, an Euler structure is exactly a trivialization of the homological Euler class $e(M)\in\C_\bullet(M;\Z)$.

In the case of 3-manifolds the set of Euler structures has the following geometric description \cite{TuraevTorsion}.

\begin{prop}
Let $M$ be a closed oriented 3-manifold. There is a canonical bijection between the set of Euler structures on $M$ and the set of $\Spinc$-structures $\sigma$. Under this correspondence the characteristic class $c(\xi)\in\rH_1(M;\Z)$ of the Euler structure (see \cite[Section 5.2]{FarberTuraevAbsolute}) corresponds to the first Chern class $c_1(\sigma)\in\rH^2(M;\Z)\cong \rH_1(M;\Z)$ of the $\Spinc$-structure.
\label{prop:SpincEuler}
\end{prop}

Next, let us describe a trivialization of $[\C_\bullet(M; \Z)]\in\Omega^\infty\K(\Z)$.

\begin{defn}
A \defterm{homology orientation} of $M$ is an orientation of the $\R$-line $\det \rH_\bullet(M; \R)$.
\end{defn}

Using the canonical isomorphism
\[\phi\colon \det \C_\bullet(M; \R)\cong \det \rH_\bullet(M; \R)\]
as well as the base change isomorphism
\[(\det \C_\bullet(M;\Z))\otimes_\Z\R\cong \det \C_\bullet(M; \R)\]
we see that a homology orientation is the same as the choice of an isomorphism $\det\C_\bullet(M; \Z)\cong\Z$ as abelian groups. Thus, a homology orientation defines a trivialization of $[\C_\bullet(M; \Z)]\in \Omega^\infty\K(\Z)$.

\begin{remark}
Suppose $M$ is a closed oriented topological manifold of dimension $d$. Depending on $d$, there might be a canonical homology orientation:
\begin{itemize}
\item Suppose $d$ is odd. Choose an arbitrary orientation of $\det \rH_m(M; \R)$ for $m=0\dots (d-1)/2$. Poincare duality gives an isomorphism $\det\rH_m(M; \R)\cong (\det\rH_{d-m}(M; \R))^{-1}$ which, therefore, induces an orientation of $\det \rH_m(M; \R)$ for $m=(d+1)/2\dots d$. As the homology groups are paired, the corresponding homology orientation is independent of the original choices.

\item Suppose $d\equiv 2\pmod{4}$. Then $\rH_\bullet(M; \R)[-d/2]$ carries a symplectic structure. So, the symplectic volume form from \cref{sect:symplecticvolume} provides a canonical homology orientation.

\item If $d$ is divisible by $4$, there is no canonical homology orientation, i.e. in general $\det(\rH_\bullet(M; \R))$ defines a nontrivial character of the oriented mapping class group of $M$. As in the case $d\equiv 2\pmod{4}$, using Poincare duality a homology orientation is the same as an orientation of $\det \rH_{d/2}(M; \R)$. But as remarked in \cite{TuraevEuler}, the complex conjugation on $M=\mathbf{CP}^2$ preserves the orientation of $M$, but reverses an orientation of $\det \rH_2(M;\R)$.
\end{itemize}
\label{rmk:canonicalhomologyorientation}
\end{remark}

\begin{prop}
Suppose $M$ is a connected finite CW complex with $\chi(M) = 0$. Choose an Euler structure and a homology orientation on $M$. Then there is a canonical homotopy $[\cO_{M_\B}]\sim 0$ in $\Omega^\infty\K^\omega(M_\B)$. In other words, in this case there is a canonical Euler structure on the stack $M_\B$.
\label{prop:turaevtrivialization}
\end{prop}
\begin{proof}
We continue the proof of \cref{prop:CWlift} with the additional structure given in the present statement. The choice of the Euler structure allows us to make a canonical choice of the fundamental family of cells in $\tilde{M}$, so that the homotopy
\[
[\C_k(\tilde{M};\Z)] \sim \sum_{k=0} (-1)^k \left[\Z[\pi_1(M)]\otimes_\Z \C_k(M;\Z)\right]
\]
gives a well-defined lift along $\tau_{\leq 1}\K(\Z) \to \tau_{\leq 1} \C_\bullet(M; \K(\Z)) \to \K(M)$. It remains to trivialize
\[
\sum_{k=0} (-1)^k \left[\C_k(M;\Z)\right] \in\Omega^\infty\K(\Z),
\]
but this is exactly the datum of a homology orientation.
\end{proof}

\subsection{Poincare duality}

In this section we assume $2$ is invertible in $k$. Let $M\in\cS$ be a space and $\xi\in\Sp^M$ an invertible parametrized spectrum. Consider the visible Poincar\'e structure $\Qoppa^v_\xi\colon \Sp^{M, \omega, \op}\rightarrow \Sp$ from \cite[Definition 4.4.4]{CDHHLMNNS1}. By \cite[Corollary 4.6.1]{CDHHLMNNS2} there is a natural equivalence
\[\GW(\Sp^{M, \omega}, \Qoppa^v_\xi)\cong \LA^v(M, \xi),\]
where $\LA^v(M, \xi)$ are the visible $LA$-spectra from \cite{WeissWilliamsIII}.

Suppose $M\in\cS$ is finitely dominated. Then one can define the \defterm{Spivak normal fibration} $\zeta_M$ which satisfies the universal property
\[p_\sharp((-)\otimes \zeta_M)\cong p_*(-),\]
where $p\colon M\rightarrow \pt$ and $p_\sharp$ ($p_*$) is the left (right) adjoint to $p^*\colon \Sp\rightarrow \Sp^M$. Assume $\zeta_M$ is invertible (i.e. $M$ is a Poincar\'e duality space) and let $\xi=\zeta_M^{-1}$. Then $\bS_M\in\Sp^M$ has a canonical structure of a Poincar\'e object in $\Sp^{M, \omega}, \Qoppa^v_\xi)$ (see e.g. \cite[Corollary 4.4.20]{CDHHLMNNS1}). In particular, it defines a class
\[[\bS_M]\in \LA^v(M, \xi)\]
called the \defterm{visible signature} of $M$. There is a parametrized spectrum $\LA^v(\pt, \xi)$ over $M$ whose fiber at $x\in M$ is $\LA^v(\pt, \xi|_x)$. In this setting we still have the assembly map
\[\C_\bullet(M; \LA^v(\pt, \xi))\longrightarrow \LA^v(M, \xi),\]
where $\C_\bullet(M; \LA^v(\pt, \xi))=p_\sharp \LA^v(\pt, \xi)$. The following is shown in \cite[Section 10]{WeissWilliamsIII}.

\begin{thm}
Suppose $M$ is a closed topological manifold. Then there is a canonical lift of the visible signature $[\bS_M]\in\Omega^\infty\LA^v(M, \xi)$ along the assembly map
\[\C_\bullet(M; \LA^v(\pt, \xi))\longrightarrow \LA^v(M, \xi).\]
\label{thm:WWlift}
\end{thm}

Now suppose $M$ is a closed oriented topological $d$-manifold. Under the equivalence
\[\Sp^M\otimes_{\Sp} \Mod_k\cong \LocSys(M)\]
we have $\zeta_M\boxtimes k\mapsto k_M[-d]$, where $k_M\in\LocSys(M)$ is the constant local system over $M$ with fiber $k$. Thus, the base change from the sphere spectrum to $k$ defines a Poincar\'e functor
\[(\Sp^{M, \omega}, \Qoppa^v_\xi)\longrightarrow (\LocSys(M)^\omega, \Qoppa^{[d]}).\]
Using the equivalence provided by \cref{prop:fundamentalclassorientation} between fundamental classes and Poincar\'e structures, the Poincar\'e structure on $k_M\in\LocSys(M)^\omega$ corresponds to the usual fundamental class $[M]\in\rH_d(M; k)$.

\begin{remark}
As we are working over a ring $k$ where $2$ is invertible, there is no difference between symmetric, visible and quadratic Poincar\'e structures on $\LocSys(M)$.
\end{remark}

\begin{prop}
Let $M$ be a closed oriented topological $d$-manifold. Then the fundamental class of $M$ provides a fundamental class of $M_\B$ of degree $d$. Moreover, there is a simple structure with the Euler class $e_{\GW}(M_\B)\in\Omega^\infty\C_\bullet(M; \GW^{[d]}(k))$ on $M_\B$ compatible with Poincar\'e duality.
\end{prop}
\begin{proof}
The fundamental class
\[[M]\colon k\longrightarrow \C_\bullet(M; k)[-d]\]
of $M$ defines a fundamental class
\[[M_\B]\colon k\longrightarrow p_\sharp \cO_{M_\B}[-d].\]

The base change from the sphere spectrum to $k$ provides a commutative diagram of assembly maps
\[
\xymatrix{
\C_\bullet(M; \LA^v(\pt, \xi)) \ar[r] \ar[d] & \LA^v(M, \xi) \ar[d] \\
\C_\bullet(M; \GW^{[d]}(k)) \ar[r] & \GW^{\omega, [d]}(M_\B)
}
\]
By what we have explained above, the visible signature $[\bS_M]\in\LA^v(M, \xi)$ under the right vertical map goes to the class of $[\cO_{M_\B}]\in\Omega^\infty\GW^{\omega, [d]}(M_\B)$. Thus, the lift of the visible signature along the top assembly map provided by \cref{thm:WWlift} provides a lift of $[\cO_{M_\B}]\in\Omega^\infty\GW^{\omega, [d]}(M_\B)$ along the bottom assembly map, i.e. a simple structure on $M_\B$ compatible with Poincar\'e duality.
\end{proof}

\begin{remark}
There is a natural map $\GW^{[d]}(k)\rightarrow \rL(k)[-d]$, where $\rL(k)$ is the $L$-theory spectrum of symmetric bilinear forms over $k$. Under this map the class $e_{GW}(M_\B)\in\rH_0(M; \GW^{[d]}(k))$ goes to the fundamental $L$-homology class
\[[M]_{\rL}\in\rH_d(M; \rL(k))\]
from \cite[Proposition 16.16]{Ranicki}. There is a homomorphism $W(k)=\rL_0(k)\rightarrow \Z/2$ from the Witt group to $\Z/2$ given by rank modulo 2. Under this homomorphism the fundamental $L$-homology class goes to the usual $\Z/2$ homology class $[M]\in\rH_d(M; \Z/2)$.
\end{remark}

\subsection{Reidemeister torsion}
\label{sect:Reidemeister}

In this section we explain the relationship between our construction and Reidemeister torsion.

\begin{defn}
Let $G$ be an algebraic group and $M\in\cS$ a space. Let $\B G=[\pt/G]$ be the classifying stack. The \defterm{character stack} is the derived stack
\[\Loc_G(M) = \Map(M_\B, \B G)\]
parametrizing $G$-local systems on $M$.
\end{defn}

Now suppose $M$ is a connected finite CW complex equipped with a homology orientation. Consider the derived stack $\Loc_{\GL_n}(M)$ of rank $n$ local systems on $M$. It has a natural map
\[\Loc_{\GL_n}(M)\longrightarrow \uPerf(M_\B)\]
and we may pullback the determinant line to $\Loc_{\GL_n}(M)$. Its fiber at a $k$-point $F\in\Loc_{\GL_n}(M)$ is
\[\cD_F\cong \detgr(\C_\bullet(M; F)).\]

Recall the natural isomorphism
\[\phi\colon \det(\C_\bullet(M; F)) \cong \det(\rH_\bullet(M; F))\]
from \eqref{eq:Eulermap}. Given a local system $F$ of vector spaces over $M$ together with an Euler structure on $M$ one can define the Reidemeister--Turaev (refined) torsion (see \cite{TuraevReidemeister,FarberTuraevAbsolute}) which is a a nonzero element
\[\tau(M; F)\in \det(\rH_\bullet(M; F)).\]

\begin{prop}
Choose a homology orientation and an Euler structure on $M$ inducing an Euler structure on $M_\B$ as in \cref{prop:turaevtrivialization}. Let $\Delta$ be the determinant section of the determinant line bundle $\cD$ on $\Loc_{\GL_n}(M)$. Then its fiber at $F$, under the isomorphism $\phi$, coincides with the Reidemeister--Turaev torsion $\tau(F)$.
\label{prop:determinantReidemeister}
\end{prop}
\begin{proof}
Let us unpack the trivialization of $\det(\C_\bullet(M; F))$ from \cref{thm:pushforwardEulertrivial}. The local system $F$ corresponds to a $k[\pi_1(M)]$-module $F_0$, the fiber of $M$ at the basepoint of $M$. Then
\[\C_\bullet(M; F)\cong F_0\otimes_{k[\pi_1(M)]} \C_\bullet(\tilde{M}; k),\]
where $\tilde{M}\rightarrow M$ is the universal cover. The proof of \cref{prop:turaevtrivialization} gives a model of $\C_\bullet(\tilde{M}; k)$ as a chain complex of free based finite rank $k[\pi_1(M)]$-modules. Therefore, this gives a model of $\C_\bullet(M; F)$ as a chain complex whose $d$-th term is $F_0^{\oplus \# A_d}$, where $A_d$ is the set of $d$-cells on $M$. This identifies
\[\det(\C_\bullet(M; F))\cong \det(F_0)^{\chi(M)} = k\]
which gives the trivialization defined in \cref{thm:pushforwardEulertrivial}. But this is precisely the description of the refined torsion from \cite[Section 1.5]{FarberTuraevAbsolute}.
\end{proof}

We may also use the section $\Delta$ to define a volume form on the character stack $\Loc_G(M)$. Before we introduce it, let us consider the following construction. Suppose $\lambda\colon G\rightarrow \GL_1$ is a character of $G$ and $h\in \rH_1(M;\Z)$. Then there is a natural function $\langle h, \lambda\rangle$ on $\Loc_G(M)$ obtained by taking the holonomy of the rank 1 local system determined by $\lambda$ along $h$. For instance, we may apply this construction to the modular character $\Delta_G$ of $G$, i.e. the character of the $G$-representation $\det(\g)$.

\begin{prop}
Suppose $M$ is a finite CW complex equipped with a homology orientation and $G$ an algebraic group. Choose either of the following pieces of data:
\begin{itemize}
\item A $G$-invariant volume form on the Lie algebra $\g$ of $G$.
\item An Euler structure on $M$.
\end{itemize}
Then $\Loc_G(M)$ carries a canonical torsion volume form $\vol_{\Loc_G}$ and
\[\dim\Loc_G(M) = -\chi(M) \dim(G).\]
Changing the Euler structure by $h\in\rH_1(M;\Z)$ changes the volume form by
\[\vol_{\Loc_G}\mapsto \langle h, \Delta_G\rangle \vol_{\Loc_G}\]
and multiplying the volume form on $\g$ by a scalar $A\in k^*$ changes the volume form by
\[\vol_{\Loc_G}\mapsto A^{\chi(M)} \vol_{\Loc_G}.\]
\label{prop:LocGvolumeform}
\end{prop}
\begin{proof}
The classifying stack $\B G$ has pure dimension $\dim(\B G) = -\dim(G)$. A $G$-invariant volume form on $\g$ is the same as a volume form on the stack $\B G$. An Euler structure together with a homology orientation on $M$ gives rise to an Euler structure on $M_\B$ by \cref{prop:turaevtrivialization}. So, the result follows from \cref{thm:volumelift}.
\end{proof}

\begin{remark}
Note that there is a $G$-invariant volume form on $\g$ if, and only if, $G$ is unimodular, i.e. $\Delta_G = 1$. An Euler structure on $M$ exists if, and only if, $\chi(M) = 0$.
\end{remark}

Let us now describe some examples of computation of the volume form $\vol_{\Loc_G}$ on $\Loc_G(M)$. Suppose $G$ carries a $G$-invariant volume form on the Lie algebra $\g$ of $G$. Then there is a volume form $\vol_G$ on $G$ which is uniquely determined by the property that it is bi-invariant and which coincides with the chosen volume form on $\g=\T_e G$ at the unit. It induces a volume form on $G^n$ and a quotient volume form $\vol_{[G^n/G]}$ on $[G^n/G]$ by \cref{ex:quotientvolumeform}.

\begin{example}
Consider a wedge of $n$ circles
\[V_n = (S^1)^{\vee n}.\]
It has a standard CW structure with one $0$-cell $p$ and $n$ $1$-cells. Choose paths from each $1$-cell to the $0$-cell given anticlockwise with respect to the standard orientation of $S^1$ (see \cref{fig:circle}). This gives a model of the cellular chain complex $\C_\bullet(V_n; \cL)$ of a local system $\cL$ as
\[\cL_p^{\oplus n}\longrightarrow \cL_p,\]
where the differential is given by the sum of monodromies. A choice of ordering of the circles induces a homology orientation on $V_n$. We have
\[\Loc_G(V_n) \cong [G^n / G],\]
where $G$ acts on $G^n$ by a simultaneous conjugation. We claim that the torsion volume form $\vol_{\Loc_G}$ is given by the quotient volume form on $[G^n/G]$.

Indeed, let $f\colon G^n\rightarrow [G^n/G]$ be the projection. The torsion volume form on $\Loc_G(V^n)$ has the following description. First, we may identify
\[f^*\bL_{\Loc_G(V_n)}\cong \left((\cO_{G^n}\otimes \g^*)^{\oplus n}\rightarrow \cO_{G^n}\otimes \g^*\right),\]
where we use the left-invariant trivialization of $\bL_G$ (see \cref{sect:circle} for more details). The torsion volume form $\vol_{\Loc_G}$ is obtained by trivializing the determinant of $\bL_{\Loc_G(V_n)}$ using the trivialization of the determinant of $\g^*$ given by the chosen volume $G$-invariant volume form on $\g$ which is precisely the description of the quotient volume form $\vol_{[G^n/G]}$ on $[G^n/G]$.
\label{ex:wedgeofcircles}
\end{example}

\begin{example}
Consider a closed oriented surface $\Sigma$ of genus $g$. Consider a decomposition $\Sigma = \Sigma^\circ\cup_{S^1} D$ obtained by removing a disk from $\Sigma$. Then $\Sigma^\circ$ is homotopy equivalent to $V_{2g}$. This equivalence is compatible with simple structures as the Whitehead group of the free group $\pi_1(V_{2g})$ is zero. Moreover, by \cref{prop:gluedCWsimplestructures} the simple structure on $\Sigma_\B$ (coming from its structure as a finite CW complex) is glued from the simple structures on $\Sigma^\circ_\B$ and $D_\B$ along $S^1_\B$. By \cref{prop:gluedvolonmappingstacks} we obtain that the torsion volume form on $\Loc_G(\Sigma)$ is glued from the torsion volume forms on $\Loc_G(\Sigma^\circ)$ and $\Loc_G(D)$. Namely, we have
\[\Loc_G(\Sigma) \cong [G^{2g}/G]\times_{[G/G]} \B G\]
and we have shown that the volume from on $\Loc_G(\Sigma)$ is glued from $\vol_{[G^{2g}/G]}$, $\vol_{\B G}$ and $\vol_{[G/G]}$.
\label{ex:surfacevol}
\end{example}

\begin{example}
Consider a closed oriented 3-manifold $M$ together with a Heegaard splitting $M=N_1\cup_\Sigma N_2$, where $N_1$ and $N_2$ and handlebodies and where $\Sigma$ has genus $g$. Then $N_i$ are homotopy equivalent to $V_g$ and so by \cref{prop:gluedvolonmappingstacks} the torsion volume form on
\[\Loc_G(M)\cong [G^g/G]\times_{\Loc_G(\Sigma)} [G^g/G]\]
is glued from $\vol_{[G^g/G]}$ and $\vol_{\Loc_G(\Sigma)}$ (which was described in \cref{ex:surfacevol}).
\end{example}

\subsection{Symplectic volume forms on mapping stacks of surfaces}

Assume $2$ is invertible in $k$ throughout this section. Let $\Sigma$ be a closed oriented surface; the fundamental class $[\Sigma]\in\rH_2(\Sigma; \Z)$ endows $\Sigma_\B$ with an $\cO$-orientation of degree $2$ by \cref{prop:fundamentalclassorientation}. In this section we consider the following two closely related settings:
\begin{itemize}
    \item $(Y, \omega_Y)$ an $n$-shifted symplectic stack for $n\equiv 2\pmod 4$. By the AKSZ construction from \cite[Theorem 2.5]{PTVV} there is a natural $(n-2)$-shifted symplectic structure on $\Map(\Sigma_\B, Y)$. As $(n-2)$ is divisible by $4$ by assumption, we obtain the symplectic volume form $\vol_{\Map}$ on $\Map(\Sigma_\B, Y)$ as explained in \cref{sect:symplecticvolume}.
    \item $R$ is a (discrete) commutative $k$-algebra with $S=\Spec R$. $V\in\Perf^+(S\times \Sigma_\B)^{C_2}$ is a local system of perfect complexes of $R$-modules over $\Sigma$ equipped with a nondegenerate symmetric bilinear pairing. By \cref{sect:symplecticvolume} we obtain a symplectic volume form, which is an invertible element $\vol_{p_\sharp V}\in\det(p_\sharp V)$. We will say $V$ is \defterm{unimodular} if it comes with a trivialization of $\det(V)$ squaring to the canonical one provided by the nondegenerate pairing on $\det(V)$.
\end{itemize}

The two settings are connected by considering a morphism $f\colon S\rightarrow \Map(\Sigma_\B, Y)$ classifying a map $\tilde{f}\colon S\times \Sigma_\B\rightarrow Y$ and setting $V = \tilde{f}^* \bT_Y[-n/2]$ with the symmetric bilinear pairing induced by the symplectic structure $\omega_Y$.

Next, let us construct a torsion volume form in the two settings. We consider the following data that goes into its construction using \cref{thm:volumelift}:
\begin{itemize}
    \item The intersection pairing gives a symplectic structure on $\rH_\bullet(\Sigma;\Z)[-1]$ and hence the symplectic volume form gives a homology orientation $o$.
    \item The mod 2 Euler class $e(\Sigma)\in\C_\bullet(\Sigma; \Z/2)$ is the second Stiefel--Whitney class $w_2(\Sigma)$. As any oriented surface $\Sigma$ admits a spin structure, $w_2(X) = 0\in\rH_2(\Sigma; \Z/2)$. A trivialization of $e(X)$ on the chain level is the same as a spin structure $s$ on $\Sigma$.
    \item Using the $n$-shifted symplectic structure on $Y$ (where we recall that $n$ is assumed to be even) we get an isomorphism $\bT_Y\rightarrow \bL_Y[n]$ whose determinant defines a squared volume form on $Y$, i.e. a trivialization of $\det(\bL_Y)^{\otimes 2}$, as described in \cref{sect:Poincarevolumeform}.
\end{itemize}

By \cref{thm:volumelift} we obtain a torsion volume form $\tau_s(Y)$ on $\Map(\Sigma_\B, Y)$, where we emphasize the dependence on the spin structure $s$. Similarly, we obtain an invertible element $\tau_S(V)\in\det(p_\sharp V)$.

\begin{example}
Kasteleyn orientations give a convenient method to describe spin structures on a surface combinatorially as explained in \cite{CimasoniReshetikhin} (we refer to that paper for details on dimer configurations and Kasteleyn orientations). Consider a finite CW structure on $\Sigma$ and let the graph $\Gamma\subset \Sigma$ be the corresponding 1-skeleton. Then the Euler class is
\[e(\Sigma) = \sum_{v\in V} x_v - \sum_{e\in E} x_e + \sum_{f\in F} x_f\in\C_0(\Sigma; \Z),\]
where $V,E,F$ are the sets of $0$-, $1$- and $2$-cells and $x_{\dots}$ are some points in the interiors of the corresponding cells (as the cells are contractible, the precise location is irrelevant).

Choose a dimer configuration $D$ on $\Gamma$, i.e. a collection of edges in $\Gamma$ such that each vertex of $\Gamma$ is adjacent to exactly one edge in $D$. (There are combinatorial obstructions to the existence of dimer configurations; for instance, the number of $0$-cells has to be even.) Next, choose a Kasteleyn orientation $K$ on $\Gamma$. We are now going to define a class $e_{1/2}(\Sigma)\in\C_0(\Sigma; \Z)$ such that $2e_{1/2}(\Sigma) - e(\Sigma)\in\C_0(\Sigma; \Z)$ is the boundary of a 1-chain. Split the set of vertices $V=V^+\coprod V^-$ into even ones and odd ones as follows: given an edge $e\in D$ which flows towards a vertex $v\in V$ (with respect to the orientation $K$) we say $v$ is even; otherwise, $v$ is odd. Each edge $e$ borders two faces $f_1,f_2$ which are distinguished using the Kasteleyn orientation: $\epsilon^K_{f_1}(e) = -\epsilon^K_{f_2}(e)$. Let $n_f$ be the number of edges $e\in\partial f$ such that $\epsilon^K_f(e) = -1$ (this number is odd since the orientation is Kasteleyn). Let
\[e_{1/2}(X) = \sum_{v\in V^+} x_v +  \sum_{f\in F}\frac{1-n_f}{2} x_f.\]
The 1-chain whose boundary is $2e_{1/2}(X) - e(X)$ is given by the sum of oriented edges in $e\in D$ (which flows odd vertices to even vertices) and the paths from $x_e$ for each $e\in E$ into the face $f\in F$ that it borders with $\epsilon_f^K(e) = -1$. This combinatorial description of a spin structure gives a canonical element in $\det\C_\bullet(\Sigma; V)$ for any orthogonal local system $V$ over $\Sigma$ and hence it allows one to describe the torsion volume form.
\end{example}

\ 

Consider the ratio
\[\sigma_s(Y) = \frac{\vol_{\Map}}{\tau_s(Y)},\]
which is an invertible function on $\Map(\Sigma_B, Y)$. Pulling back this function to a derived  affine scheme $S$ along $f\colon S\rightarrow \Map(\Sigma_\B, Y)$, corresponding to $\tilde{f}\colon S\times \Sigma_\B\rightarrow Y$, we obtain
\[\sigma_s(V) = \frac{\vol_{p_\sharp V}}{\tau_s(V)},\]
where $v=\tilde{f}^* \bT_Y[-n/2]$. The goal of this section is to describe this ratio.

\begin{remark}
Note that we make a simplifying assumption that $R$ is discrete. So, we will describe the restriction of the function $\sigma_s(Y)$ to the underlying classical stack $t_0(\Map(\Sigma_\B, Y))$.
\end{remark}

We begin by establishing elementary properties of the function $\sigma_s(V)$. The determinant line of $W$ carries a nondegenerate pairing and hence it defines an element
\[\det(V)\in\rH^1(\Sigma; \mu_2(R)).\]
Consider the natural pairing
\[\langle -, -\rangle\colon \rH^1(\Sigma; \mu_2(R))\otimes_\Z \rH_1(\Sigma; \Z/2)\longrightarrow \mu_2(R)\otimes_\Z \Z/2\rightarrow \mu_2(R),\]
where the first map is the pairing between cohomology and homology and the second map is given by $\sigma\in\mu_2(R), n\in\Z/2\mapsto \sigma^n\in\mu_2(R)$.

\begin{prop}
Consider a local system $V\in\Perf^+(S\times \Sigma_\B)^{C_2}$ as above.
\begin{enumerate}
    \item $\sigma_s(V)\in\mu_2(R)$.
    \item Given two spin structures $s_1, s_2$ on $\Sigma$ whose difference is an element $h\in\rH_1(\Sigma; \Z/2)$ we have
    \[\sigma_{s_2}(V) = \langle \det(V), h\rangle \sigma_{s_1}(V).\]
    \item For a pair of local systems $V_1, V_2$ as above
    \[\sigma_s(V_1\oplus V_2) = \sigma_s(V_1)\sigma_s(V_2).\]
    \item Consider the constant local system $\underline{R}^{\oplus n}$ of rank $n$ with the symmetric bilinear pairing $(e_i, e_j) = \delta_{ij}$. Then
    \[\sigma_s(V\oplus \underline{R}^{\oplus n}) = \sigma_s(V).\]
\end{enumerate}
\label{prop:symplectictorsionsign}
\end{prop}
\begin{proof}
To show that $\sigma_s(V)^2 = 1$, we have to show that the squared volume forms $(\vol_{p_\sharp V})^2$ and $\tau_s(V)^2$ coincide. For this we have to prove that
\[(\tau_s(V), \tau_s(V)) = 1,\]
where $(-, -)\colon \det(p_\sharp V)\otimes \det(p_\sharp V)\rightarrow R$ is the nondegenerate pairing induced on $\det(p_\sharp V)$ from the symplectic structure on $p_\sharp V$. The left-hand side was computed in \cref{thm:volumeliftdualitysquared} to be
\[\langle o, o\rangle_{[\Sigma]}^{\chi(V)}.\]
But since the homology orientation $o$ was chosen to be given by the symplectic volume form, $\langle o, o\rangle_{[\Sigma]} = 1$. This proves the first claim.

To show that $\sigma_{s_2}(V) = \langle \det(V), h\rangle \sigma_{s_1}(V)$ we have to show that
\[\tau_{s_2}(V) = \langle \det(V), h\rangle \tau_{s_1}(V).\]
A trivialization of the mod 2 Euler class (i.e. a spin structure $s$) is the same as a $0$-chain $e^s_{1/2}(\Sigma)\in\C_0(\Sigma; \Z)$ and a $1$-chain $h_s\in\C_1(\Sigma; \Z)$ which satisfy $2e^s_{1/2}(\Sigma) - e(\Sigma) = \partial h_s$. In this case $h = h_{s_1}-h_{s_2}\in\rH_1(\Sigma; \Z/2)$. The volume form $\tau_s(V)\in\det(p_\sharp V)$ is constructed using the isomorphism
\[\det(p_\sharp V)\cong \langle \det(V), e(X)\rangle \otimes \det(p_\sharp \cO_X)^{\otimes \chi(V)}\cong \langle \det(V)^{\otimes 2}, e^s_{1/2}(X)\rangle \otimes \det(p_\sharp \cO_X)^{\otimes \chi(V)},\]
where the first isomorphism is the canonical isomorphism depending on the simple structure constructed using \cref{thm:indextheorem} and the second isomorphism is constructed from the 1-chain $h_s$. Thus, the two volume forms constructed using the two spin structures $s_1, s_2$ differ by a factor of $\langle \det(V), h_{s_1}-h_{s_2}\rangle$. This proves the second claim.

By \cref{thm:symplecticvolumeform} we have $\vol_{p_\sharp V_1}\otimes \vol_{p_\sharp V_2}\mapsto \vol_{p_\sharp(V_1\oplus V_2)}$ under the natural isomorphism
\[\det(p_\sharp V_1)\otimes \det(p_\sharp V_2)\cong \det(p_\sharp(V_1\oplus V_2)).\] Thus, we have to show that under the same isomorphism $\tau_s(V_1)\otimes \tau_s(V_2)\mapsto \tau_s(V_1\oplus V_2)$. I.e. the construction of the torsion volume form $\tau_s(-)$ is additive in $V$. This follows by analyzing each step of the construction:
\begin{itemize}
    \item By \cref{thm:indextheorem} we have a homotopy $[\pi_\sharp W]\sim \langle \epsilon([V]), e_\K(X)\rangle$ which is additive in $V$ as all maps in the claim are maps of spectra.
    \item By \cref{thm:determinantmap} the determinant is an $\bE_\infty$ ring map $\detgr\colon \uK\rightarrow \uPic^\Z$ which shows that the isomorphism $\det(p_\sharp V)\cong \langle V, e(X)\rangle \otimes \det(p_\sharp \cO_X)^{\chi(V)}$ is additive in $V$.
\end{itemize}
This proves the third claim.

To simplify the notation, denote by $M=R^{\oplus n}$ the free $R$-module with the symmetric bilinear pairing $(e_i, e_j) = \delta_{ij}$ and let $\underline{M}$ be the corresponding constant local system over $\Sigma$. By the third claim to show the fourth claim it is enough to prove that $\sigma_s(\underline{M}) = 1$. We have $p_\sharp(\underline{M})\cong M\otimes_k\C_\bullet(\Sigma; k)$, where the symplectic structure is the product of the symplectic structure on $\C_\bullet(\Sigma; k)$ and the symmetric pairing on $M$. By \cref{prop:symplectictimesorthogonal} under the natural isomorphism
\[\det(M)^{\chi(\Sigma)}\otimes_k \det(\C_\bullet(\Sigma; k))^n\longrightarrow \det(M\otimes_k \det(\C_\bullet(\Sigma; k)))\cong \det(p_\sharp \underline{M})\]
we have
\[(e_1\wedge \dots\wedge e_n)^{\chi(\Sigma)}\otimes o^n\mapsto \vol_{p_\sharp\underline{M}}.\]
But the torsion volume form $\tau_s(\underline{M})$, by definition, is the image of the same element under this isomorphism.
\end{proof}

By the previous claim, if $V$ is unimodular, then $\sigma_s(V)$ is independent of the spin structure; in this case we denote it by $\sigma(V)$.

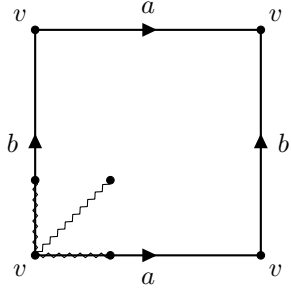
\begin{figure}[ht]
\begin{tikzpicture}
\draw[fill] (0, 0) circle (0.05);
\draw (-0.2, -0.2) node {$v$};
\draw[fill] (3, 0) circle (0.05);
\draw (3.2, -0.2) node {$v$};
\draw[fill] (3, 3) circle (0.05);
\draw (3.2, 3.2) node {$v$};
\draw[fill] (0, 3) circle (0.05);
\draw (-0.2, 3.2) node {$v$};
\draw[thick] (0, 0) -- (3, 0) node[currarrow, pos=0.5]{};
\draw (1.5, -0.3) node {$a$};
\draw[thick] (3, 0) -- (3, 3) node[currarrow, pos=0.5,sloped]{};
\draw (3.3, 1.5) node {$b$};
\draw[thick] (3, 3) -- (0, 3) node[currarrow, pos=0.5]{};
\draw (1.5, 3.3) node {$a$};
\draw[thick] (0, 0) -- (0, 3) node[currarrow, pos=0.5,sloped]{};
\draw (-0.3, 1.5) node {$b$};
\draw [line join=round, decorate, decoration={zigzag, segment length=4, amplitude=.9}] (0,0) -- (1, 1);
\draw [line join=round, decorate, decoration={zigzag, segment length=4, amplitude=.9}] (0,0) -- (0, 1);
\draw [line join=round, decorate, decoration={zigzag, segment length=4, amplitude=.9}] (0,0) -- (1, 0);
\draw[fill] (1, 1) circle (0.05);
\draw[fill] (1, 0) circle (0.05);
\draw[fill] (0, 1) circle (0.05);
\end{tikzpicture}
\caption{Torus with a chosen Euler structure.}
\label{fig:torus}
\end{figure}

\begin{example}
Consider the 2-torus $\Sigma=T^2$ with the CW structure with a unique $0$-cell $v$, two $1$-cells $a,b$ and a unique $2$-cell $f$. We orient them in the standard way as shown in the \cref{fig:torus}. With this orientation $f\in\rH_2(\Sigma; \Z)$ represents the fundamental class and the intersection pairing is $a\cdot b = 1$. Thus, the symplectic volume form for $\C_\bullet(\Sigma; k)$ provides a homology orientation which is
\[v\otimes (a\wedge b)^{-1}\otimes f\in\det\rH_\bullet(\Sigma; \Z).\]
The Euler structure (shown in the picture by squiggly lines) allows us to identify the chain complex $\C_\bullet(\tilde{M}; \Z)$ of chains on the universal cover, as an $R=k[\pi_1(T^2)]=k[x^{\pm 1}, y^{\pm 1}]$-module, with the free graded $R$-module on generators $v,a,b,f$ and with the differential
\begin{align*}
\partial f &= (1-y)a + (x-1)b \\
\partial a &= (x-1) v \\
\partial b &= (y-1) v
\end{align*}
The chosen Euler structure induces a spin structure on $T^2$ that we denote by $s$. A rank 1 local system equipped with a nondegenerate symmetric bilinear pairing is equivalently a $\mu_2$-local system $\cL=\cL_{\epsilon_x\epsilon_y}$ over $T^2$ which is specified by a pair of signs $\epsilon_x,\epsilon_y\in\mu_2$ given by the monodromies around the $a$ and $b$ cycle. By construction
\[\sigma_s(\cL_{++}) = 1.\]
Let us now show that
\[\sigma_s(\cL_{+-}) = \sigma_s(\cL_{-+}) = \sigma_s(\cL_{--}) = -1.\]
All these local systems are acyclic. In this case it is convenient to use the isomorphism \eqref{eq:Eulermap} to identify
\[\phi\colon \det\C_\bullet(\Sigma; \cL)\longrightarrow k.\]
Under this isomorphism $\vol_{p_\sharp \cL}$ goes to $1$ (the symplectic volume form of the zero vector space). Therefore,
\[\sigma_s(\cL) = \phi(\tau_s(\cL)).\]
Assume for simplicity that $\epsilon_x\neq 1$ (i.e. we consider the case $\cL_{-+}$ or $\cL_{--}$). We will use the formulas and the notation from \cite[Section 2.2]{FarberTuraevPR} to compute $\phi(\tau_s(\cL))$. Choose $b_1 = \{a\}$ and $b_2=\{f\}$. We have
\[[\partial(b_1) / v] = x-1,\qquad [\partial(b_2),b_1 / a,b] = \det\begin{pmatrix} 1-y & x-1 \\ 1 & 0\end{pmatrix} = 1-x,\qquad [b_2 / f] = 1.\]
Therefore,
\[\phi(\tau_s(\cL)) = -1.\]
\label{ex:toruscomputation}
\end{example}

\begin{remark}
For the surface $\Sigma$ the sign $\sigma_s(V)$ is diffeomorphism-invariant, i.e. for any diffeomorphism $f\colon \Sigma\rightarrow \Sigma$ we have
\[\sigma_{f^* s}(f^* V) = \sigma_s(V).\]
By \cref{prop:symplectictorsionsign} we get
\[\sigma_s(f^* W) = \langle \det(V), f^* s - s\rangle \sigma_s(V).\]
There is a unique (odd) diffeomorphism-invariant spin structure on $T^2$ and this equality as well as the computation of the sign $\sigma_s(\cL_{\epsilon_x\epsilon_y})$ performed in \cref{ex:toruscomputation} shows that the chosen spin structure $s$ is diffeomorphism-invariant.
\label{rmk:torusspinstructure}
\end{remark}

As explained in \cite{Bass} (where $\GW^+_n(R)$ is denoted by $KO_n(R)$) there is a canonical homomorphism
\[w_2\colon \GW^+_2(R)\longrightarrow \mu_2(R)\]
which is constructed by stabilizing the spin extension of the special orthogonal group. Equivalently we may think about it as a morphism of group prestacks
\[w_2\colon \tau_{\geq 2}\uGW^+\longrightarrow \B^2\mu_2.\]
We can apply this construction to define the second Stiefel--Whitney class in the two settings we consider:
\begin{itemize}
    \item Let $(Y, \omega_Y)$ be an $n$-shifted symplectic stack with $n\equiv 2\pmod 4$. Assume $\dim(Y) = 0$ and suppose $Y$ is equipped with a volume form $\vol_Y$ such that $(\vol_Y, \vol_Y)_{\omega_Y} = 1$. The shifted tangent complex defines a morphism $[\bT_Y[-n/2]]\colon Y\rightarrow \tau_{\geq 0}\uGW^+$. Let $\sh(-)$ be the \'etale sheafification of a prestack. Then the tangent complex defines a morphism
    \[[\bT_Y[-n/2]]\colon Y\longrightarrow \sh(\tau_{\geq 0} \uGW^+).\]
    By \cref{prop:GW1type} the morphism of derived stacks
    \[\detgr\colon \sh(\tau_{\geq 0}\uGW^+)\longrightarrow \uPic^{\Z, +, C_2}\]
    is an equivalence on 1-truncations, so using the volume form we lift the tangent complex to a morphism
    \[[\bT_Y[-n/2]]\colon Y\longrightarrow \sh(\tau_{\geq 2} \uGW^+).\]
    We define
    \[w_2(Y) = w_2(\bT_Y[-n/2])\colon Y\longrightarrow \B^2\mu_2.\]
    \item Let $V\in\Perf^+(S\times \Sigma_\B)^{C_2}$ be a unimodular local system with $\chi(V)=0$, where $S=\Spec R$. We have $\epsilon([V])\in\Omega^\infty\C^\bullet(\Sigma; \tau_{\geq 0}\GW^+(R))$. Again working \'etale locally on $R$ and using that $R\mapsto \rH^2(\Sigma; \mu_2(R))$ is an \'etale sheaf, we get a class
    \[w_2(V) = w_2(\epsilon([V]))\in\rH^2(\Sigma; \mu_2(R)).\]
\end{itemize}

\begin{remark}
More generally, if $Y$ is of pure dimension $\dim(Y)$ or $V$ has a constant rank $\chi(V)$, we define the second Stiefel--Whitney class by $w_2(V) = w_2(V \oplus \cO^{-\chi(V)})$, where $\cO^{-\chi(V)}$ is the (virtual) trivial local system of rank $-\chi(V)$.
\end{remark}

\begin{example}
Let $G$ be a connected algebraic group over a field $k$ whose Lie algebra $\g$ is equipped with a nondegenerate $G$-invariant symmetric bilinear pairing. Then the classifying stack $Y = \B G$ has a $2$-shifted symplectic structure. The adjoint representation defines a homomorphism $\rho\colon G\rightarrow \SO(\g)$. The pullback of the spin extension of $\SO(\g)$ to $G$ defines a homomorphism $\pi_1(G)\rightarrow \mu_2$ and, correspondingly, a morphism $w_2\colon \B G\rightarrow \B^2\mu_2$. For instance, for $G=\PGL_2$ the homomorphism $\pi_1(\PGL_2)\rightarrow \mu_2$ is nontrivial.
\end{example}

We are ready to state the main result of this section comparing the symplectic volume form and the torsion volume form.

\begin{thm}
Let $\Sigma$ be a closed oriented surface.
\begin{enumerate}
    \item Let $Y$ be an $n$-shifted symplectic stack for $n\equiv 2\pmod 4$, $\vol_{\Map}$ the symplectic volume form on $\Map(\Sigma_\B, Y)$ and $\tau(Y)$ the torsion volume form. Choose a volume form $\vol_Y$ on $Y$ such that $(\vol_Y, \vol_Y)_{\omega_Y} = 1$ and assume that $Y$ is of pure dimension $\dim(Y)$. Then, after restriction to the classical truncation $t_0(\Map(\Sigma_\B, Y))$, we have
    \[\vol_{\Map} = \left(\int_\Sigma \ev^* w_2(Y)\right) \tau(Y).\]
    \item Let $V\in \Perf^+(S\times \Sigma_\B)^{C_2}$ be a unimodular local system, where $S=\Spec R$, and assume that $\chi(V)$ is constant. Then
    \[\vol_{p_\sharp V} = \left(\int_\Sigma w_2(V)\right)\tau(V).\]
\end{enumerate}
\label{thm:symplectictorsionvolumeform}
\end{thm}

We begin with a lemma.

\begin{lm}
Suppose $R=F$ is an algebraically closed field. Then
\[w_2\colon \GW^+_2(F)/2\longrightarrow \mu_2(F)\]
is an isomorphism.
\label{lm:GW2}
\end{lm}
\begin{proof}
When $F$ is a field of characteristic different from $2$, the group $\GW^+_2(F)$ is computed in \cite[Theorem 1.2]{RoendigsOstvaer} to be the kernel of a homomorphism
\[\K_2(F)\oplus \mu_2(F)\longrightarrow {}_2\Br(F),\]
where ${}_2\Br(F)$ is the $2$-torsion subgroup of the Brauer group which vanishes when $F$ is algebraically closed. Moreover, in this case $K_2(F)/2$ also vanishes (e.g. by the norm residue isomorphism).
\end{proof}

\begin{proof}[Proof of \cref{thm:symplectictorsionvolumeform}]
The first statement follows from the second statement by pulling back to the affine scheme $S$ along a map $S\rightarrow \Map(\Sigma_\B, Y)$, so it is enough to prove the second statement, i.e. that
\[\sigma(V) = \int_\Sigma w_2(V).\]

Let us explicate the construction of $\sigma_s\colon \uGW^+(\Sigma_\B)\rightarrow \mu_2$. Consider the map
\[p_\sharp[-1]\colon \uGW^+(\Sigma_\B)\longrightarrow \uGW^-.\]
Consider the commutative diagram
\[
\xymatrix{
\tau_{\geq 0}\uGW^+(\Sigma_\B) \ar^{p_\sharp[-1]}[drr] \ar^{\epsilon\otimes e_{\GW}(\Sigma)}[d] && \\
\C^\bullet(\Sigma; \tau_{\geq 0}\uGW^+)\otimes \tau_{\geq 0}\C_\bullet(\Sigma; \GW^-(k)) \ar^-{\langle -, -\rangle}[r] \ar^{\detgr\otimes\detgr}[d] & \tau_{\geq 0}(\uGW^+\otimes \GW^-(k)) \ar^{\detgr\otimes\detgr}[d] \ar[r] & \tau_{\geq 0}\uGW^- \ar^{\detgr}[d] \\
\C^\bullet(\Sigma; \uPic^{\Z, +, C_2})\otimes \C_\bullet(\Sigma; \Pic^{\Z/2, -, C_2}(k)) \ar^-{\langle -, -\rangle}[r] & \uPic^{\Z, +, C_2}\otimes \Pic^{\Z, -, C_2}(k) \ar[r] & \uPic^{\Z, -, C_2}
}
\]
Here the top triangle commutes by \cref{thm:indextheoremwithPoincare} and the rightmost square commutes by \cref{prop:dualitydeterminant}. The rightmost vertical map $\detgr\colon \tau_{\geq 0}\uGW^-\rightarrow \uPic^{\Z, -, C_2}$ is nullhomotopic using the symplectic volume form. The element $\detgr(e_{\GW}(\Sigma))\in\Omega^\infty \C_\bullet(\Sigma; \Pic^{\Z/2, -, C_2}(k))$ admits a nullhomotopy using the spin structure and the canonical homology orientation on $\Sigma$ provided by the symplectic volume form on $\rH_\bullet(\Sigma; \Z)[-1]$. Comparison of the two nullhomotopies provides the morphism $\sigma_s\colon \uGW^+_0(\Sigma_\B)\rightarrow \pi_1 \uPic^{\Z, -, C_2}$.

Considering just the bottom part of the diagram we obtain a morphism
\[\pi_0(\C^\bullet(\Sigma; \tau_{\geq 2}\uGW^+)\otimes \tau_{\geq 0}\C_\bullet(\Sigma; \GW^-(k)))=\rH^2(\Sigma; \uGW^+_2)\otimes_\Z \rH_0(\Sigma; \GW^-(k))\longrightarrow \mu_2\]
under which $\epsilon([V])\otimes e_{\GW}(\Sigma)$ is sent to $\sigma(V)$. Observe that by \cref{prop:symplectictimesorthogonal} this morphism is trivial when restricted to $\rH^2(\Sigma; \uGW^+_2)\otimes_\Z \rH_0(\Sigma; \tau_{\geq 0}\GW^-(k))$. Due to the exact sequence
\[\rH_0(\Sigma, \tau_{\geq 0}\GW^-(k))\longrightarrow \rH_0(\Sigma, \GW^-(k))\longrightarrow \rH_0(\Sigma, \tau_{\leq -1}\GW^-(k))\longrightarrow 0\]
it factors through
\[\rH^2(\Sigma; \uGW^+_2)\otimes_\Z \rH_0(\Sigma; \tau_{\leq -1}\GW^-(k))\longrightarrow \mu_2.\]
As we are comparing elements of $\mu_2(R)$, it is enough to prove the claim when $R=F$ is an algebraically closed field in which case we obtain a morphism
\[\rH^2(\Sigma; \GW^+_2(F))\otimes_\Z \rH_0(\Sigma; \tau_{\leq -1}\GW^-(F))\longrightarrow \mu_2(F).\]
There is a natural morphism $\GW^-(F)\rightarrow \rL(F)[-2]$ of spectra which induces an isomorphism on negative homotopy groups (see e.g. \cite[Main Theorem]{CDHHLMNNS2}). As $\rL_1(F) = 0$ and $\rL_0(F) = W(F)\cong \Z/2$ is the Witt group, we obtain a morphism
\[\rH^2(\Sigma; \GW^+_2(F))\otimes_\Z \rH_2(\Sigma; \Z/2)\longrightarrow \GW^+_2(F)\otimes_\Z \Z/2\longrightarrow \mu_2(F),\]
where the first morphism is given by the natural pairing of chains and cochains. By \cref{lm:GW2} the second Stiefel--Whitney class provides an isomorphism $w_2\colon \GW^+_2(F)/2\rightarrow \mu_2(F)$, so the claim boils down to the fact that the homomorphism we have constructed is nontrivial (and, therefore, coincides with $w_2$). For this it is enough to exhibit an example of a local system $V$ where the sign $\sigma(V)$ is nontrivial.

To prove this claim, consider $\Sigma=T^2$ and the local systems $\cL_{\epsilon_1\epsilon_2}$ from \cref{ex:toruscomputation}. Consider the element
\[[V] = ([\cL_{+-}]-[\cO])([\cL_{-+}]-[\cO])\in \Omega^\infty\GW^+(\Spec F\times \Sigma_\B).\]
As $[\cL_{\pm\mp}]-[\cO]$ has (virtual) rank zero, $\detgr(\epsilon([V]))$ is trivial. Let us then compute the sign $\sigma_s([V]) = \sigma([V])$, where $s$ is the spin structure on $T^2$ from \cref{ex:toruscomputation}. We have $[V] = [\cL_{--}] + [\cL_{++}] - [\cL_{+-}] - [\cL_{-+}]$. Therefore,
\[\sigma_s([V]) = \sigma_s(\cL_{--}) \sigma_s(\cL_{++}) \sigma_s(\cL_{+-})^{-1} \sigma_s(\cL_{-+})^{-1} = -1\]
using the computation from \cref{ex:toruscomputation}. This finishes the proof.
\end{proof}

\begin{example}
Let $G$ be a connected simply-connected algebraic group over a field $k$ whose Lie algebra $\g$ is equipped with a nondegenerate $G$-invariant symmetric bilinear pairing. As $\pi_1(G)$ is trivial, $w_2(\B G)$ is trivial. Therefore, by \cref{thm:symplectictorsionvolumeform} we get that the torsion volume form on the character stack $\Loc_G(\Sigma)$ coincides with the symplectic volume form.
\end{example}

In \cref{ex:toruscomputation} we have computed $\sigma_s(V)$ for all rank 1 orthogonal local systems over $T^2$. In that case the sign $\sigma_s$ defines a function
\[\sigma_s\colon \rH^1(\Sigma; \mu_2)\longrightarrow \mu_2.\]
However, examining the precise values for $\Sigma=T^2$ we see that this map is not linear. Let us explain the precise behavior.

Recall that Johnson \cite{Johnson} has defined a quadratic function
\[q_s\colon \rH^1(\Sigma; \mu_2)\longrightarrow \mu_2\]
for a closed oriented surface $\Sigma$ and any spin structure $s$. Its underlying symmetric bilinear form is the intersection pairing
\[\frac{q_s(\alpha\beta)}{q_s(\alpha) q_s(\beta)} = \int_\Sigma \alpha\cup \beta.\]
Moreover, for two spin structures $s_1, s_2$ differing by $h\in\rH_1(\Sigma; \Z/2)$ it satisfies
\begin{equation}
q_{s_2}(\alpha) = \langle \alpha, h\rangle q_{s_1}(\alpha).
\label{eq:Johsonspindependence}
\end{equation}

There is a canonical homomorphism
\[w_1\colon \uGW^+_1\longrightarrow \mu_2\]
given by the stabilizing the determinant map on the orthogonal group. Equivalently, it is obtained by taking $\pi_1$ of the determinant morphism $\detgr\colon \uGW^+\rightarrow \uPic^{\Z, +, C_2}$. It allows us to identify orthogonal rank 1 local systems $\cL$ over $\Sigma$ with classes $w_1(\cL)\in\rH^1(\Sigma; \mu_2)$.

\begin{prop}
Let $\Sigma$ be a closed oriented surface and $s$ a spin structure on $\Sigma$. Consider a function
\[\sigma_s\colon \rH^1(\Sigma; \mu_2)\longrightarrow \mu_2\]
given by $\cL\mapsto \sigma_s(\cL)$, where we identify classes in $\rH^1(\Sigma; \mu_2)$ with orthogonal rank 1 local systems $\cL$ using $w_1$. Then $\sigma_s$ coincides with Johnson's quadratic refinement of the intersection pairing $q_s$.
\label{prop:Johnsonsign}
\end{prop}
\begin{proof}
We begin by showing that $\sigma_s$ is, indeed, a quadratic refinement of the intersection pairing. Consider orthogonal rank 1 local systems $\cL_1, \cL_2$ on $\Sigma$. Let $\cO$ be the trivial rank 1 local system and consider
\[[V] = ([\cL_1]-[\cO])([\cL_2]-[\cO]) = [\cL_1\otimes \cL_2] + [\cO] - [\cL_1] - [\cL_2]\in\Omega^\infty \uGW^+(\Sigma_\B).\]
As in the proof of \cref{thm:symplectictorsionvolumeform}, $\detgr(\epsilon([V]))$ is canonically trivial. Therefore, by \cref{thm:symplectictorsionvolumeform}
\[\frac{\sigma_s(\cL_1\otimes \cL_2)}{\sigma_s(\cL_1)\sigma_s(\cL_2)} = \sigma_s([V]) = \int_\Sigma w_2([V]).\]
We want to prove that
\[w_2([V]) = w_1(\cL_1)\cup w_1(\cL_2)\in\rH^2(\Sigma; \mu_2).\]
This is equivalent to showing the commutativity of the square
\[
\xymatrix{
\uGW^+_1\otimes_\Z \uGW^+_1\ar[r] \ar^{w_1\otimes w_1}[d] & \uGW^+_2 \ar^{w_2}[d] \\
\mu_2\otimes_\Z\mu_2 \ar[r] & \mu_2.
}
\]

Using that $w_1\colon \uGW^+_1\rightarrow \mu_2$ is an isomorphism \'etale locally (by \cref{prop:GW1type}) we are reduced to checking that the multiplication
\[\GW^+_1(F)\otimes \GW^+_1(F)\longrightarrow \GW^+_2(F)\]
is a nontrivial map for $F$ an algebraically closed field. For this we can compute $w_2([V])$ for some $\Sigma$, e.g. $\Sigma=T^2$. But this was done in \cref{thm:symplectictorsionvolumeform} where it was shown that $w_2([V])$ defines a nontrivial element of $\rH^2(T^2; \mu_2)$ for $\cL_1 = \cL_{+-}$ and $\cL_2=\cL_{-+}$. This finishes the proof that $\sigma_s$ is a quadratic refinement of the intersection pairing.

The difference ratio $\sigma_s/q_s$ defines a homomorphism
\[\sigma_s/q_s\colon \rH^1(\Sigma; \mu_2)\longrightarrow \mu_2.\]
Using the formula for the dependence of $\sigma_s$ on the spin structure from \cref{prop:symplectictorsionsign} and the formula \eqref{eq:Johsonspindependence} for the dependence of $q_s$ on the spin structure we see that the ratio defines a homomorphism $\sigma_s/q_s\colon \rH^1(\Sigma; \mu_2)\rightarrow \mu_2$ independent of the spin structure. Both $\sigma_s$ and $q_s$ are invariant under orientation-preserving diffeomorphisms, but no nontrivial elements of $\rH_1(\Sigma; \Z/2)$ are stable under all orientation-preserving diffeomorphisms. So, $\sigma_s/q_s = 1$.
\end{proof}

\begin{example}
Using that $\sigma_s=q_s$ is Johnson's quadratic refinement of the intersection pairing for rank 1 local systems we can compute the Arf invariant of the spin structure on $T^2$ from \cref{ex:toruscomputation}. We have
\[\Arf(\sigma_s) = \frac{1}{2}(\sigma_s(\cL_{++}) + \sigma_s(\cL_{--}) + \sigma_s(\cL_{+-}) + \sigma_s(\cL_{-+})) = -1.\]
So, the spin structure $s$ from \cref{ex:toruscomputation} is odd. As there is a unique odd spin structure on $T^2$ and the Arf invariant is preserved under diffeomorphisms, we get another proof of the assertion from \cref{rmk:torusspinstructure} that the spin structure $s$ is diffeomorphism-invariant.
\end{example}

\subsection{Cohomological DT invariants of 3-manifolds}

Let $M$ be an oriented 3-dimensional Poincar\'e complex, i.e. a finitely dominated space equipped with a fundamental class $[M]\in\rH_3(M; \Z)$ satisfying Poincar\'e duality. By \cref{prop:fundamentalclassorientation} the stack $M_\B$ carries an $\cO$-orientation of degree $3$.

Let $G$ be a connected algebraic group whose Lie algebra $\g$ is equipped with a nondegenerate $G$-invariant symmetric bilinear pairing. Then the classifying stack $\B G$ has a $2$-shifted symplectic structure. Therefore, by the AKSZ construction \cite[Theorem 2.5]{PTVV} the character stack
\[\Loc_G(M) = \Map(M_\B, \B G)\]
carries a $(-1)$-shifted symplectic structure. Recall the following notion from \cite{BBBBJ}.

\begin{defn}
Let $X$ be a $(-1)$-shifted symplectic stack. \defterm{Orientation data} on $X$ is a choice of a line bundle $K_X^{1/2}$ together with an isomorphism $(K_X^{1/2})^{\otimes 2}\cong \det(\bL_X)$, i.e. a square root of $\det(\bL_X)$.
\end{defn}

\begin{prop}
Let $M$ and $G$ be as before. Moreover, suppose $M$ is a finite CW complex (for instance, $M$ is a closed oriented 3-manifold). Then $\Loc_G(M)$ has a canonical orientation data.
\label{prop:LocGoriented}
\end{prop}
\begin{proof}
The assumptions on $G$ imply that $\g$ has a $G$-invariant volume form. Therefore, $\B G$ carries a volume form. Since $\dim(M) = 3$, there is a canonical homology orientation on $M$ coming from Poincar\'e duality (see \cref{rmk:canonicalhomologyorientation}). Consider the torsion volume form $\vol_{\Loc_G}$ on $\Loc_G(M)$ from \cref{thm:volumelift}. Then we may choose $K_{\Loc_G(M)}^{1/2} = \cO_{\Loc_G(M)}$ and the isomorphism $\cO_{\Loc_G(M)} = (K_X^{1/2})^{\otimes 2}\cong \det(\bL_X)$ given by $1\mapsto \vol_{\Loc_G}$.
\end{proof}

Until the end of this section assume $k=\bC$ is the field of complex numbers. By \cite{BBBBJ}, if $X$ is a $(-1)$-shifted symplectic stack equipped with orientation data, then its underlying classical stack $t_0(X)$ carries a canonical perverse sheaf $\phi$ of $\Q$-vector spaces globalizing the sheaf of vanishing cycles. Therefore, for any closed oriented 3-manifold $M$ and a group $G$ as above we may consider the cohomology
\[\rH^\bullet(t_0(\Loc_G(M)), \phi_{\Loc_G(M)}),\]
which is a \defterm{cohomological DT invariant} of $M$.

Now suppose that $G$ is a split connected reductive group, $P\subset G$ a parabolic subgroup and $L$ the Levi factor. Let $\g, \fp, \fl$ be the corresponding Lie algebras. Let $\Delta_P\colon P\rightarrow \GL_1$ be the modular character of $P$, i.e. the character of $\det(\fp)$.

We will be interested when $\Delta_P$ admits a square root. For this, choose a Borel subgroup $B\subset G$ and suppose $P$ is a standard parabolic as in \cite[Proposition 14.18]{Borel} associated to a subset $I\subset \Delta$ of simple roots. Let $\Phi^+$ be the set of positive roots with respect to $B$, $[I]$ is the root subsystem generated by $I$ and $\Phi(I)^+ = \Phi^+\setminus [I]$ the set of positive roots not lying in $[I]$. Consider the integral weight
\[2\rho_I = \sum_{\alpha\in\Phi(I)^+} \alpha.\]
The modular character of $P$ restricted to the maximal torus has weight $2\rho_I$. So, it admits a square root if, and only if, $\rho_I$ is an integral weight.

\begin{example}
The whole group $G\subset G$ is a parabolic subgroup; in this case $\Delta_G = 1$ admits a square root.
\end{example}

\begin{example}
For the Borel subgroup $B\subset \SL_2$ the modular character $\Delta_B$ admits a square root, but for the Borel subgroup $B\subset \PGL_2$ the modular character $\Delta_B$ does not admit a square root.
\end{example}

The nondegenerate pairing on $G$ restricts to one on $L$, so that $\B L$ has a 2-shifted symplectic structure and
\[
\xymatrix{
& \B P \ar[dl] \ar[dr] & \\
\B L && \B G
}
\]
is a 2-shifted Lagrangian correspondence, see \cite[Lemma 3.4]{SafronovImplosion}. Therefore,
\begin{equation}
\xymatrix{
& \Loc_P(M) \ar_{\pi_L}[dl] \ar^{\pi_G}[dr] & \\
\Loc_L(M) && \Loc_G(M)
}
\label{eq:Locparabolicinduction}
\end{equation}
is a $(-1)$-shifted Lagrangian correspondence. In this setting there is a relative notion of orientation data introduced in \cite[Definition 5.3]{AmorimBB} which we now review. Suppose $f\colon L\rightarrow X$ is a Lagrangian morphism to a $(-1)$-shifted symplectic stack. Then there is a fiber sequence
\[\bT_L\longrightarrow f^* \bT_X\longrightarrow \bL_L[-1]\]
and hence, after taking determinants, there is a canonical isomorphism
\begin{equation}
f^*\det(\bL_X)\cong \det(\bL_L)^{\otimes 2}.
\label{eq:Lagrangiansquare}
\end{equation}

\begin{defn}
Let $X$ be a $(-1)$-shifted symplectic stack equipped with orientation data $K_X^{1/2}$.  An \defterm{orientation data} on a Lagrangian morphism $f\colon L\rightarrow X$ is the data of an isomorphism $\det(\bL_L)\cong f^* K_X^{1/2}$ whose square coincides with the canonical isomorphism $f^* \det(\bL_X)\cong \det(\bL_L)^{\otimes 2}$ defined above.
\end{defn}

We can construct orientation data on the Lagrangian correspondence \eqref{eq:Locparabolicinduction} as follows.

\begin{thm}
Let $M$, $G$, $P$ and $L$ be as before. Suppose that $M$ is a closed oriented PL 3-manifold and one of the following holds:
\begin{enumerate}
\item $M$ is equipped with a spin structure.
\item The modular character $\Delta_P$ admits a square root.
\end{enumerate}
Then the Lagrangian correspondence
\[
\xymatrix{
& \Loc_P(M) \ar[dl] \ar[dr] & \\
\Loc_L(M) && \Loc_G(M)
}
\]
has canonical orientation data.
\label{thm:Locparabolicinductionorientation}
\end{thm}
\begin{proof}
By assumptions $M$ has trivial Euler characteristic, so we may choose an Euler structure $\xi$ on $M$. In fact, since the second Stiefel--Whitney class $w_2(M)\in\rH^2(M;\Z/2\Z)$ vanishes, we may choose a \emph{canonical} Euler structure in the sense of \cite[Section 3.2]{FarberTuraevAbsolute}, i.e. an Euler structure with characteristic class $c(\xi)=0\in\rH_1(M;\Z)$. Consider the torsion volume forms $\vol_{\Loc_G}$, $\vol_{\Loc_L}$ and $\vol_{\Loc_P}$ on the corresponding moduli spaces defined in \cref{prop:LocGvolumeform}. In particular, $\vol_{\Loc_G}$ and $\vol_{\Loc_L}$ define orientation data on the $(-1)$-shifted symplectic stacks $\Loc_G(M)$ and $\Loc_L(M)$. The volume form $\vol_{\Loc_P}$ defines an isomorphism $\det(\bL_{\Loc_P(M)})\cong \cO_{\Loc_P(M)}$. To check that this gives an orientation data on the Lagrangian correspondence we have to show that
\[\vol_{\Loc_G} \vol_{\Loc_L} = \vol_{\Loc_P}^2\]
under the isomorphism \eqref{eq:Lagrangiansquare}.

Fix a $P$-local system $Q\rightarrow M$ and consider the adjoint bundles
\[\ad_P Q = Q\times^P \fp,\qquad \ad_G Q = Q\times^P \g,\qquad \ad_L Q = Q\times^P \fl.\]

The isomorphism \eqref{eq:Lagrangiansquare} at $Q\in\Loc_P(M)$ boils down to an isomorphism
\begin{equation}
\det(\rH_\bullet(M; \ad_G Q\oplus \ad_L Q))\cong \det(\rH_\bullet(M; \ad_P Q))^{\otimes 2}
\label{eq:inductionisomorphism}
\end{equation}
constructed as a combination of the following two isomorphisms \eqref{eq:isomorphism1} and \eqref{eq:isomorphism2}. First, we have an exact sequence of $P$-representations
\[0\longrightarrow \fp\longrightarrow \g\oplus \fl\longrightarrow \fp^*\longrightarrow 0.\]
Taking the adjoint bundles and using multiplicativity of the determinant we obtain an isomorphism
\begin{equation}
\det(\rH_\bullet(M; \ad_G Q\oplus \ad_L Q))\cong \det(\rH_\bullet(M; \ad_P Q))\otimes \det(\rH_\bullet(M; (\ad_P Q)^*)).
\label{eq:isomorphism1}
\end{equation}
Second, using Poincar\'e duality on $M$ we obtain an isomorphism
\begin{equation}
\det(\rH_\bullet(M; \ad_P Q))\xrightarrow{\sim}\det(\rH_\bullet(M; (\ad_P Q)^*)).
\label{eq:isomorphism2}
\end{equation}

By \cref{prop:determinantReidemeister} the value of $\vol_{\Loc_P}$ at $Q\in\Loc_P(M)$ coincides with the Reidemeister--Turaev torsion $\tau(M; \ad_P Q)$ and similarly for the other groups. So, we have to show that under \eqref{eq:inductionisomorphism} we have
\[\tau(M; \ad_G Q) \tau(M; \ad_L Q) = \tau(M; \ad_P Q).\]
By the multiplicativity of torsions (see \cite[Theorem 7.1]{FarberTuraevPR}) we have
\[\tau(M; \ad_G Q)\tau(M;\ad_L Q) = \tau(M;\ad_P Q) \tau(M;(\ad_P Q)^*).\]
By the duality of torsions (see \cite[Theorem 7.2]{FarberTuraevPR}) we have
\[\tau(M; (\ad_P Q)^*) = \tau(M; \ad_P Q).\]
Note that the characteristic class of $\xi$ vanishes since we have assumed that $\xi$ is canonical. This proves that the volume forms $\vol_{\Loc_P}$, $\vol_{\Loc_G}$, $\vol_{\Loc_L}$ define orientation data on the Lagrangian correspondence.

We have defined the orientation data on the Lagrangian correspondence depending on the choice of a canonical Euler structure. Let us now consider the two possible assumptions:
\begin{enumerate}
\item If $M$ is equipped with a spin structure, it also carries a $\Spinc$-structure with trivial first Chern class. Therefore, by \cref{prop:SpincEuler} it carries a canonical Euler structure.
\item Suppose the modular character $\Delta_P$ admits a square root $\Delta_P^{1/2}$. Two canonical Euler structures $\xi_1,\xi_2$ on $M$ differ by a 2-torsion element $h\in\rH_1(M;\Z)$. By \cref{prop:LocGvolumeform} the volume form $\vol_{\Loc_P}$ changes as follows:
\[\vol_{\Loc_P, \xi_2} = \langle h, \Delta_P\rangle \vol_{\Loc_P, \xi_1},\]
while $\vol_{\Loc_G}$ and $\vol_{\Loc_L}$ do not change since $G$ and $L$ are unimodular. Using the square root of $\Delta_P$ we have
\[\langle h, \Delta_P\rangle = \langle h, \Delta^{1/2}_P\rangle^2 = \langle 2h, \Delta^{1/2}_P\rangle = 1.\]
In other words, in this case $\vol_{\Loc_P}$ is independent of the choice of a canonical Euler structure.
\end{enumerate}
\end{proof}

The above result has the following application. Let us recall the following conjecture of Joyce (see \cite[Conjecture 5.18]{AmorimBB}).

\begin{conjecture}
Let $X$ be a $(-1)$-shifted symplectic stack and $f\colon L\rightarrow X$ a Lagrangian morphism, where both $X$ and $L\rightarrow X$ are equipped with orientation data. Then there is a natural morphism
\[\mu_L\colon \Q_{t_0(L)}[\dim L]\longrightarrow f^! \phi_X.\]
\label{conj:Joyce}
\end{conjecture}

Let us now apply the conjecture to the Lagrangian correspondence \eqref{eq:Locparabolicinduction} which carries orientation data according to \cref{thm:Locparabolicinductionorientation}.

\begin{thm}
Suppose \cref{conj:Joyce} holds. Suppose either $M$ is equipped with a spin structure or the modular character $\Delta_P$ admits a square root. Then there is a natural parabolic induction map
\[
\rH^\bullet(t_0(\Loc_L(M)), \phi_{\Loc_L(M)})\longrightarrow \rH^\bullet(t_0(\Loc_G(M)), \phi_{\Loc_G(M)})
\]
between the cohomological DT invariants of $M$.
\label{thm:DTinduction}
\end{thm}
\begin{proof}
Let us first show that the morphism $\pi_G$ is representable and proper (as a morphism of underived stacks). Without loss of generality we may assume that $M$ is connected. The fundamental group of $M$ is finitely generated which gives a closed immersion $\Loc_G(M)\subset [G^n / G]$. Now consider a closed $G$-equivariant subscheme $X\subset G^n\times G/P$ consisting of elements $(g_1, \dots, g_n, [h])$ satisfying the equations $g_i[h] = [h]\in G/P$. Then we have a pullback diagram
\[
\xymatrix{
\Loc_P(M) \ar[r] \ar^{\pi_G}[d] & [X/G] \ar[d] \\
\Loc_G(M) \ar[r] & [G^n / G]
}
\]
The $G$-equivariant morphism $X\rightarrow G^n$ is obtained as a composition of a closed immersion $X\subset G^n\times G/P$ and a projection on the first factor, both of which are proper. Since proper morphisms are stable under base change, $\pi_G\colon \Loc_P(M)\rightarrow \Loc_G(M)$ is proper as well. Moreover, by \cref{prop:LocGvolumeform} we have $\dim(\Loc_P(M)) = 0$ since $\chi(M) = 0$.

Using the orientation data on the Lagrangian correspondence \eqref{eq:Locparabolicinduction} constructed in \cref{thm:Locparabolicinductionorientation} and \cref{conj:Joyce} we get a morphism
\[\Q_{t_0(\Loc_P(M))}\longrightarrow (\pi_L\times \pi_G)^!(\phi_{\Loc_L(M)}\boxtimes \phi_{\Loc_G(M)}).\]
Applying Verdier duality we get
\[\pi_L^* \phi_{\Loc_L(M)}\otimes \pi_G^* \phi_{\Loc_G(M)}\longrightarrow \omega_{t_0(\Loc_P(M))}\]
and, applying adjunctions,
\[\phi_{\Loc_L(M)}\longrightarrow (\pi_L)_*\pi_G^! \phi_{\Loc_G(M)}.\]
Let $p_L\colon t_0(\Loc_L(M))\rightarrow \pt$ be the projection. Applying $(p_L)_*$ to the above morphism we get
\[\rH^\bullet(t_0(\Loc_L(M)), \phi_{\Loc_L(M)})\longrightarrow \rH^\bullet(t_0(\Loc_G(M)), (\pi_G)_*\pi_G^! \phi_{\Loc_G(M)})\longrightarrow \rH^\bullet(t_0(\Loc_G(M)), \phi_{\Loc_G(M)}),\]
where the second map uses the counit $(\pi_G)_*\pi_G^!\rightarrow \id$ of the adjunction which exists since $\pi_G$ is proper.
\end{proof}

\section{Dolbeault and de Rham setting}
\label{sect:DolbeaultdeRham}

In this section we explain how to apply the results of \cref{sect:determinant} in the case of de Rham stacks.

\subsection{Setting}

Let $M$ be a smooth scheme. In this section we will be interested in the following stacks:
\begin{itemize}
    \item The \defterm{de Rham stack} $X=M_{\dR}$ is defined by the functor of points
    \[(M_{\dR})(R) = M(\rH^0(R)^{\red}).\]
    Let $\widehat{M\times M}$ be the formal completion of $M\times M$ along the diagonal. The two projections $\widehat{M\times M}\rightrightarrows M$ form a groupoid and one may identify
    \[M_{\dR}\cong [M / \widehat{M\times M}]\]
    with the groupoid quotient. One may identify $\QCoh(M_{\dR})$ with the derived $\infty$-category of $D$-modules on $M$ \cite{GaitsgoryRozenblyumCrys}.
    \item The \defterm{Dolbeault stack} $X=M_{\Dol}$ is defined to be the quotient
    \[M_{\Dol} = [M / \widehat{\T} M]\]
    of $M$ by the formal group scheme $\widehat{\T} M\rightarrow M$ given by the formal completion of the tangent bundle along the zero section with the group structure given by addition along the fibers. The pullback $\QCoh(M_{\Dol})\rightarrow \QCoh(M)$ under the projection $M\rightarrow M_{\Dol}$ is monadic and identifies $\QCoh(M_{\Dol})$ with $\Mod_{\Sym(\T_M)}(\QCoh(M))\cong \QCoh(\T^* M)$.
\end{itemize}

\begin{prop}
Suppose $M$ is a smooth and proper scheme. Then $X=M_{\dR}$ and $X=M_{\Dol}$ satisfy \cref{mainassumptionduality}.
\end{prop}
\begin{proof}
Denote as usual $p\colon X\rightarrow \pt$.

The claim for $X=M_{\dR}$ follows from the usual functoriality of $D$-modules and we omit the proof.

Let us now consider the case $X=M_{\Dol}$. Let $s\colon M\hookrightarrow \T^* M$ be the inclusion of the zero section, $\tilde{p}\colon M\rightarrow \pt$ and $f\colon M\rightarrow M_{\Dol}$ and $\tilde{\pi}\colon \T^* M\rightarrow M$. Then under the identification $\QCoh(M_{\Dol})\cong \QCoh(\T^* M)$ we have
\[p^* = s_*\tilde{p}^*,\qquad f^* = \tilde{\pi}_*.\]
Since $\tilde{p}$ is smooth and proper, $\tilde{p}^*$ admits colimit-preserving left and right adjoints. Moreover, since $s$ is a regular immersion, $s_*$ admits colimit-preserving left and right adjoints. Therefore, $p^*$ admits colimit-preserving left and right adjoints.

The pullback under the composite
\[\pt\xrightarrow{i} M\xrightarrow{\pi} M_{\Dol},\]
where $i$ is an inclusion of a point, is
\[i^*\tilde{\pi}_*\colon \QCoh(\T^* M)\longrightarrow \Mod_k.\]
It has a left adjoint satisfying the projection formula as $i^*$ does.

Finally, under the identification $\QCoh(M_{\Dol})\cong \QCoh(\T^* M)$ the functor 
\[\Delta^*\colon \QCoh(M_{\Dol}\times M_{\Dol})\rightarrow \QCoh(M_{\Dol})\]
goes to the integral transform along the correspondence
\[
\xymatrix{
& \T^* X\times_X \T^* X\ar^{m}[dr] \ar[dl] & \\
\T^* X\times \T^* X && \T^* X
}
\]
where $m\colon \T^* X\times_X \T^* X\rightarrow \T^* X$ is given by the addition along the fibers. So, $\Delta^*$ admits a left adjoint satisfying the projection formula precisely because the pullback along $\Delta_X\colon X\rightarrow X\times X$ admits a left adjoint satisfying the projection formula.
\end{proof}

\subsection{Lifts along the assembly map}

We begin by describing the $K$-theory of $M_{\Dol}$. Using the equivalence $\QCoh(M_{\Dol})\cong \QCoh(\T^* M)$ we have $\K^\omega(M_{\Dol})\cong \K(\T^* M)$. Under this equivalence $\cO_{M_{\Dol}}\in \QCoh(M_{\Dol})$ is sent to $\cO_{\T^*_M M}\in\QCoh(\T^* M)$, the structure sheaf of the zero section $\T^*_M M\subset \T^* M$. Using the Koszul resolution we obtain the following result.

\begin{prop}
Under the isomorphism $\K^\omega(M_{\Dol})\cong \K(M)$ the class $[\cO_{M_{\Dol}}]\in\Omega^\infty\K^\omega(M_{\Dol})$ is sent to the $K$-theoretic Euler class
\[e_K(\T^*_M) = \sum_{k=0}^\infty (-1)^k \left[\bigwedge^k \T_M\right]\in \Omega^\infty\K(M)\]
of the cotangent bundle $\T^*_M$. For instance, if $M$ is of pure dimension $d$, this is the top Chern class $c_d(\T^*_M)$.
\label{prop:DolbeaultEulerclass}
\end{prop}

Next, let us describe the $K$-theory of $M_{\dR}$. Recall that $\QCoh(M_{\dR})$ is the derived $\infty$-category of $D$-modules on $M$. It is compactly generated: $\QCoh(M_{\dR}) = \Ind\QCoh(M_{\dR})^\omega$. For a conic subset $S\subset \T^* M$ we denote by $\QCoh_S(M_{\dR})^\omega\subset \QCoh(M_{\dR})^\omega$ the subcategory of $D$-modules with singular support in $S$. For instance, for $S=\T^*_M M$ the zero section we get $\QCoh_{\T^*_M M}(M_{\dR})^\omega = \Perf(M_{\dR})$. Let $\K^\omega_S(M_{\dR})$ be the $K$-theory of $\QCoh_S(M_{\dR})^\omega$.

For a smooth scheme $X$ and a subset $S\subset X$ we denote by $\K_S(X)$ the $K$-theory of $X$ with support on $S$. The following is shown in \cite[Chapter 6, Theorem 7]{Quillen} and \cite[Corollary 3.1.16]{Patel}.

\begin{prop}
There is a commutative diagram
\[
\xymatrix{
\K^\omega_S(M_{\dR}) \ar[r] \ar[d] & \K^\omega(M_{\dR}) \ar^{\sim}[d] \\
\K_S(\T^* M) \ar[r] & \K(\T^* M).
}
\]
Under these morphisms $[\cO_{M_{\dR}}]\in\Omega^\infty\K^\omega_{\T^*_M M}(M_{\dR})$ goes to the class of the structure sheaf of the zero section $\T^*_M M\subset \T^* M$.
\label{prop:KD}
\end{prop}

\begin{remark}
Given a coherent $D$-module $\cF$ with a good filtration, the class $[\cF]\in\Omega^\infty\K^\omega(M_{\dR})$ goes to $[\gr\cF]\in\Omega^\infty\K(\T^* M)$.
\end{remark}

Using \cref{prop:KD} and the isomorphism $\K(\T^* M)\cong \K(M)$ we see that the assembly map for $M_{\dR}$ and $M_{\Dol}$ coincides with the assembly map for $M$ itself:
\[\C_\bullet(M(k); \K(k))\longrightarrow \K(M).\]

To construct a lift of $e_K(\T^*_M)\in \K(M)$ we will use the construction of de Rham $\epsilon$-factors from \cite{Groechenig}. Suppose $M$ is of pure dimension $d$. Consider the setting of \cite[Situation 3.1]{Groechenig}:
\begin{itemize}
    \item $Z\subset M$ is a closed subset of dimension $0$. $U=M\setminus Z$ is the complement.
    \item Consider an open covering $U=\cup_{i=1}^d U_i$ and regular one-forms $\nu_i$ on $U_i$ for each $i$.
    \item The one-forms $\{\nu_i\}$ satisfy the following condition. For each ordered subset $\{i_1<\dots<i_l\}\subset \{1,\dots,d\}$ we require that $\sum_{j=1}^l \lambda_j \nu_{i_j}$ nowhere vanishes on $U_{i_1\dots i_l} = \cap_{j=1}^l U_{i_j}$ for any $\lambda_1,\dots, \lambda_j\in k$ satisfying $\sum_{j=1}^l \lambda_j = 1$.
\end{itemize}

\begin{example}
Consider $M=\bP^2$ with homogeneous coordinates $[x:y:z]$. Let $Z=[1:0:0]\cup [0:0:1]$ and consider the open sets $U_1 = \{x\neq 0, z\neq 0\}$ and $U_2 = \{y\neq 0\}$ covering the complement of $Z$. Then the one-forms
\[\nu_1 = d(x/z),\qquad \nu_2 = d(x/y)\]
satisfy the assumptions.
\end{example}

\begin{prop}
Let $M$ be a smooth and proper scheme of pure dimension $d$ and $\{\nu_1, \dots, \nu_d\}$ a collection of 1-forms satisfying the above conditions. Then there is a simple structure on $M_{\Dol}$ and $M_{\dR}$.
\end{prop}
\begin{proof}
By \cite[Section 3.1]{Groechenig} there is a morphism $\nu^*\colon \K(\T^* M\setminus \T^*_M M)\rightarrow \K(U)$ determined by the collection of one-forms $\{\nu_i\}$ which fits into a commutative diagram
\[
\xymatrix{
\K(\T^* M) \ar[r] & \K(\T^*M\setminus \T^*_M M) \ar[d] \\
\K(M) \ar^{\tilde{\pi}^*}_{\sim}[u] \ar[r] & \K(U)
}
\]
Taking the fibers of the horizontal maps we obtain a morphism
\[\phi_{\underline{\nu}}\colon \K_{\T^*_M M}(\T^* M)\longrightarrow \K_Z(M).\]
We have
\[\K_Z(M) = \bigoplus_{z\in Z} \K(k)\]
and so the above commutative square constructs a lift of $[\cO_{M_{\Dol}}]$ (equivalently, $[\cO_{M_{\dR}}]$) under the assembly map.
\end{proof}

\begin{example}
Suppose $M$ is a smooth and proper curve. Let $\nu$ be a nonzero rational one-form on $M$ and $v=\nu^{-1}$ the corresponding rational vector field with divisor $\sum_i n_i x_i$ for some points $x_i\in M$. The vector field $v$ identifies $\T_M\cong \cO(\sum_i n_i x_i)$. Let
\[
\begin{cases}
[J_n\cO_x] = \sum_{j=1}^n [\T^j_{M, x}]\in\K_0(M), & \text{if $n>0$} \\
[J_n\cO_x] = -\sum_{j=1}^{-n} [\T^{1-j}_{M, x}]\in\K_0(M) & \text{if $n<0$}
\end{cases}
\]
where $\T^n_{M, x}$ is the skyscraper sheaf at $x\in M$ with fiber the $n$-th power of the tangent space. Using the exact sequences
\[0\longrightarrow \cO\longrightarrow \cO(x)\longrightarrow \T_{M, x}\longrightarrow 0\]
and
\[0\longrightarrow \cO(-x)\longrightarrow \cO\longrightarrow \cO_x\longrightarrow 0\]
we may identify
\[\left[\cO\left(\sum_i n_i x_i\right)\right] = \cO + \sum_i [J_{n_i}\cO_{x_i}]\in\K_0(M).\]
Therefore, using $\nu$ we identify
\[e(\T^*_M) = -\sum_i [J_{n_i}\cO_{x_i}]\in\K_0(M)\]
and the right-hand side lies in the source of the assembly map.
\end{example}

\section{Circle and the exponential map}

In this section we describe the behavior of the torsion volume form on the derived loop space under the exponential map. In this section $Y$ is a derived prestack which admits a deformation theory with a perfect cotangent complex.

\subsection{Circle}
\label{sect:circle}

Consider $M=S^1$ with the standard cell structure with a $0$-cell $p\in S^1$ and a $1$-cell $\gamma$ as shown in \cref{fig:circle}. Choose a clockwise orientation of $\gamma$ which induces a homology orientation. Equivalently, it is the canonical homology orientation induced using \cref{rmk:canonicalhomologyorientation} from the clockwise orientation of $S^1$. As the Euler structure $\xi$ we take the one given by an anticlockwise path from $\gamma$ to $p$. With this Euler structure and a homology orientation on $S^1$ we obtain a nullhomotopy
\[h_{S^1_\B}\colon [\cO_{S^1_\B}]\sim 0\in \Omega^\infty\K^\omega(S^1_\B)\]
by \cref{prop:turaevtrivialization}.

\begin{figure}[ht]
\begin{tikzpicture}
\draw[thick] (0, 1) circle (1);
\draw[fill] (0, 0) circle (0.05);
\draw (0, -0.2) node {$p$};
\draw [decorate, decoration={zigzag, segment length=4, amplitude=.9}] (0,0) arc (-90:-140:1);
\draw[fill] (-0.72, 0.32) circle (0.05);
\draw (0, 2.2) node {$\gamma$};
\end{tikzpicture}
\caption{Circle with a chosen Euler structure.}
\label{fig:circle}
\end{figure}
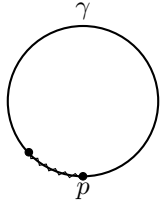

Let us unpack this nullhomotopy. First, let us identify the $\infty$-category $\QCoh(S^1_\B)$.

\begin{prop}
Let $S=\Spec R$ be a derived affine scheme. Then we have equivalences
\begin{itemize}
    \item $\QCoh(S\times S^1_\B)\cong \Mod_{k[z, z^{-1}]}(\QCoh(S)) = \Mod_{R[z, z^{-1}]}$, the $\infty$-category of quasi-coherent complexes $\cF$ on $S$ together with an automorphism $z\colon \cF\rightarrow \cF$. Under this equivalence $\cO_{S\times S^1_\B}$ goes to $\cO_S$ equipped with the identity automorphism.
    \item $\QCoh^\omega(S\times S^1_\B)\cong \Perf(R[z, z^{-1}])$.
    \item $\Perf(S\times S^1_\B)\cong \Mod_{k[z, z^{-1}]}(\Perf(S))$, the $\infty$-category of perfect complexes on $S$ together with an automorphism $z\colon \cF\rightarrow \cF$.
\end{itemize}
\label{prop:QCohS1}
\end{prop}
\begin{proof}
As $\QCoh(S)$ is dualizable, $\boxtimes\colon \QCoh(S)\otimes \QCoh(S^1_\B)\rightarrow \QCoh(S\times S^1_\B)$ is an equivalence. We have $\QCoh(S^1_\B) = \LocSys(S^1) = \Mod_{k[z, z^{-1}]}$, the $\infty$-category of modules over the group algebra of $S^1$. Thus,
\[\QCoh(S)\otimes \QCoh(S^1_\B)\cong \Mod_R\otimes \Mod_{k[z, z^{-1}]}\cong \Mod_{R[z, z^{-1}]}.\]

Compact objects in the $\infty$-category of modules are given by perfect modules which proves the claim about $\QCoh^\omega(S\times S^1_\B)$.

Finally, a quasi-coherent complex on $S\times S^1_\B$ is perfect, if, and only if, it is perfect when pulled back to $S\times \{p\}$ which proves the claim about $\Perf(S\times S^1_\B)$.
\end{proof}

Under the equivalence $\QCoh(S^1_\B)\cong \Mod_{k[z, z^{-1}]}$ the structure sheaf $\cO_{S^1_\B}$ goes to the augmentation module $k = k[z, z^{-1}]/(z-1)$. The nullhomotopy $[k]\sim 0$ is provided by choosing a resolution of $k$ by the chain complex $\C_\bullet(\tilde{S}^1; k)$ of free based $k[z, z^{-1}]$-modules of chains on the universal cover of $S^1$, which is
\[\C_\bullet(\tilde{S}^1; k) = (k[z, z^{-1}]\xrightarrow{z-1} k[z, z^{-1}]).\]

\begin{remark}
Two nullhomotopies of $[k]\in \K(k[z, z^{-1}])$ differ by an element of $\K_1(k[z, z^{-1}])$ which can be identified with $k[z, z^{-1}]^\times$ using the determinant map on $\K_1$ as $k[z, z^{-1}]$ is a Euclidean domain. The different nullhomotopies correspond to different choices of the generator of the free $k[z, z^{-1}]$-module $\ker(k[z, z^{-1}]\rightarrow k)$.
\end{remark}

Consider the derived loop stack
\[LY = \Map(S^1_\B, Y)\]
with $p\colon LY\rightarrow Y$ given by evaluation at $p\in S^1$. By \cref{thm:volumelift} we obtain a torsion volume form $\vol_{LY}$ on $LY$. The goal of this section is to give a more explicit description of this volume form. We begin with the following observation.

\begin{prop}
Consider the isomorphism
\[\det \bL_{LY}\cong p^*\det(\bL_Y)\otimes \det \bL_{LY/Y},\]
induced by the fiber sequence
\[p^*\bL_Y\longrightarrow \bL_{LY}\longrightarrow \bL_{LY/Y}.\]
The pullback diagram
\[
\xymatrix{
LY \ar^{p}[r] \ar^{p}[d] & Y \ar^{\Delta}[d] \\
Y \ar^{\Delta}[r] & Y\times Y
}
\]
induces an isomorphism
\[p^*\bL_{Y/Y\times Y}\cong \bL_{LY/Y}\]
and a fiber sequence
\[\bL_Y\oplus \bL_Y\longrightarrow \bL_Y\longrightarrow \bL_{Y/Y\times Y}\]
inducing an isomorphism
\[\det\bL_{Y/Y\times Y}\cong \det(\bL_Y)^{-1}.\]
The torsion volume form $\vol_{LY}$ is obtained by a sequence of the above isomorphisms:
\[\det \bL_{LY}\cong p^* \det\bL_Y\otimes \det \bL_{LY/Y}\cong p^*\det \bL_Y\otimes p^* \det \bL_{Y/Y\times Y}\cong p^*\det \bL_Y\otimes p^*\det(\bL_Y)^{-1}\cong \cO_{LY}.\]
\label{prop:LYvolume}
\end{prop}
\begin{proof}
The presentation of $S^1$ as a CW complex gives a pushout diagram
\[
\xymatrix{
\pt\coprod \pt \ar[r] \ar[d] & D \ar^{\gamma}[d] \\
\pt \ar^{p}[r] & S^1
}
\]

Let $\ev\colon S^1_\B\times LY\rightarrow Y$ be the evaluation map. The volume form $\omega$ is constructed using the following steps:
\begin{enumerate}
\item Consider the fiber sequence
\[\C_\bullet(\pt; p^* \bL_Y) \longrightarrow \C_\bullet(S^1; \ev^* \bL_Y)\longrightarrow \C_\bullet(S^1, \pt; \ev^*\bL_Y)\]
which corresponds to the fiber sequence
\[p^*\bL_Y\longrightarrow \bL_{LY}\longrightarrow \bL_{LY/Y}\]
via \cref{prop:mappingcotangent}.
\item Use the above pushout diagram to identify
\[\C_\bullet(S^1, \pt; \ev^*\bL_Y)\cong \C_\bullet(D, \pt\coprod \pt; \ev^*\bL_Y|_D)\]
which corresponds to the isomorphism
\[\bL_{LY/Y}\cong p^*\bL_{Y/Y\times Y}.\]
\item Identify
\[\C_\bullet(D, \pt\coprod \pt; \ev^*\bL_Y|_D)\cong p^*\bL_Y[1]\]
using the Euler structure and an orientation of the 1-cell $D$ which corresponds to the isomorphism
\[\bL_{Y/Y\times Y}\cong \bL_Y[1].\]
\item Take determinants of the above isomorphisms and fiber sequences to obtain a trivialization of $\det(\bL_{LY})$.
\end{enumerate}
These are precisely the isomorphisms described in the statement of the proposition.
\end{proof}

\begin{remark}
The sequence of isomorphisms above defines the ``canonical orientation'' of $LY$ in the sense of \cite[Construction 3.2.1]{KondyrevPrikhodko}.
\end{remark}

Let us now give a more geometric interpretation of the volume form $\vol_{LY}$. Let $Y\rightarrow Z$ be a morphism in an $\infty$-category with finite limits. Its \v{C}ech nerve is an augmented simplicial object $Y_\bullet$ with $Y_n = Y\times_Z \dots\times_Z Y$ (the product taken $n$ times). In this case $Y_\bullet$ is a groupoid object, see \cite[Proposition 6.1.2.11]{LurieHTT}. Concretely, the pullback diagram
\[
\xymatrix{
Y_2 \ar^{d_2}[r] \ar^{d_0}[d] & Y_1 \ar^{d_0}[d] \\
Y_1\ar^{d_1}[r]  & Y_0
}
\]
identifies $Y_2\cong Y_1\times_{Y_0} Y_1$ and the multiplication is given by $d_1\colon Y_1\times_{Y_0} Y_1\rightarrow Y_1$.

\begin{remark}
If $Y_0 = Y$ is the final object, $Y_1 = Y\times_Z Y$ is a group object.
\end{remark}

\begin{defn}
Let $Y\rightarrow Z$ be a morphism of derived stacks, where both of them admit a cotangent complex, and let $Y_\bullet$ be its \v{C}ech nerve. Consider the pullback diagram
\[
\xymatrix{
Y_1 \ar^{d_1}[r] \ar^{d_0}[d] & Y \ar^{d_0}[d] \\
Y \ar^{d_0}[r] & Z
}
\]
The induced isomorphism
\[\bL_{Y_1 / Y_0}\cong d_1^* \bL_{Y/Z}\]
is the \defterm{left-invariant trivialization} of the relative cotangent complex.
\label{def:leftinvariant}
\end{defn}

\begin{remark}
Reflecting the above diagram along the diagonal we obtain a \emph{right-invariant trivialization}.
\end{remark}

For instance, consider the diagonal map $Y\rightarrow Y\times Y$ in the $\infty$-category of derived stacks over $Y$, where we consider the projection on the first factor $Y\times Y\rightarrow Y$ on the right. Its \v{C}ech nerve gives the simplicial object
\[
\simp{Y}{LY}{LY\times_Y LY}
\]
which induces a group structure on $LY$ relative to $Y$ given by the loop composition.

\begin{example}
Consider $Y=\B G$, the classifying stack of an algebraic group $G$. The product $G\times G\rightarrow G$ is conjugation-invariant, so it defines a group structure on $LY=[G/G]$ relative to $Y=\B G$.
\end{example}

The isomorphism
\[\bL_{LY/Y}\cong p^*\bL_{Y/Y\times Y}\]
provided by \cref{prop:LYvolume} is given by the left-invariant trivialization of the relative cotangent complex using the group structure on $LY\rightarrow Y$.

\subsection{Formal circle}
\label{sect:BGa}

In this section we assume $k$ is a field of characteristic zero. Consider the classifying stack $X = \B \hGa$ of the formal additive group $\hGa$. It has a natural $\Gm$-action coming from the $\Gm$-action on $\hGa$.

\begin{prop}
Let $S = \Spec R$ be a derived affine scheme. Then we have equivalences
\begin{itemize}
    \item $\QCoh(S\times \B\hGa)\cong \Mod_{k[x]}(\QCoh(S)) = \Mod_{R[x]}$, the $\infty$-category of quasi-coherent complexes $\cF$ on $S$ together with an endomorphism $x\colon \cF\rightarrow \cF$. Under this equivalence $\cO_{S\times \B\hGa}$ is sent to $\cO_S$ equipped with the zero endomorphism.
    \item $\QCoh^\omega(S\times \B\hGa)\cong \Perf(R[x])$.
    \item $\Perf(S\times \B\hGa)\cong \Mod_{k[x]}(\Perf(S))$, the $\infty$-category of perfect complexes on $S$ together with an endomorphism $x\colon \cF\rightarrow \cF$.
\end{itemize}
\label{prop:QCohBhGa}
\end{prop}
\begin{proof}
As in the proof of \cref{prop:QCohS1}, the functor $\QCoh(S)\otimes \QCoh(\B\hGa)\rightarrow \QCoh(S\times \B\hGa)$ is an equivalence.

Since $\hGa$ is a formally smooth indscheme, by \cite[Theorem 10.1.1]{GaitsgoryRozenblyumIndSch} the functor
\[\Upsilon\colon \QCoh(\B \hGa)\longrightarrow \IndCoh(\B \hGa)\]
is an equivalence. Under this functor $\cO_{\B\hGa}$ is sent to $\omega_{\B\hGa}$.

If $\g$ is a Lie algebra and $\widehat{G}$ the corresponding formal group, by \cite[Chapter 7, Corollary 5.2.4]{GaitsgoryRozenblyum2} there is an identification
\[\IndCoh(\B\widehat{G})\cong \Mod_{\U\g}\]
under which $\omega_{\B\widehat{G}}$ is sent to the augmentation module $k$. Applying this result to $\widehat{G} = \hGa$ we get
\[\IndCoh(\B\hGa)\cong \Mod_{k[x]}\]
and hence $\QCoh(\B\hGa)\cong \Mod_{k[x]}$.
\end{proof}

\begin{prop}
The derived prestack $\B\hGa$ satisfies \cref{mainassumption}.
\end{prop}
\begin{proof}
The fact that the functor $p^*$ from \cref{mainassumption} admits a left adjoint follows from \cite[Chapter 3, Proposition 2.1.2]{GaitsgoryRozenblyum2}. The object $k\in\Mod_{k[x]}$ is compact, so $p^*$ also admits a right adjoint $p_*$.
\end{proof}

We define a nullhomotopy
\[h_{\B\hGa}\colon [\cO_{\B\hGa}]\sim 0\in\Omega^\infty\K^\omega(\B\hGa)\]
as follows. By \cref{prop:QCohBhGa} we may identify this class with $[k]\in\Omega^\infty\K(k[x])$. Consider the resolution of $k$ by the chain complex of free based $k[x]$-modules
\[k\cong (k[x]\xrightarrow{x} k[x])\]
concentrated in degrees $-1$ and $0$. This provides a homotopy $[k]\sim [k[x]] - [k[x]] = 0$.

\begin{prop}
There is a pushout square of derived stacks equipped with $\Gm$-actions
\[
\xymatrix{
\Spec k[x]/x^2 \ar[r] \ar[d] & \pt \ar[d] \\
\pt \ar[r] & \B\hGa
}
\]
where $\Gm$ acts on $x$ with weight $1$.
\label{prop:BGapushout}
\end{prop}
\begin{proof}
Recall from \cite[Chapter 5]{GaitsgoryRozenblyum2} that a formal moduli problem over $\pt$ is a derived stack $X$ satisfying the following conditions: $X$ is locally almost of finite type, $X$ admits a deformation theory and $X^{\red} = \pt$. An equivalence due to Lurie and Pridham (see \loccit) asserts that the $\infty$-category of formal moduli problems over $\pt$ are equivalent to the $\infty$-category of dg Lie algebras.

The stacks $\Spec k[x]/x^2,\pt,\B\hGa$ are all formal moduli problems over $\pt$ which correspond to the following Lie algebras:
\begin{itemize}
    \item The Lie algebra corresponding to $\Spec k[x]/x^2$ is two-dimensional with generators $a,b$ of degrees $\deg(a)=1,\deg(b)=2$ and the bracket $[a,a]=b$. Equivalently, it is the free Lie algebra $\flie(a)$ on a generator $a$ of degree $1$.
    \item The Lie algebra corresponding to $\pt$ is 0.
    \item The Lie algebra corresponding to $\B\hGa$ is $k$ in degree $0$. Equivalently, it is the free Lie algebra $\flie(x)$ on a generator $x$ of degree $0$.
\end{itemize}

Thus, using the equivalence between formal moduli problems and Lie algebras, we have to construct a pushout square
\[
\xymatrix{
\flie(a) \ar[r] \ar[d] & 0 \ar[d] \\
0 \ar[r] & \flie(x)
}
\]

But such a square is obtained by applying the functor $\flie$ (as it is a left adjoint, it preserves colimits) to the pushout square
\[
\xymatrix{
k[-1] \ar[r] \ar[d] & 0 \ar[d] \\
0 \ar[r] & k
}
\]
of chain complexes.
\end{proof}

Consider the derived mapping prestack $\Map(\B\hGa, Y)$ together with a projection $p\colon\Map(\B\hGa, Y)\rightarrow Y$ given by evaluation at the basepoint $p\in\B\hGa$. We also have the shifted tangent bundle $\T[-1] Y$ which is an example of a linear stack in the sense of \cite{Monier}.

\begin{prop}
There is an equivalence of derived prestacks $\Map(\B\hGa, Y)\cong \T[-1] Y$ compatible with the $\Gm$-action under which $p$ corresponds to the projection map $p\colon \T[-1] Y\rightarrow Y$.
\end{prop}
\begin{proof}
Applying $\Map(-, Y)$ to the pushout square from \cref{prop:BGapushout} we obtain a pullback square
\[
\xymatrix{
\Map(\B\hGa, Y) \ar[r] \ar[d] & Y \ar[d] \\
Y \ar[r] & \Map(\Spec k[x]/x^2, Y)
}
\]
By \cite[Proposition 1.4.1.9]{HAGII} we may identify $\Map(\Spec k[x]/x^2, Y)\cong \T Y$, so that inclusion of $Y$ is given by the zero section. The claim then follows from the fact that
\[
\xymatrix{
\bT_Y[-1] \ar[r] \ar[d] & 0 \ar[d] \\
0 \ar[r] & \bT_Y
}
\]
is a pullback diagram in $\QCoh(Y)$.
\end{proof}

Using \cref{prop:QCohBhGa} we obtain a natural forgetful functor
\[\QCoh(\B\hGa\times \T[-1] Y)\longrightarrow \Mod_{k[x]}(\QCoh(\T[-1] Y)),\]
i.e. a quasi-coherent complex on $\B\hGa\times \T[-1] Y$ gives rise to a quasi-coherent complex on $\T[-1] Y$ equipped with an endomorphism. This construction is equivariant for the natural $\Gm$-action on $\B\hGa$, so that the endomorphism has weight $1$. Now suppose $E\in \QCoh(Y)$ is a quasi-coherent complex. Then under this functor $\ev^* E$ corresponds to the quasi-coherent complex $p^* E$ equipped with an endomorphism $p^* E\rightarrow p^* E$ of weight $1$, i.e. a map $E\rightarrow p_* p^* E$ of weight $1$. If we further assume that $E\in\QCoh(X)^-$ is bounded above, by \cite[Theorem 2.5]{Monier} it is the same as a map
\[\at_E\colon E\longrightarrow E\otimes \bL_Y[1].\]

\begin{defn}
Let $E\in\QCoh(Y)^-$ be a bounded above quasi-coherent complex. The \defterm{Atiyah class} of $E$ is the map $\at_E\colon E\rightarrow E\otimes \bL_Y[1]$ defined above (equivalently, a weight 1 endomorphism $\at_E\colon p^* E\rightarrow p^* E$).
\end{defn}

Consider the map $Y\rightarrow \T Y$ given by the inclusion of the zero section. Its \v{C}ech nerve gives the simplicial object
\[
\simp{Y}{\T[-1] Y}{\T[-1] Y\times_Y \T[-1] Y}
\]
which induces an abelian group structure on $\T[-1] Y$ relative to $Y$ given by addition in the fiber coordinate.

\begin{example}
Consider $Y=\B G$, the classifying stack of an algebraic group $G$ and let $\g$ be the Lie algebra of $G$. The addition map $\g\times \g\rightarrow \g$ is conjugation-invariant, so it defines an abelian group structure on $\T[-1] Y = [\g / G]$ relative to $\B G$.
\end{example}

Given the nullhomotopy $[\cO_{\B\hGa}]\sim 0$ in $\Omega^\infty\K^\omega(\B\hGa)$ constructed above, by \cref{thm:volumelift} we obtain a torsion volume form $\vol_{\T[-1] Y}$ on $\T[-1] Y$. Our goal is to give a geometric description of this volume form similar to the description of the volume form $\omega_{LY}$ on $LY$ given in the previous section.

\begin{prop}
The pullback $\pi^*\colon \QCoh(\pt)\rightarrow \QCoh(\Spec k[x]/x^2)$ along $\Spec k[x]/x^2\rightarrow \pt$ admits a left adjoint $\pi_\sharp$, such that $\pi_\sharp\cO_{\Spec k[x]/x^2}\cong (k[x]/x^2)^*$ and the counit is dual to the inclusion of the unit $k\rightarrow k[x]/x^2$.
\end{prop}
\begin{proof}
$\Spec k[x]/x^2$ is a proper lci scheme. Therefore, $\pi_*$ is colimit-preserving and it admits a right adjoint $\pi^!$ differing from $\pi^*$ by tensoring by a line bundle \cite[Proposition 7.3.8]{Gaitsgory}. Therefore, $\pi^*$ admits a left adjoint $\pi_\sharp$.

We have
\[\Hom(\pi_\sharp\cO_{\Spec k[x]/x^2}, k)\cong \Hom(\cO_{\Spec k[x]/x^2}, \cO_{\Spec k[x]/x^2}) = k[x]/x^2,\]
so we may canonically identify $\pi_\sharp\cO_{\Spec k[x]/x^2}\cong (k[x]/x^2)^*$.
\end{proof}

Let $s\colon Y\rightarrow \T Y$ be the inclusion of the zero section. Using \cref{prop:mappingcotangent} we identify the fiber sequence
\[s^*\bL_{\T Y}\longrightarrow \bL_Y\longrightarrow \bL_{Y/\T Y}\]
with
\[(k[x]/x^2)^*\otimes \bL_Y\longrightarrow \bL_Y\longrightarrow \bL_{Y/\T Y},\]
where the first map is induced by $(k[x]/x^2)^*\rightarrow k$. Thus, there is a canonical identification
\[\bL_{Y/\T Y}\cong \bL_Y[1].\]
The following statement is proven analogously to \cref{prop:LYvolume}.

\begin{prop}
Consider the fiber sequence
\[
p^*\bL_Y\longrightarrow \bL_{\T[-1] Y}\longrightarrow \bL_{\T[-1] Y/Y}
\]
inducing an isomorphism
\[\det \bL_{\T[-1] Y}\cong p^*\det(\bL_Y)\otimes \det \bL_{\T[-1] Y/Y}\]
and the left-invariant trivialization (see \cref{def:leftinvariant})
\[\bL_{\T[-1] Y/Y}\cong p^* \bL_Y[1]\]
of the relative cotangent complex. The torsion volume form $\vol_{\T[-1] Y}$ is obtained by a sequence of the above isomorphisms
\[\det \bL_{\T[-1] Y}\cong p^* \det\bL_Y\otimes \det \bL_{\T[-1] Y/Y}\cong p^*\det \bL_Y\otimes p^*\det(\bL_Y)^{-1}\cong \cO_{\T[-1] Y}.\]
\end{prop}

\subsection{Exponential}

The goal of this section is to compare the Euler structures on $S^1_\B$ and $\B\hGa$ constructed in the previous two sections. For this, consider the stack $\B\Ga$. It carries natural maps
\[
\xymatrix{
& \B\Ga & \\
S^1_\B \ar^{i^m}[ur] && \B\hGa \ar_{i^a}[ul]
}
\]
obtained by taking classifying stacks of the inclusions $\Z\rightarrow\Ga$ and $\hGa\rightarrow\Ga$. We denote by
\[\pi^u\colon \B\Ga\longrightarrow \pt,\qquad \pi^m\colon S^1_\B\longrightarrow \pt,\qquad \pi^a\colon \B\hGa\longrightarrow \pt\]
the natural projections.

We begin by describing the $\infty$-category $\QCoh(\B\Ga)$. We denote by
\[\Mod^0_{k[x]} = \colim_n \Mod_{k[x]/x^n}\]
the $\infty$-category of $x$-nilpotent $k[x]$-modules.

\begin{prop}
Let $S = \Spec R$ be a derived affine scheme. Then we have equivalences
\begin{itemize}
    \item $\QCoh(S\times \B\Ga)\cong \Mod^0_{k[x]}(\QCoh(S)) = \Mod^0_{R[x]}$, the $\infty$-category of quasi-coherent complexes $\cF$ on $S$ together with a nilpotent endomorphism $x\colon \cF\rightarrow \cF$. Under this equivalence $\cO_{S\times \B\Ga}$ is sent to $\cO_S$ equipped with the zero endomorphism.
    \item $\Perf(S\times \B\Ga)\cong \Mod_{k[x]}(\Perf(S))$, the $\infty$-category of perfect complexes on $S$ together with a nilpotent endomorphism $x\colon \cF\rightarrow \cF$.
\end{itemize}
\label{prop:QCohBGa}
\end{prop}
\begin{proof}
The pullback along $p\colon S\rightarrow S\times \B\Ga$ defines a comonadic functor $f^*\colon \QCoh(S\times \B\Ga)\rightarrow \QCoh(S)$ which identifies
\[\QCoh(S\times \B\Ga)\cong \coMod_{\cO(\Ga)}(\QCoh(S)).\]
We may identify $\cO(\Ga)\cong \colim_n k[t]/t^n$ as coalgebras. Thus,
\[\coMod_{\cO(\Ga)}(\QCoh(S))\cong \colim_n \coMod_{k[t]/t^n}(\QCoh(S)).\]
Identifying $(k[t]/t^n)^*\cong k[x]/x^n$ as algebras, we get
\[\coMod_{\cO(\Ga)}(\QCoh(S))\cong \colim_n \Mod_{k[x]/x^n}(\QCoh(S)) =\Mod^0_{k[x]}(\QCoh(S)).\]
\end{proof}

For $\B\Ga$ perfect and compact quasi-coherent complexes coincide:
\[\Perf(\B\Ga)\cong \QCoh^\omega(\B\Ga),\]
see \cite[Corollary 3.22]{BZFN}. Our next goal is to define the integration map $\pi^u_\sharp\colon \Perf(S\times\B\Ga)\rightarrow \Perf(S)$. We begin with the following lemma.

\begin{lm}
Let $\cC$ be a presentable $k$-linear $\infty$-category and $i\colon \cD\hookrightarrow \cC$ a full subcategory. Let $x\in\cD$ be an object and consider the colimit-preserving functors $F_\cD\colon \Mod_k\rightarrow \cD$ and $F_\cC\colon \Mod_k\rightarrow \cC$ given by $V\mapsto V\otimes x$. Suppose $F_\cC$ has a left adjoint $F^L_\cC\colon \cC\rightarrow \Mod_k$. Then $F_\cD$ has a left adjoint given by the composite
\[\cD\xrightarrow{i} \cC\xrightarrow{F^L_\cC} \Mod_k.\]
\label{lm:fullyfaithfuladjoint}
\end{lm}
\begin{proof}
For $M\in\cD$ and $V\in\Mod_k$ we have a sequence of equivalences
\[
\Hom_{\Mod_k}(F^L_\cC i(M), V)\cong \Hom_\cC(i(M), i(F_\cD(V)))\cong \Hom_\cD(M, F_\cD(V)),
\]
where the first equivalence uses that $F^L_\cC$ is a left adjoint and the second equivalence uses that $i$ is fully faithful.
\end{proof}

By \cref{prop:QCohBGa} the pullback functors
\[(i^a)^*\colon \QCoh(S\times \B\Ga)\longrightarrow \QCoh(S\times \B\hGa),\qquad (i^m)^*\colon \QCoh(S\times \B\Ga)\longrightarrow \QCoh(S\times S^1_\B)\]
are fully faithful. Therefore, by \cref{lm:fullyfaithfuladjoint} we obtain a functor $\pi^u_\sharp\colon \Perf(S\times \B\Ga)\rightarrow \Perf(S)$, left adjoint to the pullback 
functor $(\pi^u)^*\colon\Perf(S)\rightarrow \Perf(S\times \B\Ga)$, which is compatible with the corresponding functors $\pi^a_\sharp$ and $\pi^m_\sharp$ on $\B\hGa$ and $S^1_\B$ using the diagram
\[
\Perf^\vee(S^1_\B)\longrightarrow \Perf^\vee(\B\Ga)\longleftarrow \Perf^\vee(\B\hGa)
\]

To obtain a volume form on the mapping stack from $\B\Ga$, we have to trivialize the pushforward functor $\pi^u_\sharp$ for $\B\Ga$ in $K$-theory. For this, consider a commutative diagram
\begin{equation}
\xymatrix{
& \Perf(\B\Ga)\cong \QCoh^\omega(\B\Ga) \ar_{(i^m)^*}[dl] \ar^{(i^a)^*}[dr] & \\
\Perf(k[z,z^{-1}]) \ar_{z=\exp(x)}[r] &\Perf(k[\![x]\!]) & \Perf(k[x]) \ar[l]
}
\label{eq:BGacommutative}
\end{equation}
of stable $\infty$-categories, where the two bottom functors denote induction functors along the inclusions $k[z, z^{-1}]\hookrightarrow k[\![x]\!]$ and $k[x]\hookrightarrow k[\![x]\!]$.

Let $S=\Spec R$ be a derived affine scheme. Recall that for any $X$ satisfying \cref{mainassumption} we have the functor
\[\tens_X\colon \QCoh^\omega(X)\longrightarrow \Fun^{ex}(\Perf(S\times X), \Perf(S)).\]

Let us unpack this functor for $X = \B\hGa$. By \cref{prop:QCohBhGa} the functor $\tens_X$ is equivalent to a functor
\[\Mod_{k[x]}(\Perf(R))\otimes \Perf(k[x])\longrightarrow \Perf(R)\]
which sends $M\in \Mod_{k[x]}(\Perf(R))$ and $N\in \Perf(k[x])$ to $M\otimes_{k[x]} N$.

We will now define a functor
\begin{equation}
\Perf(k[\![x]\!])\longrightarrow \uPerf^\vee(\B\Ga)
\label{eq:BGadualityintegration}
\end{equation}
analogous to $\tens_{\B\hGa}$. For $S=\Spec R$ a derived affine scheme we define it to be
\[\Mod^0_{k[x]}(\Perf(R))\otimes \Perf(k[\![x]\!])\longrightarrow \Perf(R)\]
which sends $M\in \Mod^0_{k[x]}(\Perf(R))$ and $N\in\Perf(k[\![x]\!])$ to $M\otimes_{k[\![x]\!]} N$.

Let us make several observations about this construction:
\begin{itemize}
    \item The functors $\tens_{\B\hGa}$ and \eqref{eq:BGadualityintegration} fit into a commutative diagram
    \[
    \xymatrix@C=1cm{
    & \Perf(\B\Ga)\cong\QCoh^\omega(\B\Ga) \ar[dl] \ar[dr] & \\
    \Perf(k[z,z^{-1}]) \ar_{z=\exp(x)}[r] \ar[d] &\Perf(k[\![x]\!]) \ar[d] & \Perf(k[x]) \ar[l] \ar^{\tens_{\B\hGa}}[d] \\
    \Perf^\vee(S^1_\B) \ar[r] & \Perf^\vee(\B\Ga) & \Perf^\vee(\B\hGa) \ar[l]
    }
    \]

    \item Under $\Perf(k[\![x]\!])\rightarrow \Perf^\vee(\B\Ga)$ the module $N=k[\![x]\!]$ is sent to the functor
    \[\Perf(S\times\B\Ga)\cong \Mod^0_{k[x]}(\Perf(R))\rightarrow \Perf(R)\]
    given by forgetting the $x$-module structure. The module $N=k[\![x]\!]/(x)=k$ is sent to the functor $\pi_\sharp\colon \Perf(S\times \B\Ga)\rightarrow \Perf(S)$.
\end{itemize}

Applying $K$-theory to the commutative diagram \eqref{eq:BGacommutative} of stable $\infty$-categories we get a commutative diagram
\[
\xymatrix{
& \K(\B\Ga)\cong\K^\omega(\B\Ga) \ar_{(i^m)^*}[dl] \ar^{(i^a)^*}[dr] & \\
\K^\omega(S^1_\B) \ar[r] & \K(k[\![x]\!]) & \K^\omega(\B\hGa) \ar[l]
}
\]

We have the following data:
\begin{itemize}
\item There is a class $[\cO_{\B\Ga}]\in\Omega^\infty\K(\B\Ga)$ which maps to $[\cO_{S^1_\B}]\in\Omega^\infty\K^\omega(S^1_\B)$, $[\cO_{\B\hGa}]\in\Omega^\infty\K^\omega(\B\hGa)$ and $[k]\in\Omega^\infty\K(k[\![x]\!])$.
\item There is a nullhomotopy $h_{S^1_\B}\colon [\cO_{S^1_\B}]\sim 0\in\Omega^\infty\K^\omega(S^1_\B)$ constructed in \cref{sect:circle}.
\item There is a nullhomotopy $h_{\B\hGa}\colon [\cO_{\B\hGa}]\sim 0\in\Omega^\infty\K^\omega(\B\hGa)$ constructed in \cref{sect:BGa}.
\end{itemize}

So, we can take the difference $h_{S^1_\B} - h_{\B\hGa}$ of the two nullhomotopies in $\Omega^\infty\K(k[\![x]\!])$ to obtain an element of $\K_1(k[\![x]\!])$.

\begin{remark}
As opposed to the other examples we have considered previously, the class $[\cO_{\B\Ga}]\in\Omega^\infty\K^\omega(\B\Ga)$ is nontrivial. Indeed, by devissage pullback along $\pt\rightarrow \B\Ga$ induces an equivalence $\K(\B\Ga)\rightarrow \K(\pt)$ and under this equivalence $[\cO_{\B\Ga}]\in\Omega^\infty\K(\B\Ga)$ goes to $[k]\in\Omega^\infty\K(k)$.
\end{remark}

\begin{prop}
The images of the nullhomotopies $h_{S^1_\B}$ and $h_{\B\hGa}$ in $\Omega^\infty\K(k[\![x]\!])$ differ by $J(x)$, where
\[J(x) = \frac{x}{\exp(x)-1}\in k[\![x]\!]^\times\cong \K_1(k[\![x]\!]).\]
\label{prop:Duflodifference}
\end{prop}
\begin{proof}
The nullhomotopy $h_{\B\hGa}$ is represented by the free based complex $k[x]\xrightarrow{x} k[x]$ (in degrees $-1, 0$) of $k[x]$-modules with a quasi-isomorphism to $k$. Similarly, the nullhomotopy $h_{S^1_\B}$ is represented by the free based complex $k[z, z^{-1}]\xrightarrow{z-1} k[z, z^{-1}]$ of $k[z, z^{-1}]$-modules with a quasi-isomorphism to $k$. Thus, the corresponding loop is given by
\[
0 \sim \left(k[\![x]\!] \xrightarrow{\exp(x)-1} k[\![x]\!]\right) \sim k \sim \left(k[\![x]\!] \xrightarrow{x} k[\![x]\!]\right) \sim 0
\]
Note that the middle two paths can be composed (by lifting the identity on $k$ to the resolutions) to
\[
\xymatrix@C=2cm{
k[\![x]\!] \ar[r]^{\exp(x)-1} \ar[d]_{J^{-1}} & k[\![x]\!] \ar[d]^{\id} \\
k[\![x]\!] \ar[r]^{x} & k[\![x]\!].
}
\]
Thus, the total loop is represented by the free based acyclic complex of $k[\![x]\!]$-modules
\[k[\![x]\!]\xrightarrow{(\exp(x)-1)\oplus J^{-1}} k[\![x]\!]^{\oplus 2}\xrightarrow{-\id\oplus x}k[\![x]\!]\]
in degrees $-1,0,1$. It is contractible with the nullhomotopy $h$ given by
\[
k[\![x]\!] \xleftarrow{\id\oplus J(2-\exp(x))} k[\![x]\!]^{\oplus 2}\xleftarrow{(\exp(x)-2)\oplus J^{-1}}k[\![x]\!]
\]
This contractible complex gives rise to an invertible matrix $d+h$ from the even part to the odd part given by
\[
\begin{pmatrix}
1 & J(x)(2-\exp(x)) \\
-1 & x
\end{pmatrix}
\]
whose determinant is $J$.

Since $k[\![x]\!]$ is a Euclidean domain, the determinant map $\K_1(k[\![x]\!])\rightarrow k[\![x]\!]^\times$ is an isomorphism splitting the obvious inclusion $k[\![x]\!]^\times\rightarrow \K_1(k[\![x]\!])$.
\end{proof}

Now consider the \defterm{unipotent loop space}
\[L^u Y=\Map(\B\Ga, Y)\]
which carries maps
\[
\xymatrix{
& L^u Y \ar_{q^m}[dl] \ar^{q^a}[dr] & \\
LY && \T[-1] Y
}
\]

By \cref{prop:mappingcotangent} it admits a perfect cotangent complex given by
\[\bL_{L^u Y} = \pi^u_\sharp \ev^*\bL_Y.\]

\begin{example}
Consider $Y=\B G$, the classifying stack of an algebraic group $G$ with Lie algebra $\g$. Let $U\subset G$  be the variety of unipotent elements and $\cN\subset \g$ the variety of nilpotent elements. Then
\[L^u Y\cong [\widehat{G}_U / G]\cong [\widehat{\g}_\cN / G],\]
where $\widehat{G}_U$ is the formal completion of $G$ along the unipotent cone and similarly for $\widehat{\g}_\cN$ \cite[Proposition 2.1.25]{Chen}.
\end{example}

\begin{prop}
The maps $q^m\colon L^u Y\rightarrow LY$ and $q^a\colon L^u Y\rightarrow \T[-1] Y$ are formally \'etale.
\end{prop}
\begin{proof}
The statements for $LY$ and $\T[-1] Y$ are proven in the same way, so we only consider the second case. Consider the commutative diagram
\[
\xymatrix{
& \B\Ga\times L^u Y \ar^{\ev^u}[dr] & \\
\B\hGa\times L^u Y \ar^{i^a\times \id}[ur] \ar_{\id\times q^a}[dr] && Y \\
& \B\hGa\times \T[-1] Y \ar_{\ev^a}[ur] &
}
\]
induced by the map $i^a\colon \B\hGa\rightarrow \B\Ga$. We have
\[\bL_{L^u Y}\cong \pi^u_\sharp(\ev^u)^*\bL_Y,\qquad \bL_{\T[-1] Y}\cong \pi^a_\sharp (\ev^a)^*\bL_Y.\]
Using the above commutative diagram the map $(q^a)^* \bL_{\T[-1] Y}\rightarrow \bL_{LY}$ can be identified with
\[\pi^a_\sharp (i^a)^*(\ev^u)^*\bL_Y\longrightarrow \pi^u_\sharp (\ev^u)^*\bL_Y,\]
but $\pi^a_\sharp (i^a)^*\rightarrow \pi^u_\sharp$ is an equivalence by \cref{lm:fullyfaithfuladjoint}.
\end{proof}

Using the previous proposition we obtain two volume forms $(q^a)^*\omega_{\T[-1] Y}$ and $(q^m)^*\omega_{LY}$ on $L^u Y$. We can relate them as follows. Let $p\colon L^u Y\rightarrow Y$ be the projection. For a bounded above quasi-coherent complex $E\in\QCoh(Y)^-$, the restriction of the Atiyah class to $L^u Y$ defines a nilpotent endomorphism
\[\at_E\colon p^* E\longrightarrow p^* E.\]
For any invertible power series $f(x)\in k[\![x]\!]$ we obtain an automorphism
\[f(\at_E)\colon p^* E\longrightarrow p^* E.\]
If $E$ is a perfect complex, we may take the determinant of this automorphism to obtain an invertible function
\[\det f(\at_E)\in\cO(L^u Y)^\times.\]

\begin{thm}
There is an equality
\[(q^m)^* \vol_{LY} = (q^a)^*\vol_{\T[-1] Y} \det\left(\frac{\at_{\bL_Y}}{\exp(\at_{\bL_Y})-1}\right)\]
of volume forms on $L^u Y$.
\label{thm:Toddclass}
\end{thm}
\begin{proof}
Consider a morphism $f\colon S\times \B\Ga\rightarrow Y$ corresponding to an $S$-point $\tilde{f}\colon S\rightarrow L^u Y$.

Consider the composite
\[F\colon \Perf(k[\![x]\!])\longrightarrow \Perf(S\times \B\Ga)\otimes \Perf(k[\![x]\!])\longrightarrow \Perf(S),\]
where the first functor is given by the inclusion of $f^* \bL_Y$ and the second functor is \eqref{eq:BGadualityintegration}. Note that $F(k) \cong \pi_\sharp(f^*\bL_Y)$ and $F(k[\![x]\!]) = \tilde{f}^* p^* \bL_Y$. Let
\[[F]\colon \K(k[\![x]\!])\longrightarrow \K(S)\]
be the induced map on $K$-theory. Consider the element $[F](h_{S^1} - h_{\B\hGa})\in\K_1(S)$ and its determinant $\det([F](h_{S^1} - h_{\B\hGa}))\in\cO(S)^\times$. Unpacking the definitions, we have
\[\tilde{f}^* \frac{(q^m)^* \vol_{LY}}{(q^a)^* \vol_{\T[-1] Y}} = \det([F](h_{S^1} - h_{\B\hGa})).\]

Consider the commutative diagram
\[
\xymatrix{
k[\![x]\!]^\times = \Aut_{\Perf(k[\![x]\!])}(k[\![x]\!])\ar^{F}[r] \ar^{\sim}[d] & \Aut_{\Perf(S)}(\tilde{f}^* p^* \bL_Y) \ar[d] \\
\K_1(k[\![x]\!]) \ar^{[F]}[r] & \K_1(S) \ar^{\det}[d] \\
& \cO(S)^\times 
}
\]

Using \cref{prop:Duflodifference} we obtain
\[\det([F](h_{S^1} - h_{\B\hGa})) = \tilde{f}^* \det(J(\at_{\bL_Y})),\]
which proves the claim.
\end{proof}

\begin{remark}
We have
\[\det\left(\frac{\at_{\bL_Y}}{\exp(\at_{\bL_Y})-1}\right) = \det\left(\frac{\at_{\bT_Y}}{1-\exp(-\at_{\bT_Y})}\right)\in\cO(L^u Y)^\times.\]
If $Y$ is a smooth scheme, the map $q^a\colon L^u Y\rightarrow \T[-1] Y$ is an isomorphism. Then
\[\det\left(\frac{\at_{\bT_Y}}{1-\exp(-\at_{\bT_Y})}\right)\in\cO(\T[-1] Y)\cong \bigoplus_{p=0}^{\dim Y} \rH^p(Y, \Omega^p_Y)\]
is the Todd class of $Y$. Thus, \cref{thm:Toddclass} shows that the torsion volume forms $\vol_{LY}$ on $LY$ and $\vol_{\T[-1] Y}$ on $\T[-1] Y$ differ by the Todd class. We refer to \cite[Corollary 4.4.3]{KondyrevPrikhodko} for a related statement.
\end{remark}

\printbibliography

\end{document}